\documentclass[reqno]{amsart}
\usepackage{amssymb}
\usepackage{amsfonts}
\usepackage{mathrsfs}
\usepackage{bbm}
\usepackage{xcolor}
\theoremstyle{plain}
    \newtheorem{theorem}{Theorem}[section]
    \newtheorem*{theorem*}{Theorem}
    \newtheorem{lemma}[theorem]{Lemma}
    \newtheorem{proposition}[theorem]{Proposition}
    \newtheorem{corollary}[theorem]{Corollary}

\theoremstyle{definition}
    \newtheorem{definition}{Definition}[section]
    \newtheorem{remark}{Remark}[section]
    
\numberwithin{equation}{section}

\renewcommand{\r}{\right}
\newcommand{\m}{\mbox{\textnormal{mid}}}
\begin{document}
%%%%%%%%%%%%%%%%%%%%%%%%%%%%%%%%%%%%%%%%%%%%%%%%%%%%%%%%

\title[]{Long-time asymptotics of (1,3)-sign solitary waves for the damped nonlinear Klein-Gordon equation}
\author[K. Ishizuka]{Kenjiro Ishizuka}
\address[K.Ishizuka]{Azabu Junior and Senior High School, Tokyo, Japan}
\email{k-ishizuka@azabu-jh.net}
\date{\today}
%\keywords{damped wave equation, dissipation, Strichartz estimates, energy critical}
\date{}
%\subjclass[2020]{35Q55,35B44 etc.}
\maketitle

\begin{abstract}
We consider the damped nonlinear Klein-Gordon equation
\begin{align*}
\partial_{t}^2u-\Delta u+2\alpha \partial_{t}u+u-|u|^{p-1}u=0, \ & (t,x) \in \mathbb{R} \times \mathbb{R}^d,
\end{align*}
where $\alpha>0$, $2\leq d\leq 5$ and energy sub-critical exponents $p>2$. In this paper, we prove that any solution that is asymptotic to a superposition of four solitons with exactly one soliton of opposite sign evolves so that the three like-signed solitons spread out in an equilateral-triangle configuration centered at the oppositely signed soliton.
\end{abstract}

\tableofcontents
%%%%%%%%%%%%%%%%%%%%%%%%%%%%%%%%%%%%%%%%%%

\section{Introduction}
\subsection{Setting of the problem}
We consider the following damped nonlinear Klein-Gordon equation
\begin{align}
\label{DNKG}
\tag{DNKG}
	\left\{
	\begin{aligned}
		&\partial_{t}^2u-\Delta u+2\alpha \partial_{t}u+u-f(u)=0, & (t,x) \in \mathbb{R} \times \mathbb{R}^d,
		\\
		&\left(u(0,x),{\partial}_tu(0,x)\right)=\left(u_{0}(x), u_1(x)\right)\in \mathcal{H},
	\end{aligned}
	\r.
\end{align}
where $2\leq d\leq 5$,  $\mathcal{H}=H^1(\mathbb{R}^d)\times L^2(\mathbb{R}^d)$, $f(u)=|u|^{p-1}u$, with a power $p$ in the energy sub-critical range, namely
\begin{align*}
2<p<\infty \ \mbox{for}\ d=2\ \mbox{and}\ 2<p<\frac{d+2}{d-2}\ \mbox{for}\ d=3,4,5.
\end{align*}
It follows from \cite{BRS} that the Cauchy problem for \eqref{DNKG} is locally well-posed in the energy space $\mathcal{H}$. Moreover, defining the energy of a solution $\vec{u}=(u, {\partial}_tu)$ by
\begin{align*}
E(\vec{u}(t))=\frac{1}{2}\|\vec{u}(t)\|_{\mathcal{H}}^2-\frac{1}{p+1}\| u(t)\|_{L^{p+1}}^{p+1},
\end{align*}
we have
\begin{align}\label{energydecay}
E(\vec{u}(t_2))-E(\vec{u}(t_1))=-2\alpha \int_{t_1}^{t_2} \| {\partial}_tu(t)\|_{L^2}^2dt.
\end{align}
In addition, the equation \eqref{DNKG} enjoys several invariances:
\begin{align}\label{invariance}
\begin{aligned}
u(t,x)&\mapsto -u(t,x),
\\
u(t,x)&\mapsto u(t,x+y),
\\
u(t,x)&\mapsto u(t+s,x),
\end{aligned}
\end{align}
where $y\in \mathbb{R}^d$ and $s\in \mathbb{R}$.
%More recently, Burq, Raugel, and Schlag \cite{BRS} proved the soliton resolution for \eqref{DNKG} in the full limit $t\to\infty$ for all radial solutions. 

\subsection{Long-time asymptotics of multi-soliton solutions}
 In this paper, we are interested in the dynamics of multi-soliton solutions related to the ground state $Q$, which is the unique positive, radial $H^1(\mathbb{R}^d)$ solution of 
\begin{align}\label{Qode}
-\Delta Q+Q-f(Q)=0.
\end{align}
(See \cite{K2}). The ground state generates the stationary solution $(Q,0)$. By \eqref{invariance}, the function $(-Q,0)$ as well as any translate $(Q(\cdot+y),0)$ is also a solution of \eqref{DNKG}. There have been intensive studies on global behavior of general (large) solutions for nonlinear dispersive equations, where the guiding principle is the \textit{soliton resolution conjecture}, which claims that generic global solutions are asymptotic to superposition of solitons for large time. For \eqref{DNKG}, Feireisl \cite{F} proved the soliton resolution along a time sequence: for any global bounded solution $\vec{u}$ of \eqref{DNKG}, there exist a time sequence $t_n\to \infty$ and a set of stationary solutions $\{ \varphi_k\}_{k=1}^K$ ($K\geq 0$), and sequences $\{ c_{k,n}\}_{k=1,2,\ldots, K,n\in \mathbb{N}}$ such that
\begin{align*}
\lim_{n\to\infty} \left( \|u(t_n)-\sum_{k=1}^K\varphi_k(\cdot -c_{k,n})\|_{H^1}+\|{\partial}_tu(t_n)\|_{L^2}\right)=0,
\\
\lim_{n\to\infty} \left(\min_{j\neq k} |c_{j,n}-c_{k,n}|\right)=\infty.
\end{align*}
In particular, in the $d=1$ case, since the stationary solutions of \eqref{DNKG} consist only of translates of $\pm Q$, C\^{o}te, Martel, and Yuan \cite{CMY} proved that if $\vec{u}$ is a global solution of \eqref{DNKG}, there exist a sign $\sigma=\pm 1$, $K\geq 0$, and functions $z_k:[0,\infty)\to \mathbb{R}$ such that 
\begin{align*}
\lim_{t\to\infty} \left(\| u(t)-\sigma \sum_{k=1}^K (-1)^kQ(\cdot-z_k(t))\|_{H^1}+\|{\partial}_tu(t)\|_{L^2}\right)&=0,
\\
\lim_{t\to\infty} \left(z_{k+1}(t)-z_k(t)\right)&=\infty.
\end{align*}
Prior to the present result, C\^{o}te, Martel, Yuan, and Zhao \cite{CMYZ} proved in general space dimension that any two-soliton superposition built from $Q$ must consist of solitons of opposite signs, and that the solitons move so that their centers evolve along a fixed straight line. In particular, this result makes the interaction forces acting on the solitons explicit, and suggests that like-signed solitons attract whereas oppositely signed solitons repel. Using this force law, \cite{CMY} showed that in one space dimension, in a multi-soliton superposition the signs of neighboring solitons must be opposite.

On the other hand, when the space dimension is at least two, two types of difficulties arise. The first is that, in dimensions $d\geq 2$, there exist many different kinds of stationary solutions (here “different” means inequivalent even after allowing for spatial translations and sign changes). The second is that soliton configurations depend on relative directions as well as distances. When considering a superposition of copies of $Q$,  each soliton feels an attractive or repulsive interaction from other nearby solitons. In \cite{CMY}, thanks to the restriction to one space dimension, the centers of the solitons in such a superposition lie on a single line, so it suffices to examine the interactions with adjacent solitons. By contrast, in dimensions $d\geq 2$, three or more solitons can be simultaneously close, and one must analyze not only the distances but also the relative directions. Despite these difficulties, there has been notable recent progress on multi-soliton solutions that can be described as superpositions of copies of  $Q$. As mentioned above, C\^{o}te, Martel, Yuan, and Zhao \cite{CMYZ} showed that, in general space dimension, a two-soliton superposition experiences essentially no “rotational” (i.e., tangential) interaction, and that the inter-soliton force is attractive for like signs and repulsive for opposite signs. Subsequently, the author \cite{I2} showed that three-soliton superpositions are necessarily collinear.  Indeed, it is known that any solution asymptotic to a superposition of three solitons must contain exactly one soliton whose sign is opposite to that of the other two (see \cite{CD}). In this case, unless the three solitons are collinear, the oppositely signed soliton is pushed outward by repulsion from the other two; consequently, the attraction between the like-signed solitons dominates the repulsion between oppositely signed solitons, and the configuration cannot sustain a three-soliton state. We also note that, as a constructive approach, C\^{o}te and Du \cite{CD} imposed symmetry conditions and constructed solutions whose soliton centers form highly symmetric configurations, such as an equilateral triangle or a regular tetrahedron.

\subsection{Main result}
First, we define solutions that behave as the superposition of ground states.
\begin{definition}\label{defKsoli}
A solution $\vec{u}$ of \eqref{DNKG} is called a \textit{$K$-solitary waves} for $K\in\mathbb{N}$ if there exist $\sigma \in \{-1,1\}^K$ and $z:[0,\infty)\to (\mathbb{R}^d)^K$ such that 
\begin{align}
\label{Ksoli}
\begin{aligned}
\lim_{t\to \infty} \left(\|u(t)-\sum_{k=1}^K\sigma_k Q(\cdot-z_k(t))\|_{H^1}+\|{\partial}_tu(t)\|_{L^2}\right)&=0,
\\
\lim_{t\to\infty}\left( \min_{j\neq k}|z_j(t)-z_k(t)|\right)&=\infty.
\end{aligned}
\end{align}
\end{definition}

\begin{remark}
In space dimension $d=1$, thanks to the damping term $\alpha$, Definition \ref{defKsoli} remains equivalent to the original definition even if it is replaced with convergence in the sense of time sequences $t_n\to\infty$. On the other hand, in space dimension $d\geq 2$, even if a solution converges along a sequence of times, it is not clear whether it remains uniformly bounded in $\mathcal{H}$ for all $t$. Hence, the equivalence of the two notions is not known. In this paper, we are interested in the long-time behavior of the positions of solitons. Therefore, we adopt Definition \ref{defKsoli}.
\end{remark}
\begin{remark}
Hereafter, unless stated otherwise, when we discuss $K$-solitary waves we implicitly assume that it is a solution of \eqref{DNKG}.
\end{remark}

In particular, in this paper we focus on 4-solitary waves in which exactly one solitary wave has the opposite sign from the other three, that is, superpositions consisting of three solitons of one sign and one soliton of the other sign. Accordingly, in order to distinguish not only the number of solitary waves but also whether each component is a copy of $Q$ or of $-Q$, we introduce the notion of $(m,n)$-sign solitary waves as follows.
\begin{definition}\label{defm,nsoli}
A solution $\vec{u}$ of \eqref{DNKG} is called a \textit{$(m,n)$-sign solitary waves} for $m,n\in\mathbb{N}\cup\{0\}$ if there exist $z_+:[0,\infty)\to (\mathbb{R}^d)^m$ and $z_-:[0,\infty)\to(\mathbb{R}^d)^n$ such that 
\begin{align}
\label{mnsoli}
\begin{aligned}
\lim_{t\to \infty} \left(\|u(t)-\sum_{k=1}^mQ(\cdot-z_{+,k}(t))+\sum_{k=1}^n Q(\cdot-z_{-,k}(t))\|_{H^1}+\|{\partial}_tu(t)\|_{L^2}\right)&=0,
\\
\lim_{t\to\infty}\left( \min_{(\circ,j)\neq (\bullet,k)}|z_{\circ,j}(t)-z_{\bullet,k}(t)|\right)&=\infty.
\end{aligned}
\end{align}
\end{definition}

\begin{remark}
Since \eqref{DNKG} is invariant under the sign change $u\mapsto -u$, any property established for $(m,n)$-sign solitary waves solutions immediately yields the corresponding property for $(n,m)$-sign solitary waves solutions. It was also proved in \cite{CD} that no $(m,0)$-sign or $(0,m)$-sign solitary waves solutions exist.
\end{remark}
%\begin{remark}文法怪しい
%Since \eqref{DNKG} is invariant under the sign change $u\mapsto -u$, any property established for $(m,n)$-sign solitary waves  solutions immediately yields the corresponding property for $(n,m)$-sign solitary waves solutions. It was also proved in \cite{CD} that neither a $(m,0)$-sign solitary waves nor a $(0,m)$-sign solitary waves exist.
%\end{remark}
\begin{remark}
Hereafter, unless stated otherwise, when we discuss $(m,n)$-sign solitary waves we implicitly assume that it is a solution of \eqref{DNKG}.
\end{remark}

%\begin{remark}
%For \eqref{DNKG}, thanks to the damping term $\alpha$, Definition \ref{defKsoli} remains equivalent to the original definition even if it is replaced with convergence in the sense of time sequences $t_n\to\infty$. Therefore, to facilitate the subsequent discussion, we adopt Definition \ref{defKsoli}. The equivalence of the two definitions can be demonstrated using classical compactness arguments. For further details, we refer to \cite[Section 3]{I}.
%\end{remark}
In this paper, we are interested in $(1,3)$-sign solitary waves solutions. By imposing suitable symmetries, C\^{o}te and Du \cite{CD} constructed a $(1,3)$-sign solitary waves for which the center of the $+Q$-soliton converges to a point, while the centers of the remaining three solitons spread out so as to form an expanding equilateral triangle. In the present work, we prove that this geometric behavior does not rely on any symmetry assumption: if a solution is a $(1,3)$-sign solitary waves, then the center of the $+Q$-soliton converges, and the other three soliton centers diverge in an equilateral-triangle configuration with the $+Q$-soliton as the center of similarity.
Before stating our main theorem, we define spreading out in an equilateral-triangle configuration by introducing the following terminology.
\begin{definition}\label{seisankakukei}
Let $\omega_1,\omega_2,\omega_3\in S^{d-1}$. We say that $\omega_1,\omega_2,\omega_3$ form an \textit{equilateral triple} if
\begin{align*}
\omega_1+\omega_2+\omega_3=0.
\end{align*}
\end{definition}

\begin{remark}
If $\omega_1,\omega_2,\omega_3\in S^{d-1}$ form an equilateral triple, then they lie in a common
two-dimensional subspace of $\mathbb{R}^d$ (indeed, $\omega_3=-(\omega_1+\omega_2)$). 
Moreover, $\omega_1,\omega_2,\omega_3$ satisfy
\begin{align*}
|\omega_1-\omega_2|=|\omega_2-\omega_3|=|\omega_3-\omega_1|.
\end{align*}
%In fact, $\omega_i\cdot \omega_j=-\frac12$ for $i\neq j$, and hence
%\begin{equation*}
%|\omega_i-\omega_j|^2=2-2\,\omega_i\cdot \omega_j=3
%\quad\text{for } i\neq j.
%\end{equation*}
\end{remark}

Then, our main result is as follows:
\begin{theorem}\label{maintheorem}
If a solution $\vec{u}$ of \eqref{DNKG} is a $(1,3)$-sign solitary waves, there exist $z_{\infty}\in \mathbb{R}^d$, an equilateral triple $\omega_1,\omega_2,\omega_3\in S^{d-1}$, and functions $z_k:[0,\infty)\to \mathbb{R}^d$ for $k=1,2,3$ such that for all $t>2$,
\begin{align}\label{3soli}
\left\| u(t)-\left(Q(\cdot-z_{\infty})-\sum_{k=1}^3 Q(\cdot-z_k(t))\right)\right\|_{H^1}+\|{\partial}_tu(t)\|_{L^2}\lesssim t^{-1},
\end{align}
and for $k=1,2,3$,
\begin{align}\label{zestheorem}
z_k(t)=z_{\infty}+\left( \log{t}-\frac{d-1}{2}\log{(\log{t})}+c_0\right)\omega_k+O\left(\frac{\log{(\log{t})}}{\log{t}}\right),
\end{align}
where $c_0$ is a constant depending only on $d,\ \alpha$, and $p$.
\end{theorem}
\begin{remark}
If $\vec{u}$ is a $(3,1)$-sign solitary waves, then $-\vec{u}$ is a $(1,3)$-sign solitary waves by \eqref{invariance}. Therefore, by Theorem \ref{maintheorem}, the long-time dynamics of the soliton centers for a $(3,1)$-sign solitary waves are the same as those for a $(1,3)$-sign solitary waves.
\end{remark}

%C\^{o}te and Du obtained a $(1,3)$-sign solitary waves by imposing suitable group symmetries, producing an expanding equilateral triangle of three soliton centers around a fourth one. In contrast, our main result proves that this geometry is intrinsic and does not rely on symmetry: every (1,3)-sign solitary waves necessarily exhibits the same equilateral-triangle asymptotic configuration, with the distinguished soliton center converging and the other three centers diverging with it as the center of similarity.

\begin{remark}
When $d=1$, $(1,3)$-sign solitary waves solutions do not exist. This is because a $(1,3)$-sign solitary waves corresponds to a solution that spreads in three distinct directions; however, in the case $d=1$, there are only two possible directions (positive direction and negative direction). 
\end{remark}

\begin{remark}
To completely characterize the long-time dynamics of four solitary waves solutions, it is also necessary to understand the long-time behavior of $(2,2)$-sign solitary waves solutions.
In contrast to the $(1,3)$-sign case, the $(2,2)$-sign configuration has no distinguished (reference) soliton---there is no unique soliton whose sign differs from all the others---and this makes the asymptotic analysis more delicate.
%We conjecture that a $(2,2)$-sign solitary wave exhibits one of the following geometric behaviors: either the four soliton centers are collinear and ordered so that adjacent solitons have opposite signs, or the four centers form the vertices of a rhombus, with solitons of the same sign lying at opposite vertices (equivalently, each diagonal connects solitons of the same sign).
\end{remark}

\subsection{Difficulties and Ideas for the Main Result}
Determining the long-time behavior of a four-soliton configuration is more difficult than in the case of a superposition of three or fewer solitons. Indeed, since each soliton is acted on by forces generated by the other three solitons, one cannot rule out a priori the possibility that the soliton centers exhibit complicated long-time motions. For instance, in the three-soliton case, \cite{I2} shows that it suffices to first prove that the two distances from the oppositely signed soliton to the other two become comparable. In that regime, if the three centers are not collinear, then the oppositely signed soliton is pushed outward by the repulsive forces exerted by the other two, and the apex angle of the corresponding isosceles triangle decreases. By contrast, in the $(1,3)$-sign solitary waves setting considered here, the forces act from three different directions, so that even the force exerted on the oppositely signed soliton by two of the solitons may be canceled by the force coming from the remaining soliton.
Because the inter-soliton forces are so delicate, it is conceivable a priori that, for $(1,3)$-sign solitary waves, the three like-signed solitons might spread out in a non-equilateral configuration---for instance, forming a triangle with angles $50^\circ$, $60^\circ$, and $70^\circ$---rather than an equilateral one. In this paper, to study the motion of the soliton centers, we exploit the fact that the mutual distances between solitons evolve in a way that is coupled to the geometry of the configuration, in particular to the angles formed by the centers.
More precisely, taking the oppositely signed soliton as a reference, we consider the three directions along which the other solitons separate, and introduce the associated unit direction vectors; the angles between these directions are encoded by the inner products of these unit vectors.
We then derive an effective system of ordinary differential equations for these angular variables and reduce it to a closed ODE system that can be written solely in terms of the three pairwise inner products. The resulting dynamical system can be analyzed by an elementary  phase-portrait analysis.
A key idea behind this reduction is a separation of time scales: the angular variables evolve much more slowly than the radial distances, so that, to leading order, one may treat the distances as evolving as if the angles were frozen.

\subsection{Previous results}

First, we introduce results related to the soliton resolution conjecture. The soliton resolution conjecture has been studied for the energy-critical wave equation. We refer to \cite{DKM1,JL4}, in which the conjecture was proven for all radial solutions in space dimension $d\geq 3$. Historically, we refer to \cite{CDKM, DJKM, DKM1, DKM2}, where the conjecture has been proved in space dimension $d=3,5,7,\ldots$ or $d=4,6$ using the method of channels. In recent years, Jendrej and Lawrie \cite{JL4} proved the conjecture in space dimension $d\geq 4$ using a novel argument based on the analysis of collision intervals. Similar results regarding the conjecture have been obtained for cases with a potential \cite{JLX,LMZ}, with a damping term \cite{GZ}, as well as for wave maps \cite{DKMM,JL1}, for the heat equation \cite{A1}, all under some rotational symmetry. In recent years, in order to gain a finer understanding of the long-time dynamics, there has been growing interest in the analysis of intermediate (or transient) states that arise before the final soliton-resolution regime as $t\to+\infty$. For energy-critical wave equations, see \cite{S1, S2}, and for the heat equation, see \cite{A2}.

On the other hand, the conjecture remains widely open for the nonlinear Klein-Gordon equation, namely the $\alpha=0$ case. Nakanishi and Schlag \cite{NS} proved the conjecture as long as the energy of the solution is not much larger than the ground state energy in space dimension $d=3$. 
For the damped nonlinear Klein-Gordon equation case, the energy decay \eqref{energydecay} makes the analysis simpler than the undamped case. Historically, Keller \cite{K1} constructed stable and unstable manifolds in a neighborhood of any stationary solution of \eqref{DNKG}, and Feireisl \cite{F} proved the conjecture along a time sequence. Burq, Raugel and Schlag \cite{BRS} proved the conjecture in the full limit $t\to\infty$ for all radial solutions. In particular, C\^{o}te, Martel and Yuan \cite{CMY} proved the conjecture for $d=1$ without size restriction and radial assumption for a non-integrable equation. In a result without any size or symmetry restriction, Kim and Kwon \cite{KK} proved the conjecture for Calogero-Moser derivative nonlinear Sch\"{o}dinger equation. On the other hand, in the case of general dimensions $d\geq 2$, it remains unclear how solitary waves are arranged and how they move. One reason is that, in general dimensions, it is not known whether solutions remain bounded in
the energy space $\mathcal H$; another is that \eqref{DNKG} admits a wide variety of stationary solutions.

 %On the other hand, in the case of general dimensions $d\geq 2$, it remains unclear how solitary waves are arranged and how they move. This is because the ODE system for the centers of the solitary waves becomes more complex in general dimensions. 
 
 In general space dimension, C\^{o}te, Martel, Yuan, and Zhao \cite{CMYZ} proved that 2-solitary waves with same sign do not exist. Furthermore, they constructed a Lipschitz manifold in the energy space with codimension 2 of those solutions asymptotics to 2-solitary waves with the opposite signs. Subsequently, the author and Nakanishi \cite{IN} provided a complete classification of solutions into 5 types of global behavior for all initial data in a small neighborhood of each superposition of two ground states with the opposite signs. Additionally, results concerning an excited state of \eqref{Qode} (stationary solution) can be found in \cite{CY}. Recently, C\^{o}te and Du \cite{CD} constructed multi-solitary waves that asymptotically resemble various objects by imposing group-symmetric assumptions. This result has facilitated the study of the global behavior of multi-solitary waves in general dimensions.

A crucial aspect of analyzing multi-solitary waves is understanding the movement of their centers caused by interactions between solitary waves. In particular, a 2-solitary waves with the same sign does not exist because the interaction between them generates an attractive force. Moreover, the author \cite{I2} proved that, in the case of $3$-solitary waves, the three solitons become collinear, with only the middle soliton having the opposite sign. The long-time dynamics force a collinear configuration in which exactly one soliton has the opposite sign and this oppositely signed soliton lies between the other two. In \cite{I}, the author also constructed 2-solitary waves with the same signs for the damped nonlinear Klein-Gordon equation with a repulsive delta potential in $d=1$. This is because the repulsive force of the potential dominates the attractive interaction between the solitons.

 In addition, several studies have been conducted on multi-solitary waves for the nonlinear dispersive equations. For the nonlinear Klein-Gordon equation, C\^{o}te and Mu\~{n}oz \cite{CM} constructed multi-solitary waves based on the ground state. Furthermore, C\^{o}te and Martel \cite{CM1} extended the work of \cite{CM} by constructing multi-solitary waves based on the stationary solution of the nonlinear Klein-Gordon equation. Notably, multi-solitary waves usually behave asymptotically as the solitary waves moving at distinct constant speeds. More recently, Aryan \cite{A} has constructed a 2-solitary wave where the distance between each solitary wave is $2(1+O(1))\log{t}$ as $t\to \infty$. There are several results regarding multi-solitary waves with logarithmic distance; see \textit{e.g.} \cite{GI,MN,N1,N2}. Notably, for the damped nonlinear Klein-Gordon equation, multi-solitary waves of \eqref{DNKG} cannot move at a constant speed by the damping $\alpha$. A distinctive feature of the present work is that it provides a complete rigidity description for the long-time dynamics of multi-soliton solutions. Most of the aforementioned results on multi-solitons have focused primarily on construction and existence. Among them, the results in \cite{CD} and \cite{MR} concern regular polygonal configurations for dispersive equations. In \cite{CD}, the existence of such a regular polygon configuration is obtained by imposing symmetry assumptions on the solutions. In \cite{MR}, the authors constructed a multi-bubble (blow-up bubble) solution which is asymptotic to a superposition of concentrating bubbles centered at the vertices of a regular polygon. In both works, symmetry is imposed in order to simplify the interactions between solitons, since without symmetry the interactions become significantly more involved. To the best of my knowledge, the present paper is the first to show, in the context of dispersive equations, that the centers of multi-solitons are forced to form an equilateral triangle.

%To the best of my knowledge, the present paper is the first to show that the centers of multi-solitons are forced into an equilateral-triangle configuration for dispersive equations.
%本研究の特筆すべき点は,\ 多重ソリトンの長時間挙動において,\ その剛性を完全に解明している点である.\ 多重ソリトンに関する上述の結果は基本的には構成と存在に関する結果が中心であった.\ その中でも,\ \cite{CD}と\cite{MR}の結果は分散型方程式の正多角形に関する結果になっている.\ \cite{CD}においては解に対称性を課すことで正多角形の配置の存在を示している.\ \cite{MR}は,\ 多重バブル解として,\ 正多角形の頂点を中心とした縮むソリトンの和に漸近する解を構成した.\ この2つの結果はどちらもソリトン間の相互作用を単純にするため,\ 対称性を課していた.\ これは対称性を除いた場合は相互作用が複雑に作用するからである.\ そのため,\ 私の知る限りでは多重ソリトンの中心に関して,\ 正三角形の配置に強制される結果は本論文が初めてである.
%\begin{remark}
%By a `' bubble solution'', we mean a stationary solution (standing wave) obtained by applying the scaling symmetry to a fixed stationary profile. In this paper, we refer to such objects as `'solitons''. For instance, in the soliton resolution for the energy-critical wave equation, the solitons that arise are precisely bubble solutions. Moreover, by a `'multi-bubble solution'' we mean a solution that is asymptotic to a sum of bubbles. We emphasize that the term multi-bubble solution is not restricted to the case where all bubbles correspond to the same scaling, as in \cite{MR}. In many contexts, it also encompasses superpositions of bubble solutions associated with different scales. Indeed, \cite{J1,J2,JK,JL1,JL2} study two-bubble solutions whose scales are different.
%\end{remark}
\begin{remark}
By a ``bubble solution'', we mean a stationary solution (or standing wave) obtained from a fixed stationary profile by applying the scaling symmetry. In this paper, we refer to such solutions as ``solitons''. For instance, in the context of soliton resolution for the energy-critical wave equation, the solitons that arise are precisely bubble solutions. Moreover, by a ``multi-bubble solution'', we mean a solution that is asymptotic to a superposition of bubbles. We emphasize that the term ``multi-bubble solution'' is not restricted to the case where all bubbles have the same scale, as in \cite{MR}. In many contexts, it also includes superpositions of bubble solutions associated with different scales. Indeed, \cite{J1,J2,JK,JL1,JL2} study two-bubble solutions with different scales.
\end{remark}

\subsection{Notation}
Let $\{e_1,\ldots, e_d\}$ denote the canonical basis of $\mathbb{R}^d$, and let ${\partial}_k$ denote the partial derivative with respect to $x_k$. The Euclidean inner product is denoted for any pair of vectors $x,y\in \mathbb{R}^d$ by
\begin{align}\label{Euclideanproduct}
x\cdot y=\sum_{k=1}^dx_ky_k.
\end{align}
 In this paper, $X\lesssim Y$ means that $X\leq CY$ for some constant $C>0$. $X\sim Y$ means that $X\lesssim Y$ and $Y\lesssim X$. $X\ll Y$ means that there exists a sufficiently small constant $c>0$ such that $X\leq cY$. The meaning of sufficiency should be clear from the context, or clarified later in more explicit form. Moreover, in this paper we often deal with the second smallest value among three quantities. We therefore define
\begin{align}\label{middef}
\m{(a,b,c)}=a+b+c-\max{(a,b,c)}-\min{(a,b,c)}.
\end{align}

\section{Preliminaries}
In this section, we collect several preliminary results on the long-time dynamics of $K$-solitary waves solutions. Most of the material is taken from \cite[Section 2 and 3]{I2} and is included here for the reader's
convenience.

\subsection{Basic properties of the ground state}

Since $Q$ is radial, there exists $q:[0,\infty)\to\mathbb{R}$ such that
\begin{align*}
Q(x)=q(|x|).
\end{align*}
Furthermore, $q$ satisfies
\begin{align}\label{qode2}
q^{\prime \prime}+\frac{d-1}{r}q^{\prime}-q+q^p=0.
\end{align}
In addition, there exists a constant $c_q$ depending only on $d$ and $p$ such that
\begin{align}\label{qes}
|q(r)-c_qr^{-\frac{d-1}{2}}e^{-r}|+|q^{\prime}(r)+c_qr^{-\frac{d-1}{2}}e^{-r}|
\lesssim r^{-\frac{d+1}{2}}e^{-r}.
\end{align}
%We also introduce the constant
%\begin{align}\label{thetastar}
%\theta_{\star}=\min\{p-1,2\}.
%\end{align}
Next, in order to make the leading interaction term more explicit, we introduce the vector field
$H:\mathbb{R}^d\setminus\{0\}\to\mathbb{R}^d$ defined by
\begin{align}\label{Hdef}
H(z)=-\int_{\mathbb{R}^d}(\nabla Q^p(x))\,Q(x-z)\,dx.
\end{align}
Moreover, we define $g:(0,\infty)\to\mathbb{R}$ as
\begin{align*}
H(z)=\frac{z}{|z|}\,g(|z|).
\end{align*}
We note that this definition is well-defined. Furthermore, we have 
\begin{align*}
g(r+h)-g(r)
=-\int_{\mathbb{R}^d} {\partial}_1(Q^p)(x)\Bigl( Q\bigl(x-(r+h)e_1\bigr)-Q(x-re_1)\Bigr)\,dx.
\end{align*}
Therefore, we have 
\begin{align*}
g^{\prime}(r)
=\int_{\mathbb{R}^d} {\partial}_1(Q^p)(x)\,{\partial}_1Q(x-re_1)\,dx
=-\int_{\mathbb{R}^d}Q^p(x)\,{\partial}_1^2Q(x-re_1)\,dx.
\end{align*}
In addition, by direct computation,
\begin{align*}
{\partial}_1^2Q(x-re_1)
=\frac{\sum_{k=2}^dx_k^2}{|x-re_1|^3}q^{\prime}(|x-re_1|)
+\frac{(x_1-r)^2}{|x-re_1|^2}q^{\prime \prime}(|x-re_1|).
\end{align*}
Using \eqref{qode2}, we obtain
\begin{align}\label{gnozenkin1}
|g(r)-c_gr^{-\frac{d-1}{2}}e^{-r}|+|g^{\prime}(r)+c_gr^{-\frac{d-1}{2}}e^{-r}|
\lesssim r^{-\frac{d+1}{2}}e^{-r},
\end{align}
where
\begin{align*}
c_g=\int_{\mathbb{R}^d} Q^p(x)e^{-x_1}\,dx.
\end{align*}
Moreover, by \eqref{gnozenkin1}, we have for $R\gg 1$ and $|a| \ll R$
\begin{align}\label{gnozenkin2}
 \left|g(R+a)-e^{-a}g(R)\right|\lesssim \frac{g(R)}{R},
 \end{align}
 and we have for $|a|\ll 1$ and $1\ll R$
 \begin{align}\label{gnozenkin3}
\left|g(R+a)-(1-a)g(R)\right|\lesssim g(R)\left(a^2+\frac{|a|}{R}\right).
\end{align}
We will discuss further properties of $H$ and $g$ in a later subsection.

\subsection{Basic properties of the dynamics of $K$-solitary waves}

We now summarize several results concerning the long-time dynamics of $K$-solitary waves solutions.
Let $\vec{u}$ be a $K$-solitary waves. Then there exist $\sigma\in\{-1,1\}^K$ and a function $z:[0,\infty)\to(\mathbb{R}^d)^K$ such that \eqref{Ksoli} holds.
For such $z$ and $\sigma$, we define
\begin{align*}
Q_k=\sigma_k Q(\cdot-z_k),\qquad
\vec{Q}_k=
\begin{pmatrix}
Q_k\\
0
\end{pmatrix},
\qquad
Q_{\sum}=\sum_{k=1}^KQ_k,\qquad
\vec{Q}_{\sum}=
\begin{pmatrix}
Q_{\sum}\\
0
\end{pmatrix}.
\end{align*}
Furthermore, for a solution $\vec{u}$ of \eqref{DNKG} and $z$ and $\sigma$, we define
\begin{align*}
\vec{\varepsilon}=
\begin{pmatrix}
\varepsilon\\
\eta
\end{pmatrix}
=\vec{u}-\vec{Q}_{\sum}.
\end{align*}
We also introduce the minimal separation
\begin{align*}
D(t)=\min_{i\neq j}|z_i(t)-z_j(t)|.
\end{align*}

In a multi-soliton decomposition, the choice of the center parameters $z_k(t)$ is not unique due to translation invariance.
Following \cite[Lemma 3.2]{I2}, we fix this freedom by imposing orthogonality conditions that eliminate the contribution of the
neutral directions associated with translations of the linearized operator around each soliton. This leads to a more
convenient system for the modulation parameters.

\begin{lemma}\label{modKsoli}
Let $\vec{u}$ be a $K$-solitary waves. Then there exist $\sigma\in\{-1,1\}^K$ and a $C^1$ function
$z:[0,\infty)\to (\mathbb{R}^d)^K$ such that
\begin{align}
\lim_{t\to\infty}\Bigl(\left\| u(t)-\sum_{k=1}^K \sigma_kQ(\cdot-z_k(t))\right\|_{H^1}
+\|{\partial}_tu(t)\|_{L^2}\Bigr)&=0,\label{modlim1yowa}
\\
\lim_{t\to\infty} D(t)&=\infty, \label{modlim2yowa}
\end{align}
and for $1\leq l\leq d$, $1\leq i\leq K$, and $t\gg 1$,
\begin{align}\label{modeqyowa}
\int_{\mathbb{R}^d} \Bigl\{ {\partial}_tu(t)+2\alpha\Bigl(u(t)-\sum_{k=1}^K \sigma_k Q(\cdot-z_k(t))\Bigr)\Bigr\}
\, {\partial}_lQ(\cdot-z_i(t))=0.
\end{align}
\end{lemma}
\begin{proof}
See \cite[Lemma 3.2]{I2}.
\end{proof}
%By Lemma \ref{modKsoli}, to simplify the subsequent discussion we assume that the parameters have been chosen so that
%\begin{align}\label{sa}
%\left\{
%\begin{aligned}
%&\lim_{t\to \infty}\|\vec{\varepsilon}(t)\|_{\mathcal{H}}=0,
%\\
%&\lim_{t\to \infty}D(t)=\infty,
%\\
%&\|\vec{\varepsilon}(t)\|_{\mathcal{H}}+\frac{1}{D(t)}<\delta\ \ \mbox{for}\ t\geq 0,
%\\
%&\int (\eta+2\alpha \varepsilon)\,{\partial}_lQ(\cdot-z_i)=0\ \mbox{for}\ t\geq 0,\ 1\leq l\leq d,\ 1\leq i\leq K,
%\end{aligned}
%\right.
%\end{align}
%for some $\delta>0$.
We then define the scalar function $\mathcal{F}:[1,\infty)\to [0,\infty)$ by
\begin{align}\label{Fdef}
\mathcal{F}(r)=\frac{g(r)}{2\alpha\|{\partial}_1Q\|_{L^2}^2}.
\end{align}
Then, by \eqref{gnozenkin2} and \eqref{gnozenkin3}, we have for $R\gg 1$ and $|a|\ll R$
\begin{align}\label{tukaeruF1}
\left|\mathcal{F}(R+a)-e^{-a}\mathcal{F}(R)\right|\lesssim \frac{\mathcal{F}(R)}{R},
\end{align}
and we have for $R\gg1$ and $|a|\ll 1$
\begin{align}\label{tukaeruF2}
\left|\mathcal{F}(R+a)-(1-a)\mathcal{F}(R)\right|\lesssim \mathcal{F}(R)\left(a^2+\frac{|a|}{R}\right).
\end{align}
Then, by \cite[Section 3]{I2}  $\mathcal{F}$ describes the leading-order interaction force between well-separated solitons.
\begin{lemma}\label{centerdynamics}
Let $\vec{u}$ be a $K$-solitary waves. Then there  exist $\sigma\in\{-1,1\}^K$ and a $C^1$ function
$z:[0,\infty)\to (\mathbb{R}^d)^K$ such that for $t\gg 1$
\begin{align}\label{zmodvol2}
\dot{z}_k
=-\sum_{\substack{1\leq i\leq K\\ i\neq k}}\sigma_i\sigma_k\mathcal{F}(|z_k-z_i|)
\frac{z_k-z_i}{|z_k-z_i|}
+o(\mathcal{F}(D)).
\end{align}
\end{lemma}
\begin{proof}
By Lemma \ref{modKsoli}, there  exist $\sigma\in\{-1,1\}^K$ and a $C^1$ function
$z:[0,\infty)\to (\mathbb{R}^d)^K$ such that $\sigma$ and $z$ satisfy \eqref{modlim1yowa}, \eqref{modlim2yowa}, and \eqref{modeqyowa}.
Furthermore, by \cite[Lemma 3.3]{I2}, for $1<\theta<\min{(p-1,2)}$, we have
\begin{align}
\left|\dot{z}_k+\sum_{\substack{1\leq i\leq K\\ i\neq k}}\sigma_i\sigma_k\mathcal{F}(|z_k-z_i|)
\frac{z_k-z_i}{|z_k-z_i|}\right|
&\lesssim e^{-\theta D}+\|\vec{\varepsilon}\|_{\mathcal{H}}^2.\label{zes}
\end{align}
In addition, by \cite[Lemma 3.8]{I2} and
\begin{align*}
e^{-\theta D}\lesssim  \frac{\mathcal{F}(D)}{D},
\end{align*}
we obtain \eqref{zmodvol2}.
\end{proof}
Moreover, by exploiting the energy dissipation identity \eqref{energydecay} together with a Taylor expansion of the
energy around $\vec{Q}_{\sum}$, one obtains quantitative estimates both on the interaction terms and on the decay of the
remainder $\vec{\varepsilon}$. We define $V$ as 
\begin{align}\label{Vdef}
V=-\sum_{1\leq i<j\leq K}\sigma_i\sigma_j\mathcal{F}(|z_i-z_j|).
\end{align}
We next record a useful estimate of $V$.
\begin{lemma}\label{limeplem}
Let $\vec{u}$ be a $K$-solitary waves. Then there  exist $\sigma\in\{-1,1\}^K$ and a $C^1$ function
$z:[0,\infty)\to (\mathbb{R}^d)^K$ such that
\begin{align}\label{Ves1}
\liminf_{t\to\infty} \frac{V(t)}{\mathcal{F}(D(t))}\geq 0.
\end{align}
Furthermore, we have for $t\gg 1$
\begin{align}
D(t)-\log{t}+\frac{d-1}{2}\log{(\log{t})}\lesssim1, \label{Dupper}
\\
\mathcal{F}(D)\gtrsim \frac{1}{t}. \label{F(D)lower}
\end{align}
\end{lemma}
\begin{proof}
\eqref{Ves1} represents \cite[Lemma 3.8]{I2}. \eqref{Dupper} and \eqref{F(D)lower} represent \cite[Lemma 3.9]{I2}. Thus, see \cite[Lemma 3.8, Lemma 3.9]{I2}.
\end{proof}
If, in addition, the interaction functional $V$ dominates $\mathcal{F}(D)$ in the sense that
\begin{align}\label{Vrep}
\liminf_{t\to \infty}\frac{V(t)}{\mathcal{F}(D(t))}> 0,
\end{align}
then one can further improve the control on the remainder and the modulation parameters by \cite[Lemma 3.11]{I2}.
\begin{lemma}\label{epzes}
Let $\vec{u}$ be a $K$-solitary waves and satisfy \eqref{modlim1yowa}, \eqref{modlim2yowa}, and \eqref{modeqyowa} for some $\sigma$ and a $C^1$ function $z$. Furthermore, we assume \eqref{Vrep}. Then there exists $C>0$ depending on $\vec{u}$ such that
\begin{align}\label{newepes}
\|\vec{\varepsilon}(t)\|_{\mathcal{H}}\leq C \mathcal{F}(D(t)).
\end{align}
Furthermore, for any $1<\theta<\min{(p-1,2)}$, there exists $\tilde{C}>0$ such that for any $1\leq k\leq K$,
\begin{align}\label{newzes}
\left|\dot{z}_k+\sum_{\substack{1\leq i\leq K\\ i\neq k}}\sigma_i\sigma_k\mathcal{F}(|z_k-z_i|)
\frac{z_k-z_i}{|z_k-z_i|}\right|
\leq \tilde{C} e^{-\theta D}.
\end{align}
\end{lemma}

\section{Soliton interactions for $(1,3)$-sign solitary waves solutions}
In this section, we consider $(1,3)$-sign solitary waves solutions. First, to simplify the notation, we rewrite Lemma \ref{modKsoli} in the following form.
\begin{lemma}\label{Ksolikakikae}
Let $\vec{u}$ be a $(1,3)$-sign solitary waves. Then there exist $C^1$ functions $z_0,z_1,z_2,z_3:[0,\infty)\to \mathbb{R}^d$ such that 
\begin{align}
\lim_{t\to\infty}(\left\| u(t)-Q(\cdot-z_0(t))+\sum_{k=1}^3 Q(\cdot-z_k(t))\right\|_{H^1}+\|{\partial}_tu(t)\|_{L^2})&=0,\label{modlim1}
\\
\lim_{t\to\infty} \min_{0\leq i\neq j\leq 3}|z_i(t)-z_j(t)|&=\infty. \label{modlim2}
\end{align}
Furthermore, we have for $1\leq l\leq d$, $0\leq i\leq 3$, and $t\gg 1$,
\begin{align}\label{modeq}
\int_{\mathbb{R}^d} \{ {\partial}_tu(t)+2\alpha(u(t)-Q(\cdot-z_0(t))+\sum_{k=1}^3 Q(\cdot-z_k(t)))\} {\partial}_lQ(\cdot-z_i(t))=0.
\end{align}
\end{lemma}
\begin{remark}
In particular, by Lemma \ref{Ksolikakikae}, throughout the rest of this paper, whenever we assume that $\vec{u}$ is a $(1,3)$-sign solitary waves, we choose $C^1$ functions $z_0,z_1,z_2,z_3$ so that \eqref{modlim1}, \eqref{modlim2}, and \eqref{modeq} hold.
\end{remark}
Here we introduce several definitions to further simplify the notation. First, in what follows we proceed using the functions $z_k$ above. Furthermore, we define for $k=1,2,3$, 
\begin{align}\label{Zrhoudef1}
Z_k=z_k-z_0,\ \rho_k=|Z_k|,\ u_k=\frac{Z_k}{\rho_k}.
\end{align}
In addition, we define for $1\leq j,\ k\leq 3$ and $j\neq k$,
\begin{align}\label{Zrhoudef2}
Z_{jk}=z_k-z_j,\ \rho_{jk}=|Z_{jk}|,\ u_{jk}=\frac{Z_{jk}}{\rho_{jk}}.
\end{align}
In this setting, Lemma \ref{centerdynamics} yields the following dynamics for $z_0,z_1,z_2,z_3$.
\begin{lemma}\label{(1,3)centerdynamics}
Let $\vec{u}$ be a (1,3)-sign solitary waves. Then $z_0,z_1,z_2,z_3$ satisfy 
\begin{align}\label{(1,3)zdynamics}
\begin{aligned}
\dot{z}_0&=-\mathcal{F}(\rho_1)u_1-\mathcal{F}(\rho_2)u_2-\mathcal{F}(\rho_3)u_3+o\left(\mathcal{F}(D)\right),
\\
\dot{z}_1&=\mathcal{F}(\rho_1)u_1+\mathcal{F}(\rho_{12})u_{12}+\mathcal{F}(\rho_{13})u_{13}+o\left(\mathcal{F}(D)\right),
\\
\dot{z}_2&=\mathcal{F}(\rho_2)u_2+\mathcal{F}(\rho_{21})u_{21}+\mathcal{F}(\rho_{23})u_{23}+o\left(\mathcal{F}(D)\right),
\\
\dot{z}_3&=\mathcal{F}(\rho_3)u_3+\mathcal{F}(\rho_{31})u_{31}+\mathcal{F}(\rho_{32})u_{32}+o\left(\mathcal{F}(D)\right).
\end{aligned}
\end{align}
\end{lemma}
\begin{proof}
This follows by applying Lemma \ref{centerdynamics} with careful attention to the signs, and then performing a direct computation.
\end{proof}
Next, we investigate the identities associated with \eqref{Zrhoudef1} and \eqref{Zrhoudef2}. We define $c_{12},c_{13},c_{23}$ as 
\begin{align}\label{cdef}
c_{12}=u_1\cdot u_2,\ c_{13}=u_1\cdot u_3,\ c_{23}=u_2\cdot u_3.
\end{align}
We note that the matrix
\begin{align*}
\begin{pmatrix}
1 &c_{12} & c_{13}
\\
c_{12}&1&c_{23}
\\
c_{13}&c_{23}&1
\end{pmatrix}
\end{align*}
is a Gram matrix. Hence, let $A(t)$ denote its determinant; that is,
\begin{align}\label{Adef}
A(t)=1+2c_{12}c_{23}c_{13}-c_{12}^2-c_{13}^2-c_{23}^2.
\end{align}
Then, we have for $t\geq 0$
\begin{align}\label{Apositive}
A(t)\geq 0.
\end{align}
In this setting, we obtain the following identities.
\begin{lemma}\label{daijisiki}
Let $\vec{u}$ be a $(1,3)$-sign solitary waves. Then the following hold.
\begin{enumerate}
\item $Z_1,Z_2,Z_3$ satisfy for $t\gg 1$
\begin{align}\label{Zes1}
\begin{aligned}
\dot{Z}_1&=2\mathcal{F}(\rho_1)u_1+\mathcal{F}(\rho_2)u_2+\mathcal{F}(\rho_3)u_3
\\
&\quad +\mathcal{F}(\rho_{12})u_{12}+\mathcal{F}(\rho_{13})u_{13}+o\left(\mathcal{F}(D)\right),
\\
\dot{Z}_2&=\mathcal{F}(\rho_1)u_1+2\mathcal{F}(\rho_2)u_2+\mathcal{F}(\rho_3)u_3
\\
&\quad+\mathcal{F}(\rho_{21})u_{21}+\mathcal{F}(\rho_{23})u_{23}+o\left(\mathcal{F}(D)\right),
\\
\dot{Z}_3&=\mathcal{F}(\rho_1)u_1+\mathcal{F}(\rho_2)u_2+2\mathcal{F}(\rho_3)u_3
\\
&\quad +\mathcal{F}(\rho_{31})u_{31}+\mathcal{F}(\rho_{32})u_{32}+o\left(\mathcal{F}(D)\right).
\end{aligned}
\end{align}

\item $\rho_1,\ \rho_2,\ \rho_3$ satisfy for $t\gg 1$
\begin{align}\label{rhoes1}
\begin{aligned}
\dot{\rho}_1&=2\mathcal{F}(\rho_1)+\mathcal{F}(\rho_2)c_{12}+\mathcal{F}(\rho_3)c_{13}
\\
&\quad+\mathcal{F}(\rho_{12})u_{12}\cdot u_1+\mathcal{F}(\rho_{13})u_{13}\cdot u_1+o\left(\mathcal{F}(D)\right),
\\
\dot{\rho}_2&=\mathcal{F}(\rho_1)c_{12}+2\mathcal{F}(\rho_2)+\mathcal{F}(\rho_3)c_{23}
\\
&\quad+\mathcal{F}(\rho_{21})u_{21}\cdot u_2+\mathcal{F}(\rho_{23})u_{23}\cdot u_2+o\left(\mathcal{F}(D)\right),
\\
\dot{\rho}_3&=\mathcal{F}(\rho_{1})c_{13}+\mathcal{F}(\rho_2)c_{23}+2\mathcal{F}(\rho_3)
\\
&\quad+\mathcal{F}(\rho_{31})u_{31}\cdot u_3+\mathcal{F}(\rho_{32})u_{32}\cdot u_3+o\left(\mathcal{F}(D)\right).
\end{aligned}
\end{align}

\item $c_{12},c_{13},c_{23}$ satisfy for $t\gg 1$
\begin{align}\label{ces1}
\begin{aligned}
\dot{c}_{12}&=(1-c_{12}^2)\left\{\frac{\mathcal{F}(\rho_1)}{\rho_2}+\frac{\mathcal{F}(\rho_2)}{\rho_1}+\frac{\mathcal{F}(\rho_{12})}{\rho_{12}}\left(\frac{\rho_2}{\rho_1}+\frac{\rho_1}{\rho_2}\right)\right\}
\\
&\quad+\mathcal{F}(\rho_3)\left( \frac{c_{23}-c_{13}c_{12}}{\rho_1}+\frac{c_{13}-c_{23}c_{12}}{\rho_2}\right)
\\
&\quad+\frac{\mathcal{F}(\rho_{13})}{\rho_{13}}\frac{\rho_3(c_{23}-c_{12}c_{13})}{\rho_1}+\frac{\mathcal{F}(\rho_{23})}{\rho_{23}}\frac{\rho_3(c_{13}-c_{12}c_{23})}{\rho_2}+o\left(\frac{\mathcal{F}(D)}{D}\right),
\\
\dot{c}_{13}&=(1-c_{13}^2)\left\{\frac{\mathcal{F}(\rho_1)}{\rho_3}+\frac{\mathcal{F}(\rho_3)}{\rho_1}+\frac{\mathcal{F}(\rho_{13})}{\rho_{13}}\left(\frac{\rho_3}{\rho_1}+\frac{\rho_1}{\rho_3}\right)\right\}
\\
&\quad+\mathcal{F}(\rho_2)\left( \frac{c_{23}-c_{12}c_{13}}{\rho_1}+\frac{c_{12}-c_{13}c_{23}}{\rho_3}\right)
\\
&\quad+\frac{\mathcal{F}(\rho_{12})}{\rho_{12}}\frac{\rho_2(c_{23}-c_{12}c_{13})}{\rho_1}+\frac{\mathcal{F}(\rho_{23})}{\rho_{23}}\frac{\rho_2(c_{12}-c_{13}c_{23})}{\rho_3}+o\left(\frac{\mathcal{F}(D)}{D}\right),
\\
\dot{c}_{23}&=(1-c_{23}^2)\left\{\frac{\mathcal{F}(\rho_2)}{\rho_3}+\frac{\mathcal{F}(\rho_3)}{\rho_2}+\frac{\mathcal{F}(\rho_{23})}{\rho_{23}}\left(\frac{\rho_3}{\rho_2}+\frac{\rho_2}{\rho_3}\right)\right\}
\\
&\quad+\mathcal{F}(\rho_1)\left( \frac{c_{13}-c_{12}c_{23}}{\rho_2}+\frac{c_{12}-c_{13}c_{23}}{\rho_3}\right)
\\
&\quad+\frac{\mathcal{F}(\rho_{12})}{\rho_{12}}\frac{\rho_1(c_{13}-c_{12}c_{23})}{\rho_2}+\frac{\mathcal{F}(\rho_{13})}{\rho_{13}}\frac{\rho_1(c_{12}-c_{13}c_{23})}{\rho_3}+o\left(\frac{\mathcal{F}(D)}{D}\right).
\end{aligned}
\end{align}

\item $\rho_{12},\ \rho_{13},\ \rho_{23}$ satisfy for $t\gg 1$
\begin{align}\label{rhojkes1}
\begin{aligned}
\dot{\rho}_{12}&=-2\mathcal{F}(\rho_{12})+\mathcal{F}(\rho_2)(u_{12}\cdot u_2)+\mathcal{F}(\rho_1)(u_1\cdot u_{21})
\\
&\quad -\mathcal{F}(\rho_{23})(u_{21}\cdot u_{23})-\mathcal{F}(\rho_{13})(u_{12}\cdot u_{13})+o\left(\mathcal{F}(D)\right),
\\
\dot{\rho}_{13}&=-2\mathcal{F}(\rho_{13})+\mathcal{F}(\rho_1)(u_{31}\cdot u_1)+\mathcal{F}(\rho_3)(u_{13}\cdot u_3)
\\
&\quad-\mathcal{F}(\rho_{12})(u_{13}\cdot u_{12})-\mathcal{F}(\rho_{23})(u_{31}\cdot u_{32})+o\left(\mathcal{F}(D)\right),
\\
\dot{\rho}_{23}&=-2\mathcal{F}(\rho_{23})+\mathcal{F}(\rho_2)(u_{32}\cdot u_2)+\mathcal{F}(\rho_3)(u_{23}\cdot u_3)
\\
&\quad-\mathcal{F}(\rho_{12})(u_{21}\cdot u_{23})-\mathcal{F}(\rho_{13})(u_{32}\cdot u_{31})+o\left(\mathcal{F}(D)\right).
\end{aligned}
\end{align}
\end{enumerate}
\end{lemma}

\begin{proof}
By the symmetry of the formulas, it suffices to treat $Z_1,\rho_1,c_{12},\rho_{12}$, and the remaining cases follow in the same way.\ First, for $Z_1$, by Lemma \ref{(1,3)centerdynamics} we have
\begin{align*}
\dot{Z}_1&=\dot{z}_1-\dot{z}_0
\\
&=\left(\mathcal{F}(\rho_1)u_1+\mathcal{F}(\rho_{12})u_{12}+\mathcal{F}(\rho_{13})u_{13}\right)
\\
&\quad-\left(-\mathcal{F}(\rho_1)u_1-\mathcal{F}(\rho_2)u_2-\mathcal{F}(\rho_3)u_3\right)+o\left(\mathcal{F}(D)\right)
\\
&=2\mathcal{F}(\rho_1)u_1+\mathcal{F}(\rho_2)u_2+\mathcal{F}(\rho_3)u_3
\\
&\quad +\mathcal{F}(\rho_{12})u_{12}+\mathcal{F}(\rho_{13})u_{13}+o\left(\mathcal{F}(D)\right),
\end{align*}
and therefore the same computation yields \eqref{Zes1} for $Z_2,Z_3$.
\\
Next, we estimate $\rho_1$. By direct computation, we have
\begin{align*}
\dot{\rho}_1&=\frac{\dot{Z}_1\cdot Z_1}{\rho_1}
\\
&=u_1\cdot \dot{Z}_1
\\
&=2\mathcal{F}(\rho_1)+\mathcal{F}(\rho_2)c_{12}+\mathcal{F}(\rho_3)c_{13}
\\
&\quad+\mathcal{F}(\rho_{12})u_{12}\cdot u_1+\mathcal{F}(\rho_{13})u_{13}\cdot u_1+o\left(\mathcal{F}(D)\right).
\end{align*}
Thus, the same computation yields \eqref{rhoes1} for $\rho_2,\rho_3$.
\\
Next, we compute $c_{12}$.\ First, for $u_1$, a direct computation gives
\begin{align*}
\dot{u}_1&=\frac{\rho_1\dot{Z}_1-\dot{\rho}_1Z_1}{\rho_1^2}
\\
&=\frac{\dot{Z}_1-\dot{\rho}_1u_1}{\rho_1}
\\
&=\frac{2\mathcal{F}(\rho_1)u_1+\mathcal{F}(\rho_2)u_2+\mathcal{F}(\rho_3)u_3+\mathcal{F}(\rho_{12})u_{12}+\mathcal{F}(\rho_{13})u_{13}}{\rho_1}
\\
&\quad-\frac{2\mathcal{F}(\rho_1)+\mathcal{F}(\rho_2)c_{12}+\mathcal{F}(\rho_3)c_{13}+\mathcal{F}(\rho_{12})u_{12}\cdot u_1+\mathcal{F}(\rho_{13})u_{13}\cdot u_1}{\rho_1}u_1
\\
&\quad+o\left(\frac{\mathcal{F}(D)}{D}\right)
\\
&=\frac{\mathcal{F}(\rho_2)(u_2-c_{12}u_1)+\mathcal{F}(\rho_3)(u_3-c_{13}u_1)}{\rho_1}
\\
&\quad+\frac{\mathcal{F}(\rho_{12})(u_{12}-(u_{12}\cdot u_1)u_1)+\mathcal{F}(\rho_{13})(u_{13}-(u_{13}\cdot u_1)u_1)}{\rho_1}+o\left(\frac{\mathcal{F}(D)}{D}\right)
\end{align*}
as desired. Similarly, note that $u_2$ satisfies
\begin{align*}
\dot{u}_2&=\frac{\mathcal{F}(\rho_1)(u_1-c_{12}u_2)+\mathcal{F}(\rho_3)(u_3-c_{23}u_2)}{\rho_2}
\\
&\quad+\frac{\mathcal{F}(\rho_{12})(u_{21}-(u_{21}\cdot u_2)u_2)+\mathcal{F}(\rho_{23})(u_{23}-(u_{23}\cdot u_2)u_2)}{\rho_2}+o\left(\frac{\mathcal{F}(D)}{D}\right)
\end{align*}
and hence
\begin{align}\label{c12cal1}
\dot{c}_{12}&=\dot{u}_1\cdot u_2+u_1\cdot \dot{u}_2
\end{align}
can be written as the sum of
\begin{align}\label{c12cal2}
\begin{aligned}
\dot{u}_1\cdot u_2&=\frac{\mathcal{F}(\rho_2)(1-c_{12}^2)+\mathcal{F}(\rho_3)(c_{23}-c_{13}c_{12})}{\rho_1}
\\
&\quad+\frac{\mathcal{F}(\rho_{12})(u_{12}\cdot u_2-(u_{12}\cdot u_1)c_{12})+\mathcal{F}(\rho_{13})(u_{13}\cdot u_2-(u_{13}\cdot u_1)c_{12})}{\rho_1}
\\
&\quad+o\left(\frac{\mathcal{F}(D)}{D}\right)
\end{aligned}
\end{align}
and
\begin{align}\label{c12cal3}
\begin{aligned}
u_1\cdot \dot{u}_2&=\frac{\mathcal{F}(\rho_1)(1-c_{12}^2)+\mathcal{F}(\rho_3)(c_{13}-c_{23}c_{12})}{\rho_2}
\\
&\quad+\frac{ \mathcal{F}(\rho_{12})(u_{21}\cdot u_1-(u_{21}\cdot u_2)c_{12})+\mathcal{F}(\rho_{23})(u_{23}\cdot u_1-(u_{23}\cdot u_2)c_{12})}{\rho_2}
\\
&\quad+o\left(\frac{\mathcal{F}(D)}{D}\right).
\end{aligned}
\end{align}
Therefore, by \eqref{c12cal1}, \eqref{c12cal2}, \eqref{c12cal3} we obtain
\begin{align}\label{c12cal4}
\begin{aligned}
\dot{c}_{12}&=(1-c_{12}^2)\left(\frac{\mathcal{F}(\rho_1)}{\rho_2}+\frac{\mathcal{F}(\rho_2)}{\rho_1}\right)
\\
&\quad+\mathcal{F}(\rho_3)\left(\frac{c_{23}-c_{13}c_{12}}{\rho_1}+\frac{c_{13}-c_{23}c_{12}}{\rho_2} \right)
\\
&\quad+\mathcal{F}(\rho_{12})\left( \frac{u_{12}\cdot u_2-(u_{12}\cdot u_1)c_{12}}{\rho_1}+\frac{u_{21}\cdot u_1-(u_{21}\cdot u_2)c_{12}}{\rho_2}\right)
\\
&\quad+\mathcal{F}(\rho_{13})\frac{u_{13}\cdot u_2-(u_{13}\cdot u_1)c_{12}}{\rho_1}+\mathcal{F}(\rho_{23})\frac{u_{23}\cdot u_1-(u_{23}\cdot u_2)c_{12}}{\rho_2}
\\
&\quad+o\left(\frac{\mathcal{F}(D)}{D}\right).
\end{aligned}
\end{align}
Moreover,
\begin{align*}
u_{12}\cdot u_2-(u_{12}\cdot u_1)c_{12}&=\frac{(Z_2-Z_1)\cdot Z_2}{\rho_{12}\rho_2}-c_{12}\frac{(Z_2-Z_1)\cdot Z_1}{\rho_{12}\rho_1}
\\
&=\frac{\rho_2-\rho_1c_{12}}{\rho_{12}}
-c_{12}\frac{\rho_2c_{12}-\rho_1}{\rho_{12}}\\
&=\frac{\rho_2-\rho_1c_{12}-\rho_2c_{12}^2+\rho_1c_{12}}{\rho_{12}}
=\frac{\rho_2(1-c_{12}^2)}{\rho_{12}},
\end{align*}
and similarly
\begin{align*}
u_{21}\cdot u_1-(u_{21}\cdot u_2)c_{12}=\frac{\rho_1(1-c_{12}^2)}{\rho_{12}},
\end{align*}
which implies
\begin{align}\label{c12cal5}
\begin{aligned}
&\quad \mathcal{F}(\rho_{12})\left( \frac{u_{12}\cdot u_2-(u_{12}\cdot u_1)c_{12}}{\rho_1}+\frac{u_{21}\cdot u_1-(u_{21}\cdot u_2)c_{12}}{\rho_2}\right)
\\
&=(1-c_{12}^2)\frac{\mathcal{F}(\rho_{12})}{\rho_{12}}\left(\frac{\rho_2}{\rho_1}+\frac{\rho_1}{\rho_2}\right).
\end{aligned}
\end{align}
Next, a direct computation gives
\begin{align}\label{c12cal6}
\begin{aligned}
u_{13}\cdot u_2-(u_{13}\cdot u_1)c_{12}&=\frac{\rho_3c_{23}-\rho_1c_{12}}{\rho_{13}}
-c_{12}\frac{\rho_3c_{13}-\rho_1}{\rho_{13}}\\
&=\frac{\rho_3c_{23}-\rho_1c_{12}-\rho_3c_{12}c_{13}+\rho_1c_{12}}{\rho_{13}}\\
&=\frac{\rho_3\bigl(c_{23}-c_{12}c_{13}\bigr)}{\rho_{13}}.
\end{aligned}
\end{align}
and similarly
\begin{align}\label{c12cal7}
u_{23}\cdot u_1-(u_{23}\cdot u_2)c_{12}=\frac{\rho_3(c_{13}-c_{12}c_{23})}{\rho_{23}}.
\end{align}
Therefore, by \eqref{c12cal4}-\eqref{c12cal7}, we obtain the $c_{12}$ equation in \eqref{ces1}. The cases $c_{23},\ c_{13}$ follow by the same computation.
\\
Finally, we estimate $\rho_{12}$. Noting that $Z_{12}=Z_2-Z_1$, we obtain from \eqref{Zes1} that
\begin{align*}
\dot{Z}_{12}&=-\mathcal{F}(\rho_1)u_1+\mathcal{F}(\rho_2)u_2-2\mathcal{F}(\rho_{12})u_{12}+\mathcal{F}(\rho_{23})u_{23}-\mathcal{F}(\rho_{13})u_{13}+o\left(\mathcal{F}(D)\right)
\end{align*}
and hence
\begin{align*}
\dot{\rho}_{12}&=u_{12}\cdot \dot{Z}_{12}
\\
&=-2\mathcal{F}(\rho_{12})+\mathcal{F}(\rho_2)(u_{12}\cdot u_2)+\mathcal{F}(\rho_1)(u_1\cdot u_{21})
\\
&\quad -\mathcal{F}(\rho_{23})(u_{21}\cdot u_{23})-\mathcal{F}(\rho_{13})(u_{12}\cdot u_{13})+o\left(\mathcal{F}(D)\right),
\end{align*}
which yields the $\rho_{12}$ case of \eqref{rhojkes1}.\ The cases $\rho_{23},\ \rho_{13}$ are obtained similarly, and this lemma is proved.
\end{proof}

By the above computations, if $\vec{u}$ is a $(1,3)$-sign solitary waves and satisfies \eqref{Vrep}, then Lemma \ref{daijisiki} can be refined as in Lemma \ref{epzes}.
\begin{corollary}\label{corollary}
Let $\vec{u}$ be a $(1,3)$-sign solitary waves and suppose that it satisfies \eqref{Vrep}. Then for any $1<\theta<\min{(p-1,2)}$, the following hold.
\begin{enumerate}
\item $Z_1,Z_2,Z_3$ satisfy for $t\gg 1$
\begin{align}\label{Zes10210}
\begin{aligned}
\dot{Z}_1&=2\mathcal{F}(\rho_1)u_1+\mathcal{F}(\rho_2)u_2+\mathcal{F}(\rho_3)u_3
\\
&\quad +\mathcal{F}(\rho_{12})u_{12}+\mathcal{F}(\rho_{13})u_{13}+O\left(e^{-\theta D}\right),
\\
\dot{Z}_2&=\mathcal{F}(\rho_1)u_1+2\mathcal{F}(\rho_2)u_2+\mathcal{F}(\rho_3)u_3
\\
&\quad+\mathcal{F}(\rho_{21})u_{21}+\mathcal{F}(\rho_{23})u_{23}+O\left(e^{-\theta D}\right),
\\
\dot{Z}_3&=\mathcal{F}(\rho_1)u_1+\mathcal{F}(\rho_2)u_2+2\mathcal{F}(\rho_3)u_3
\\
&\quad +\mathcal{F}(\rho_{31})u_{31}+\mathcal{F}(\rho_{32})u_{32}+O\left(e^{-\theta D}\right).
\end{aligned}
\end{align}

\item $\rho_1,\ \rho_2,\ \rho_3$ satisfy for $t\gg 1$
\begin{align}\label{rhoes10210}
\begin{aligned}
\dot{\rho_1}&=2\mathcal{F}(\rho_1)+\mathcal{F}(\rho_2)c_{12}+\mathcal{F}(\rho_3)c_{13}
\\
&\quad+\mathcal{F}(\rho_{12})u_{12}\cdot u_1+\mathcal{F}(\rho_{13})u_{13}\cdot u_1+O\left(e^{-\theta D}\right),
\\
\dot{\rho}_2&=\mathcal{F}(\rho_1)c_{12}+2\mathcal{F}(\rho_2)+\mathcal{F}(\rho_3)c_{23}
\\
&\quad+\mathcal{F}(\rho_{21})u_{21}\cdot u_2+\mathcal{F}(\rho_{23})u_{23}\cdot u_2+O\left(e^{-\theta D}\right),
\\
\dot{\rho}_3&=\mathcal{F}(\rho_{1})c_{13}+\mathcal{F}(\rho_2)c_{23}+2\mathcal{F}(\rho_3)
\\
&\quad+\mathcal{F}(\rho_{31})u_{31}\cdot u_3+\mathcal{F}(\rho_{32})u_{32}\cdot u_3+O\left(e^{-\theta D}\right).
\end{aligned}
\end{align}

\item $c_{12},c_{13},c_{23}$ satisfy for $t\gg 1$
\begin{align}\label{ces10210}
\begin{aligned}
\dot{c}_{12}&=(1-c_{12}^2)\left\{\frac{\mathcal{F}(\rho_1)}{\rho_2}+\frac{\mathcal{F}(\rho_2)}{\rho_1}+\frac{\mathcal{F}(\rho_{12})}{\rho_{12}}\left(\frac{\rho_2}{\rho_1}+\frac{\rho_1}{\rho_2}\right)\right\}
\\
&\quad+\mathcal{F}(\rho_3)\left( \frac{c_{23}-c_{13}c_{12}}{\rho_1}+\frac{c_{13}-c_{23}c_{12}}{\rho_2}\right)
\\
&\quad+\frac{\mathcal{F}(\rho_{13})}{\rho_{13}}\frac{\rho_3(c_{23}-c_{12}c_{13})}{\rho_1}+\frac{\mathcal{F}(\rho_{23})}{\rho_{23}}\frac{\rho_3(c_{13}-c_{12}c_{23})}{\rho_2}+O\left(e^{-\theta D}\right),
\\
\dot{c}_{13}&=(1-c_{13}^2)\left\{\frac{\mathcal{F}(\rho_1)}{\rho_3}+\frac{\mathcal{F}(\rho_3)}{\rho_1}+\frac{\mathcal{F}(\rho_{13})}{\rho_{13}}\left(\frac{\rho_3}{\rho_1}+\frac{\rho_1}{\rho_3}\right)\right\}
\\
&\quad+\mathcal{F}(\rho_2)\left( \frac{c_{23}-c_{12}c_{13}}{\rho_1}+\frac{c_{12}-c_{13}c_{23}}{\rho_3}\right)
\\
&\quad+\frac{\mathcal{F}(\rho_{12})}{\rho_{12}}\frac{\rho_2(c_{23}-c_{12}c_{13})}{\rho_1}+\frac{\mathcal{F}(\rho_{23})}{\rho_{23}}\frac{\rho_2(c_{12}-c_{13}c_{23})}{\rho_3}+O\left(e^{-\theta D}\right),
\\
\dot{c}_{23}&=(1-c_{23}^2)\left\{\frac{\mathcal{F}(\rho_2)}{\rho_3}+\frac{\mathcal{F}(\rho_3)}{\rho_2}+\frac{\mathcal{F}(\rho_{23})}{\rho_{23}}\left(\frac{\rho_3}{\rho_2}+\frac{\rho_2}{\rho_3}\right)\right\}
\\
&\quad+\mathcal{F}(\rho_1)\left( \frac{c_{13}-c_{12}c_{23}}{\rho_2}+\frac{c_{12}-c_{13}c_{23}}{\rho_3}\right)
\\
&\quad+\frac{\mathcal{F}(\rho_{12})}{\rho_{12}}\frac{\rho_1(c_{13}-c_{12}c_{23})}{\rho_2}+\frac{\mathcal{F}(\rho_{13})}{\rho_{13}}\frac{\rho_1(c_{12}-c_{13}c_{23})}{\rho_3}+O\left(e^{-\theta D}\right).
\end{aligned}
\end{align}

\item $\rho_{12},\ \rho_{13},\ \rho_{23}$ satisfy for $t\gg 1$
\begin{align}\label{rhojkes10210}
\begin{aligned}
\dot{\rho}_{12}&=-2\mathcal{F}(\rho_{12})+\mathcal{F}(\rho_2)(u_{12}\cdot u_2)+\mathcal{F}(\rho_1)(u_1\cdot u_{21})
\\
&\quad -\mathcal{F}(\rho_{23})(u_{21}\cdot u_{23})-\mathcal{F}(\rho_{13})(u_{12}\cdot u_{13})+O\left(e^{-\theta D}\right),
\\
\dot{\rho}_{13}&=-2\mathcal{F}(\rho_{13})+\mathcal{F}(\rho_1)(u_{31}\cdot u_1)+\mathcal{F}(\rho_3)(u_{13}\cdot u_3)
\\
&\quad-\mathcal{F}(\rho_{12})(u_{13}\cdot u_{12})-\mathcal{F}(\rho_{23})(u_{31}\cdot u_{32})+O\left(e^{-\theta D}\right),
\\
\dot{\rho}_{23}&=-2\mathcal{F}(\rho_{23})+\mathcal{F}(\rho_2)(u_{32}\cdot u_2)+\mathcal{F}(\rho_3)(u_{23}\cdot u_3)
\\
&\quad-\mathcal{F}(\rho_{12})(u_{21}\cdot u_{23})-\mathcal{F}(\rho_{13})(u_{32}\cdot u_{31})+O\left(e^{-\theta D}\right).
\end{aligned}
\end{align}
\end{enumerate}
\end{corollary}

\section{Estimates on interactions between like-signed solitons}
In this section and what follows, we analyze the behavior of the soliton centers. In the $(1,3)$-sign configuration, there are six pairwise interactions, which makes the analysis of the long-time soliton dynamics more delicate. We therefore begin by comparing the distances between solitons of the same sign and those between solitons of opposite signs. To this end, we first introduce several definitions. First, we define $\tilde{D}$ and $\hat{D}$ as
\begin{align}\label{tildehatD}
\tilde{D}=\min{(\rho_{12},\rho_{23},\rho_{13})},\ \hat{D}=\min{(\rho_1,\rho_2,\rho_3)}.
\end{align}
Then, we have
\begin{align}\label{Datarimae1}
D=\min{(\tilde{D},\hat{D})}.
\end{align}
The goal of this section is to prove the following proposition concerning $\tilde{D}$.
\begin{proposition}\label{tildeDnagai}
Let $\vec{u}$ be a (1,3)-sign solitary waves. Then, we have for $t\gg 1$
\begin{align}\label{tildeDnagaieq}
\tilde{D}(t)-D(t)>D^{\frac{1}{5}}(t).
\end{align}
\end{proposition}

\subsection{Basic properties of $(1,3)$-sign solitary waves}
First, we collect the basic properties concerning $\tilde{D}$ and $\hat{D}$. By \eqref{Ves1}, we have $t\gg 1$
\begin{align}\label{repcon41}
V=\mathcal{F}(\rho_1)+\mathcal{F}(\rho_2)+\mathcal{F}(\rho_3)-\mathcal{F}(\rho_{12})-\mathcal{F}(\rho_{23})-\mathcal{F}(\rho_{13})>-\frac{1}{100}\mathcal{F}(D).
\end{align}
In particular, by \eqref{repcon41} when $\tilde{D}=D$ holds we have
\begin{align*}
\mathcal{F}(\tilde{D})<\frac{100}{33}\mathcal{F}(\hat{D}),
\end{align*}
and we obtain
\begin{align}\label{repcon42.5}
\mathcal{F}(\hat{D})>\frac{33}{100}\mathcal{F}(D).
\end{align}
By \eqref{repcon42.5}, there exists a constant $C_{\circ}>0$ such that for $t\gg 1$
\begin{align}\label{repcon43}
\hat{D}-C_{\circ}<\tilde{D}.
\end{align}
In this section, we need to analyze delicate variations in the inter-soliton distances. To this end, we establish the following lemma concerning $D$ and $\mathcal{F}(D)$.
\begin{lemma}\label{henchoumod}
Let $0<\epsilon<1$. Then for any $T\gg 1$, there exists $T^{\prime}>T$ such that the following hold.
\begin{enumerate}
\item For $T\leq t\leq T^{\prime}$,
\begin{align}
\left|D(t)-D(T)\right|\leq 2D^{\epsilon}(T).\label{TTprimekawaran}
\end{align}

\item Moreover, there exists $0<c_{\bullet}<1$ depending on $\vec{u}$ such that 
\begin{align}
 c_{\bullet}D^{\epsilon}(T)\leq \int_T^{T^{\prime}} \mathcal{F}(D(s))ds\leq \frac{1}{c_{\bullet}}D^{\epsilon}(T). \label{skeibunes41}
\end{align}
\end{enumerate}
\end{lemma}
\begin{proof}
Fix $\epsilon>0$. First, by Lemma \ref{daijisiki}, there exists $C>1$ depending on $\vec{u}$ such that for $t\geq T$
\begin{align}
|\dot{D}|\leq C\mathcal{F}(D). \label{Dosaeru}
%\\
%|\dot{c}_{12}|+|\dot{c}_{13}|+|\dot{c}_{23}|\leq \frac{C\mathcal{F}(D)}{D}. \label{cjkosaeru}
\end{align}
By Lemma \ref{limeplem}, the function $\frac{\mathcal{F}(D)}{D}$ is not in $L^1$, hence for any $T\gg 1$ there exists $T^{\prime}>T$ such that
\begin{align}\label{Tpridef}
\int_T^{T^{\prime}} \frac{\mathcal{F}(D(s))ds}{D(s)}=\frac{1}{C}D^{\epsilon-1}(T).
\end{align}
We show that, for this choice of $T^{\prime}$, \eqref{TTprimekawaran}, and \eqref{skeibunes41} hold.
%\begin{align}
%\left| c_{jk}(t)-c_{jk}(T)\right|\leq D^{-\epsilon}(T). \label{cjkes41} 
%\end{align}
%To prove \eqref{cjkes41}, fix $1\leq j<k\leq 3$ and take any $T\leq t\leq T^{\prime}$. Then by \eqref{cjkosaeru}, we have
%\begin{align*}
%|c_{jk}(t)-c_{jk}(T)|
%\leq \int_T^{t}|\dot{c}_{jk}(s)|ds
%\leq C\int_T^{T^{\prime}} \frac{\mathcal{F}(D(s))ds}{D(s)}
%= D^{-\epsilon}(T),
%\end{align*}
%which yields \eqref{cjkes41}.

First, we prove \eqref{TTprimekawaran}. For any $T\leq t\leq T^{\prime}$, we have
\begin{align}
\begin{aligned}
\left| \log{D(t)}-\log{D(T)}\right|
&=\left|\int_T^t \frac{\dot{D}(s)ds}{D(s)}\right|
\\
&\leq \int_T^t \frac{|\dot{D}(s)|ds}{D(s)}
\\
&\leq C\int_T^{T^{\prime}} \frac{\mathcal{F}(D(s))ds}{D(s)}
= D^{\epsilon-1}(T).
\end{aligned}
\end{align}
In particular, noting that $D\gg 1$, the above implies that for any $T\leq t\leq T^{\prime}$,
\begin{align*}
D(T)(1-2D^{\epsilon-1}(T))\leq e^{-D^{\epsilon-1}(T)}D(T),
\\
e^{-D^{\epsilon-1}(T)}D(T)\leq D(t)\leq e^{D^{\epsilon-1}(T)}D(T),
\\
e^{D^{\epsilon-1}(T)}D(T)\leq D(T)(1+2D^{\epsilon-1}(T)),
\end{align*}
which implies \eqref{TTprimekawaran}.

Next, we prove \eqref{skeibunes41}. By \eqref{TTprimekawaran}, we have for $T\leq t\leq T^{\prime}$
\begin{align*}
\frac{1}{2}D(T)\leq D(t)\leq 2D(T).
\end{align*}
Then by \eqref{Tpridef}, we have
\begin{align*}
D^{\epsilon-1}(T)&=C\int_T^{T^{\prime}} \frac{\mathcal{F}(D(s))ds}{D(s)}
\\
&\leq \frac{2C}{D(T)}\int_T^{T^{\prime}}\mathcal{F}(D(s))ds,
\end{align*}
and
\begin{align*}
D^{\epsilon-1}(T)&=C\int_T^{T^{\prime}} \frac{\mathcal{F}(D(s))ds}{D(s)}
\\
&\geq \frac{C}{2D(T)}\int_T^{T^{\prime}}\mathcal{F}(D(s))ds,
\end{align*}
which yields \eqref{skeibunes41} by choosing $c_{\bullet}=\min{(\frac{1}{2C},\frac{C}{2})}$. This completes the proof.
\end{proof}
%\begin{lemma}
%ある定数$c>0$が存在して,\ 以下が成立する: 任意の$T>0$に対して,\ ある$T^{\prime}>0$とが存在して,\ $1\leq j<k\leq 3$で
%\begin{align}
%\left|c_{jk}(T^{\prime})-C_{jk}(T)\right|\leq \frac{1}{200},
%\\
%\int_T^{T^{\prime}}\mathcal{F}(D(s))ds\geq c\max{(D(T^{\prime}),D(T))}.
%\end{align}
%\end{lemma}
Before proceeding to the analysis of distances, we record estimates on the angles of triangles with large side lengths.
\begin{lemma}\label{sankakukeilemma}
We assume $M\gg 1$. Then, $\varsigma_1,\varsigma_2,\varsigma_3\in\mathbb{R}^d$ satisfy the following.
\begin{enumerate}
\item When $\varsigma_1,\varsigma_2,\varsigma_3$ satisfy
\begin{align*}
\left|\ |\varsigma_1-\varsigma_2|-M\ \right|&\leq M^{\frac{99}{100}},
\\
\left|\ |\varsigma_1-\varsigma_3|-M\ \right|&\leq M^{\frac{99}{100}},
\\
|\varsigma_2-\varsigma_3|-M&\geq  -M^{\frac{99}{100}},
\end{align*}
we have
\begin{align*}
\frac{\varsigma_2-\varsigma_1}{|\varsigma_2-\varsigma_1|}\cdot \frac{\varsigma_3-\varsigma_1}{|\varsigma_3-\varsigma_1|}-\frac{1}{2}&\leq 5M^{-\frac{1}{100}},
\\
\frac{\varsigma_1-\varsigma_2}{|\varsigma_1-\varsigma_2|}\cdot \frac{\varsigma_3-\varsigma_2}{|\varsigma_3-\varsigma_2|}-\frac{1}{2}&\geq -5M^{-\frac{1}{100}},
\\
\frac{\varsigma_1-\varsigma_3}{|\varsigma_1-\varsigma_3|}\cdot \frac{\varsigma_2-\varsigma_3}{|\varsigma_2-\varsigma_3|}-\frac{1}{2}&\geq -5M^{-\frac{1}{100}}.
\end{align*}

\item When $\varsigma_1,\varsigma_2,\varsigma_3$ satisfy
\begin{align*}
\left|\ |\varsigma_1-\varsigma_2|-M\ \right|&\leq M^{\frac{99}{100}},
\\
|\varsigma_1-\varsigma_3|-M&\geq -M^{\frac{99}{100}},
\\
|\varsigma_2-\varsigma_3|-M&\geq -M^{\frac{99}{100}},
\end{align*}
we have
\begin{align*}
\frac{\varsigma_1-\varsigma_3}{|\varsigma_1-\varsigma_3|}\cdot \frac{\varsigma_2-\varsigma_3}{|\varsigma_2-\varsigma_3|}-\frac{1}{2}&\geq -5M^{-\frac{1}{100}}.
\end{align*}
\end{enumerate}
\end{lemma}
\begin{remark}
The lemma depends only on the configuration of three points and is therefore essentially two-dimensional. Thus, to gain geometric intuition, one may restrict attention to a planar triangle, in which the claim is most easily visualized.
\end{remark}
\begin{proof}
We define
\begin{align*}
\Lambda_1=|\varsigma_3-\varsigma_2|,\ \Lambda_2=|\varsigma_3-\varsigma_1|,\ \Lambda_3=|\varsigma_2-\varsigma_1|.
\end{align*}
%Furthermore, we define $\tau_1,\tau_2,\tau_3$ as 
%\begin{align*}
%\Lambda_k=M+\tau_k
%\end{align*}
%for $k=1,2,3$. 
First, we prove (1). By assumption, we have
\begin{align*}
\frac{\varsigma_2-\varsigma_1}{|\varsigma_2-\varsigma_1|}\cdot \frac{\varsigma_3-\varsigma_1}{|\varsigma_3-\varsigma_1|}&=\frac{\Lambda_2^2+\Lambda_3^2-\Lambda_1^2}{2\Lambda_2\Lambda_3}
\\
&\leq \frac{ 2(M+M^{\frac{99}{100}})^2-(M-M^{\frac{99}{100}})^2}{2(M-M^{\frac{99}{100}})^2}
\\
&\leq \frac{ M^2+7M^{\frac{199}{100}}}{2(M-M^{\frac{99}{100}})^2}
\\
&=\frac{ 1+\frac{7}{M^{\frac{1}{100}}}}{2\left(1-\frac{1}{M^{\frac{1}{100}}}\right)^2}
\\
&\leq \frac{1}{2}+\frac{5}{M^{\frac{1}{100}}}.
\end{align*}
Furthermore, we have
\begin{align*}
\frac{\varsigma_1-\varsigma_2}{|\varsigma_1-\varsigma_2|}\cdot \frac{\varsigma_3-\varsigma_2}{|\varsigma_3-\varsigma_2|}&=\frac{\Lambda_1^2+\Lambda_3^2-\Lambda_2^2}{2\Lambda_1\Lambda_3}
\\
&\geq \frac{ \Lambda_1^2+(M-M^{\frac{99}{100}})^2-(M+M^{\frac{99}{100}})^2}{2\Lambda_1\Lambda_3}
\\
&=\frac{\Lambda_1}{2\Lambda_3}-\frac{4M^{\frac{199}{100}}}{2\Lambda_1\Lambda_3}
\\
&\geq \frac{ M-M^{\frac{99}{100}}}{2(M+M^{\frac{99}{100}})}-\frac{4M^{\frac{199}{100}}}{2(M-M^{\frac{99}{100}})^2}
\\
&\geq \frac{1}{2}-\frac{5}{M^{\frac{1}{100}}}.
\end{align*}
By the same calculation, we obtain
\begin{align*}
\frac{\varsigma_1-\varsigma_3}{|\varsigma_1-\varsigma_3|}\cdot \frac{\varsigma_2-\varsigma_3}{|\varsigma_2-\varsigma_3|}\geq \frac{1}{2}-\frac{5}{M^{\frac{1}{100}}},
\end{align*}
and therefore we prove (1).

Next we prove (2). By the direct computation, we have
\begin{align*}
\frac{\varsigma_1-\varsigma_3}{|\varsigma_1-\varsigma_3|}\cdot \frac{\varsigma_2-\varsigma_3}{|\varsigma_2-\varsigma_3|}&=\frac{\Lambda_1^2+\Lambda_2^2-\Lambda_3^2}{2\Lambda_1\Lambda_2}
\\
&\geq \frac{2\Lambda_1\Lambda_2-(M+M^{\frac{99}{100}})^2}{2\Lambda_1\Lambda_2}
\\
&\geq 1-\frac{(M+M^{\frac{99}{100}})^2}{2(M-M^{\frac{99}{100}})^2}
\\
&= 1-\frac{(1+\frac{1}{M^{\frac{1}{100}}})^2}{2(1-\frac{1}{M^{\frac{1}{100}}})^2}
\\
&\geq \frac{1}{2}-\frac{5}{M^{\frac{1}{100}}}. 
\end{align*}
Therefore we prove (2).
\end{proof}

\subsection{Distance analysis I}

In Subsections 4.2-4.4, we show that, for a configuration of centers $z_0,z_1,z_2,z_3$, a $(1,3)$-sign solitary waves solution cannot occur in the regime where $\mathcal{F}(\tilde{D})$ exerts a strong influence on the center of each soliton. We begin with the case in which $z_1,z_2,z_3$ form an almost equilateral triangle (and where $\mathcal{F}(\tilde{D})$ is influential), and prove that such a configuration cannot produce a $(1,3)$-sign solitary waves. To this end, we first establish the following lemma.
\begin{lemma}\label{henchousystem}
Let $\vec{u}$ be a (1,3)-sign solitary waves.\ Furthermore, we assume that there exist $1\ll T_1<T_2\leq \infty$ such that we have for $T_1\leq t<T_2$
\begin{align}
\tilde{D}(t)-\hat{D}(t)<D^{\frac{9}{10}}(t),\label{tilhatD}
\\
\max{(\rho_{12}(t),\rho_{13}(t),\rho_{23}(t))}<\tilde{D}(t)+D^{\frac{9}{10}}(t).\label{maxtilD}
\end{align}
Then, we have for $T_1\leq t<T_2$
\begin{align}\label{osidasi1}
\begin{aligned}
&\quad \left(\hat{D}(t)-\frac{\rho_{12}(t)+\rho_{13}(t)+\rho_{23}(t)}{3}\right)-\left(\hat{D}(T_1)-\frac{\rho_{12}(T_1)+\rho_{13}(T_1)+\rho_{23}(T_1)}{3}\right)
\\
&\gtrsim \int_{T_1}^{t} \mathcal{F}(D(s))ds.
\end{aligned}
\end{align}
\end{lemma} 
%\begin{remark}
%このセクションでの議論で重要なのは,\ 三角形の三辺の長さと角度の関係から内積の不等式評価を行う点である.\ 一般次元であっても3点をとると平面上の三角形とみなすことができ,\ その角度に関しては余弦定理から考えられるが,\ ここででてくる$\cos{\theta}$に対応するのが内積になるので,\ このセクションの議論は初等幾何の考え方を忘れないようにすることが重要である.
%\end{remark}
\begin{proof}
First, by \eqref{repcon43} and \eqref{tilhatD}, we have for $T_1\leq t<T_2$
\begin{align}\label{tilDes0211}
\left| \tilde{D}(t)-D(t)\right|+\left|\hat{D}(t)-D(t)\right|<D^{\frac{19}{20}}(t).
\end{align}
Since $\rho_{12},\ \rho_{13},\ \rho_{23}$ satisfy \eqref{maxtilD}, by Lemma \ref{sankakukeilemma} and \eqref{tilDes0211}, we have
\begin{align*}
\min{(u_{12}\cdot u_{13},u_{21}\cdot u_{23},u_{31}\cdot u_{32})}-\frac{1}{2}\geq -10D^{-\frac{1}{100}}.
\end{align*}
Furthermore, by \eqref{tilDes0211} and $\rho_1,\rho_2,\rho_3\geq \hat{D}$, we have
\begin{align*}
\min{(\rho_1,\rho_2,\rho_3)}-D\geq -D^{\frac{99}{100}}.
\end{align*}
Therefore, by Lemma \ref{sankakukeilemma}, $c_{12},\ c_{13},\ c_{23}$ satisfy
\begin{align}\label{cdot0}
\min{(c_{12},c_{13},c_{23})}-\frac{1}{2}\geq -10D^{-\frac{1}{100}}.
\end{align}
By \eqref{rhojkes1}, we obtain
\begin{align}\label{avees1}
\begin{aligned}
\frac{\dot{\rho}_{12}+\dot{\rho}_{13}+\dot{\rho}_{23}}{3}&\leq -\left(\mathcal{F}(\rho_{12})+\mathcal{F}(\rho_{13})+\mathcal{F}(\rho_{23})\right)
\\
&\quad+\frac{2}{3}\left(\mathcal{F}(\rho_1)+\mathcal{F}(\rho_2)+\mathcal{F}(\rho_3)\right)+o(\mathcal{F}(D)).
\end{aligned}
\end{align}
Here, we assume $\rho_1=\hat{D}$. Then, by \eqref{maxtilD} and \eqref{tilDes0211}, we have 
\begin{align*}
|\rho_1-D|&\leq D^{\frac{99}{100}},
\\
|\rho_{12}-D|&\leq D^{\frac{99}{100}},
\\
\rho_2-D&\geq -D^{\frac{99}{100}}.
\end{align*}
Therefore, by Lemma \ref{sankakukeilemma}, we have
\begin{align}\label{cdot1}
u_1\cdot u_{12}+\frac{1}{2}\geq -5D^{-\frac{1}{100}}.
\end{align}
By the same argument, we have
\begin{align}\label{cdot2}
u_1\cdot u_{13}+\frac{1}{2}\geq -5D^{-\frac{1}{100}}.
\end{align}
By \eqref{rhoes1},\ \eqref{cdot0}, \eqref{cdot1}, and \eqref{cdot2}, we obtain
\begin{align}\label{rho1mines1}
\dot{\rho}_1&\geq 2\mathcal{F}(\rho_1)+\frac{1}{2}\left(\mathcal{F}(\rho_2)+\mathcal{F}(\rho_3)\right)-\frac{1}{2}\left(\mathcal{F}(\rho_{12})+\mathcal{F}(\rho_{13})\right)+o\left(\mathcal{F}(D)\right).
\end{align}
By \eqref{avees1} and \eqref{rho1mines1}, we have for $t\gg 1$
\begin{align*}
\dot{\rho}_1-\frac{\dot{\rho}_{12}+\dot{\rho}_{13}+\dot{\rho}_{23}}{3}&\geq \frac{4}{3}\mathcal{F}(\rho_1)-\frac{1}{6}\left(\mathcal{F}(\rho_2)+\mathcal{F}(\rho_3)\right)+o\left(\mathcal{F}(D)\right)
\\
&\gtrsim \mathcal{F}(D).
\end{align*}
When $\rho_2=\hat{D}$ or $\rho_3=\hat{D}$ hold, by the same argument, we obtain
\begin{align}\label{131es1}
\dot{\hat{D}}-\frac{\dot{\rho}_{12}+\dot{\rho}_{13}+\dot{\rho}_{23}}{3}\gtrsim \mathcal{F}(D).
\end{align}
Integrating \eqref{131es1} from $T_1$ to $t$ yields \eqref{osidasi1} and completes the proof.
\end{proof}
Using this lemma, we show that if $\vec{u}$ is a $(1,3)$-sign solitary waves, then at least one of $\rho_{12}, \rho_{13}, \rho_{23}$ must be sufficiently larger than $D$.
\begin{lemma}\label{hitotuhanagai}
Let $\vec{u}$ be a $(1,3)$-sign solitary waves. Furthermore, we assume for $T_1\leq t\leq T_2$
\begin{align}\label{estimate0124}
\tilde{D}(t)-\hat{D}(t)<D^{\frac{1}{2}}(t).
\end{align}
Then, we have for $T_1\leq t\leq T_2$
\begin{align*}
\max{(\rho_{12},\rho_{13},\rho_{23})}\geq \tilde{D}+D^{\frac{2}{3}}.
\end{align*}
\end{lemma}
\begin{proof}
We assume that there exists $T_1\leq t^{\prime}\leq T_2$ such that 
\begin{align}\label{DA11}
\max{(\rho_{12}(t^{\prime}),\rho_{13}(t^{\prime}),\rho_{23}(t^{\prime}))}<\tilde{D}(t^{\prime})+D^{\frac{2}{3}}(t^{\prime}).
\end{align} 
Then, we introduce the following bootstrap estimate
\begin{align}\label{bsass}
\hat{D}(t)-\frac{\rho_{12}(t)+\rho_{13}(t)+\rho_{23}(t)}{3}+D^{\frac{2}{3}}(t)\geq 0.
\end{align}
By Lemma \ref{henchoumod}, there exists $T_3$ such that we have for $t^{\prime}\leq t\leq T_3$
\begin{align}\label{DA12}
|D(t)-D(t^{\prime})|\leq 2D^{\frac{7}{10}}(t^{\prime}),
\end{align}
and we have
\begin{align}\label{DA13}
c_{\bullet} D^{\frac{7}{10}}(t^{\prime})\leq \int_{t^{\prime}}^{T_3}\mathcal{F}(D(s))ds\leq \frac{1}{c_{\bullet}} D^{\frac{7}{10}}(t^{\prime}).
\end{align}
Furthermore, by Lemma \ref{daijisiki}, we have
\begin{align}\label{DA0211}
\begin{aligned}
\left|\max{(\rho_{12}(t),\rho_{13}(t),\rho_{23}(t))}-\max{(\rho_{12}(t^{\prime}),\rho_{13}(t^{\prime}),\rho_{23}(t^{\prime}))}\right|&\lesssim \int_{t^{\prime}}^{t}\mathcal{F}(D(s))ds,
\\
\left|\tilde{D}(t)-\tilde{D}(t^{\prime})\right|&\lesssim\int_{t^{\prime}}^{t}\mathcal{F}(D(s))ds.
\end{aligned}
\end{align}
By \eqref{repcon43}, \eqref{DA12}, \eqref{DA13}, and \eqref{DA0211}, we have for $t^{\prime}\leq t\leq T_3$
\begin{align*}
\tilde{D}(t)-\hat{D}(t)&\lesssim \tilde{D}(t^{\prime})-\hat{D}(t^{\prime})+D^{\frac{7}{10}}(t^{\prime})
\\
&\lesssim D^{\frac{7}{10}}(t^{\prime}),
\end{align*}
and
\begin{align*}
&\quad \max{(\rho_{12}(t),\rho_{13}(t),\rho_{23}(t))}-\tilde{D}(t)
\\
&\lesssim \max{(\rho_{12}(t^{\prime}),\rho_{13}(t^{\prime}),\rho_{23}(t^{\prime}))}-\tilde{D}(t^{\prime})+D^{\frac{7}{10}}(t^{\prime}).
\\
&\lesssim D^{\frac{7}{10}}(t^{\prime}).
\end{align*}
In particular, by \eqref{DA12} and $D\gg 1$, we have for $t^{\prime}\leq t\leq T_3$
\begin{align}\label{Lemmatukaeru1}
\begin{aligned}
\tilde{D}(t)-\hat{D}(t)<D^{\frac{9}{10}}(t),
\\
\max{(\rho_{12}(t),\rho_{13}(t),\rho_{23}(t))}<\tilde{D}(t)+D^{\frac{9}{10}}(t).
\end{aligned}
\end{align}
Furthermore, since we have
\begin{align*}
&\quad \hat{D}(t^{\prime})-\frac{\rho_{12}(t^{\prime})+\rho_{13}(t^{\prime})+\rho_{23}(t^{\prime})}{3}+D^{\frac{2}{3}}(t^{\prime})
\\
&\geq \tilde{D}(t^{\prime})-D^{\frac{1}{2}}(t^{\prime})-\frac{\rho_{12}(t^{\prime})+\rho_{13}(t^{\prime})+\rho_{23}(t^{\prime})}{3}+D^{\frac{2}{3}}(t^{\prime})
\\
&\geq \frac{D^{\frac{2}{3}}(t^{\prime})}{3}-D^{\frac{1}{2}}(t^{\prime})> 0,
\end{align*}
we define $T_4\in [t^{\prime},T_3]$ by
\begin{align}\label{bs1.sono1}
T_4=\sup{\{t\in[t^{\prime},T_3]\ \mbox{such that \eqref{bsass} holds on} [t^{\prime},t]\}}.
\end{align}
We assume $T_4<T_3$. Then, we have 
\begin{align*}
\hat{D}(T_4)-\frac{\rho_{12}(T_4)+\rho_{13}(T_4)+\rho_{23}(T_4)}{3}+D^{\frac{2}{3}}(T_4)=0.
\end{align*}
Then, by \eqref{Lemmatukaeru1}, we may apply Lemma \ref{henchousystem}. Therefore, by \eqref{131es1} and $|\frac{d}{dt}D^{\frac{2}{3}}|\lesssim D^{-\frac{1}{3}}\mathcal{F}(D)$, we obtain
\begin{align*}
\frac{d}{dt}\left(\hat{D}-\frac{\rho_{12}+\rho_{23}+\rho_{13}}{3}+D^{\frac{2}{3}}\right)(T_4)>0.
\end{align*}
Therefore, we have $T_4=T_3$. Then, by Lemma \ref{henchousystem} and \eqref{DA13}, there exists $c>0$ such that 
\begin{align}\label{DA14}
\begin{aligned}
&\quad \left(\hat{D}(T_3)-\frac{\rho_{12}(T_3)+\rho_{13}(T_3)+\rho_{23}(T_3)}{3}\right)-\left(\hat{D}(t^{\prime})-\frac{\rho_{12}(t^{\prime})+\rho_{13}(t^{\prime})+\rho_{23}(t^{\prime})}{3}\right)
\\
&\geq c\int_{t^{\prime}}^{T_3}\mathcal{F}(D(s))ds
\\
&\geq cc_{\bullet}D^{\frac{7}{10}}(t^{\prime}).
\end{aligned}
\end{align}
On the other hand, since we have \eqref{estimate0124}, \eqref{DA14},
\begin{align*}
\frac{\rho_{12}(T_3)+\rho_{13}(T_3)+\rho_{23}(T_3)}{3}&\geq \tilde{D}(T_3),
\end{align*}
and
\begin{align*}
\frac{\rho_{12}(t^{\prime})+\rho_{13}(t^{\prime})+\rho_{23}(t^{\prime})}{3}&<\tilde{D}(t^{\prime})+D^{\frac{2}{3}}(t^{\prime}),
\end{align*}
we obtain
\begin{align*}
&\quad \hat{D}(T_3)-\tilde{D}(T_3)+D^{\frac{1}{2}}(t^{\prime})
\\
&\geq \hat{D}(T_3)-\tilde{D}(T_3)+\left( \tilde{D}(t^{\prime})-\hat{D}(t^{\prime})\right)
\\
&\geq \left(\hat{D}(T_3)-\frac{\rho_{12}(T_3)+\rho_{13}(T_3)+\rho_{23}(T_3)}{3}\right)-\left(\hat{D}(t^{\prime})-\frac{\rho_{12}(t^{\prime})+\rho_{13}(t^{\prime})+\rho_{23}(t^{\prime})}{3}\right)
\\
&\quad -D^{\frac{2}{3}}(t^{\prime})
\\
&\geq cc_{\bullet}D^{\frac{7}{10}}(t^{\prime})-D^{\frac{2}{3}}(t^{\prime}).
\end{align*}
Therefore, we obtain
\begin{align*}
\tilde{D}(T_3)-\hat{D}(T_3)\leq -cc_{\bullet}D^{\frac{7}{10}}(t^{\prime})+D^{\frac{2}{3}}(t^{\prime})+D^{\frac{1}{2}}(t^{\prime}),
\end{align*}
which contradicts \eqref{repcon43}. Therefore we complete the proof.

\end{proof}

\subsection{Distance analysis II}

In the previous subsection, we showed that at least one of $\rho_{12},\ \rho_{23},\ \rho_{13}$ stays away from $D$, so that the corresponding inter-soliton interaction has essentially no effect.
Next, we show that the interaction associated with the \emph{second} farthest distance among these three is also negligible.

\begin{lemma}\label{henchoumid2}
Let $\vec{u}$ be a (1,3)-sign solitary waves . Furthermore, we assume that there exist $1\ll T_1<T_2\leq \infty$ such that we have for $T_1\leq t<T_2$
\begin{align}
\tilde{D}(t)-\hat{D}(t)<D^{\frac{1}{2}}(t), \label{lemass01251}
\\
\m{(\rho_{12}(t),\rho_{13}(t),\rho_{23}(t))}<\tilde{D}(t)+D^{\frac{1}{2}}(t). \label{lemass01252}
\end{align}
Then, we have for $T_1\leq t<T_2$
\begin{align}
\begin{aligned}
&\quad \left(\hat{D}(t)-\frac{\min{(\rho_{12}(t),\rho_{13}(t),\rho_{23}(t))}+\m{(\rho_{12}(t),\rho_{13}(t),\rho_{23}(t))}}{2}\right)
\\
&-\left(\hat{D}(T_1)-\frac{\min{(\rho_{12}(T_1),\rho_{13}(T_1),\rho_{23}(T_1))}+\m{(\rho_{12}(T_1),\rho_{13}(T_1),\rho_{23}(T_1))}}{2}\right)
\\
&\gtrsim \int_{T_1}^t \mathcal{F}(D(s))ds.
\end{aligned}
\end{align}
\end{lemma}
\begin{proof}
By symmetry, we may assume that we have for $T_1\leq t<T_2$
\begin{align*}
\rho_{12}(t)&<\tilde{D}(t)+D^{\frac{1}{2}}(t),
\\
\rho_{13}(t)&<\tilde{D}(t)+D^{\frac{1}{2}}(t).
\end{align*}
In this case, by Lemma \ref{hitotuhanagai}, we have $\rho_{23}(t)-D(t)\gtrsim D^{\frac{2}{3}}(t)$ for $T_1\leq t\leq T_2$. Moreover, since $\rho_{12},\ \rho_{23},\ \rho_{13}$ are continuous and \eqref{lemass01252} holds on $[T_1,T_2)$, we may assume that we have for $T_1\leq t<T_2$
\begin{align*}
\rho_{23}(t)=\max{(\rho_{12}(t),\rho_{13}(t),\rho_{23}(t))}.
\end{align*}
Then, by \eqref{rhoes1} and \eqref{rhojkes1}, we have
\begin{align*}
\dot{\rho_1}&=2\mathcal{F}(\rho_1)+\mathcal{F}(\rho_2)c_{12}+\mathcal{F}(\rho_3)c_{13}
\\
&\quad+\mathcal{F}(\rho_{12})u_{12}\cdot u_1+\mathcal{F}(\rho_{13})u_{13}\cdot u_1+o\left(\mathcal{F}(D)\right),
\\
\dot{\rho}_2&=\mathcal{F}(\rho_1)c_{12}+2\mathcal{F}(\rho_2)+\mathcal{F}(\rho_3)c_{23}
\\
&\quad+\mathcal{F}(\rho_{21})u_{21}\cdot u_2+o\left(\mathcal{F}(D)\right),
\\
\dot{\rho}_3&=\mathcal{F}(\rho_{1})c_{13}+\mathcal{F}(\rho_2)c_{23}+2\mathcal{F}(\rho_3)
\\
&\quad+\mathcal{F}(\rho_{31})u_{31}\cdot u_3+o\left(\mathcal{F}(D)\right),
\\
\dot{\rho}_{12}&=-2\mathcal{F}(\rho_{12})+\mathcal{F}(\rho_2)(u_{12}\cdot u_2)+\mathcal{F}(\rho_1)(u_1\cdot u_{21})
\\
&\quad -\mathcal{F}(\rho_{13})(u_{12}\cdot u_{13})+o\left(\mathcal{F}(D)\right),
\\
\dot{\rho}_{13}&=-2\mathcal{F}(\rho_{13})+\mathcal{F}(\rho_1)(u_{31}\cdot u_1)+\mathcal{F}(\rho_3)(u_{13}\cdot u_3)
\\
&\quad-\mathcal{F}(\rho_{12})(u_{13}\cdot u_{12})+o\left(\mathcal{F}(D)\right).
\end{align*}
In particular, we have
\begin{align}\label{1213avees}
\begin{aligned}
\frac{\dot{\rho}_{12}+\dot{\rho}_{13}}{2}&\leq -\frac{\mathcal{F}(\rho_{12})+\mathcal{F}(\rho_{13})}{2}+\frac{u_1\cdot u_{21}+u_1\cdot u_{31}}{2}\mathcal{F}(\rho_1)
\\
&\quad+\frac{u_{12}\cdot u_2}{2}\mathcal{F}(\rho_2)+\frac{u_{13}\cdot u_3}{2}\mathcal{F}(\rho_3)+o\left(\mathcal{F}(D)\right).
\end{aligned}
\end{align}
By \eqref{repcon43}, \eqref{lemass01251}, and \eqref{lemass01252}, we have
\begin{align*}
|\rho_{12}-D|&\leq D^{\frac{99}{100}},
\\
|\rho_{13}-D|&\leq D^{\frac{99}{100}},
\\
\rho_1-D&\geq -D^{\frac{99}{100}},
\\
\rho_2-D&\geq -D^{\frac{99}{100}},
\\
\rho_3-D&\geq -D^{\frac{99}{100}}.
\end{align*}
Therefore, by Lemma \ref{sankakukeilemma} we have
\begin{align}\label{c12c13es1}
c_{12}-\frac{1}{2}\geq -5D^{-\frac{1}{100}},\ c_{13}-\frac{1}{2}\geq -5D^{-\frac{1}{100}}.
\end{align}
Here, we assume that  $\rho_1=\hat{D}$. Then, we have
\begin{align*}
|\rho_{12}-D|\leq D^{\frac{99}{100}},
\\
|\rho_{13}-D|\leq D^{\frac{99}{100}},
\\
|\rho_1-D|\leq D^{\frac{99}{100}},
\end{align*}
and therefore by Lemma \ref{sankakukeilemma}, we obtain
\begin{align*}
u_2\cdot u_{12}-\frac{1}{2}\geq -5D^{-\frac{1}{100}},\ u_1\cdot u_{12}+\frac{1}{2}\geq -5D^{-\frac{1}{100}},
\\
u_3\cdot u_{13}-\frac{1}{2}\geq -5D^{-\frac{1}{100}},\ u_1\cdot u_{13}+\frac{1}{2}\geq -5D^{-\frac{1}{100}}.
\end{align*}
Thus, by \eqref{c12c13es1}, we have
\begin{align}\label{rho1hyouka}
\dot{\rho}_1&\geq 2\mathcal{F}(\rho_1)+\frac{1}{2}\left(\mathcal{F}(\rho_2)+\mathcal{F}(\rho_3)\right)-\frac{1}{2}\left(\mathcal{F}(\rho_{12})+\mathcal{F}(\rho_{13})\right)+o\left(\mathcal{F}(D)\right).
\end{align}
By \eqref{1213avees} and \eqref{rho1hyouka}, we have
\begin{align}\label{est01251}
\dot{\rho}_1-\frac{\dot{\rho}_{12}+\dot{\rho}_{13}}{2}\geq \mathcal{F}(\rho_1)+o\left(\mathcal{F}(D)\right)\gtrsim \mathcal{F}(D).
\end{align}
Next, suppose that $\rho_2=\hat{D}$. We have
\begin{align}\label{naisekihyouka0131}
|u_2+u_3-u_1|^2=3+2c_{23}-2c_{12}-2c_{13}\geq 0.
\end{align}
By \eqref{c12c13es1} and \eqref{naisekihyouka0131}, we have
\begin{align}\label{c23es1}
c_{23}+\frac{1}{2}\gtrsim -D^{-\frac{1}{100}}.
\end{align}
Furthermore, since $\rho_2$ satisfies $|\rho_2-D|\leq D^{\frac{99}{100}}$, we obtain
\begin{align}\label{naisekihyouka2}
u_{21}\cdot u_2+\frac{1}{2}\gtrsim -D^{-\frac{1}{100}}.
\end{align}
Therefore, we have
\begin{align}\label{rho2es01251}
\dot{\rho_2}\geq 2\mathcal{F}(\rho_2)+\frac{1}{2}\mathcal{F}(\rho_1)-\frac{1}{2}\mathcal{F}(\rho_3)-\frac{1}{2}\mathcal{F}(\rho_{12})+o\left(\mathcal{F}(D)\right).
\end{align}
Next, we estimate $\frac{\dot{\rho}_{12}+\dot{\rho}_{13}}{2}$. We first treat the terms involving $\rho_1$.
If $\rho_1$ satisfies $\rho_1-D>D^{\frac{99}{100}}$, then we have $\mathcal{F}(\rho_1)\lesssim D^{-\frac{1}{100}}\mathcal{F}(D)$; hence, together with \eqref{naisekihyouka2} we have
\begin{align*}
\frac{\dot{\rho}_{12}+\dot{\rho}_{13}}{2}\leq -\frac{1}{2}\left(\mathcal{F}(\rho_{12})+\mathcal{F}(\rho_{13})\right)+\frac{1}{4}\mathcal{F}(\rho_2)+\frac{1}{2}\mathcal{F}(\rho_3)+o\left(\mathcal{F}(D)\right).
\end{align*}
On the other hand, if $\rho_1$ satisfies $\rho_1-D\leq D^{\frac{99}{100}}$, then by Lemma \ref{sankakukeilemma} we have
\begin{align*}
u_1\cdot u_{21}-\frac{1}{2}\lesssim D^{-\frac{1}{100}}.
\end{align*}
Therefore we obtain
\begin{align*}
\frac{\dot{\rho}_{12}+\dot{\rho}_{13}}{2}&\leq -\frac{1}{2}\left(\mathcal{F}(\rho_{12})+\mathcal{F}(\rho_{13})\right)+\frac{3}{4}\mathcal{F}(\rho_1)
\\
&\quad+\frac{1}{4}\mathcal{F}(\rho_2)+\frac{1}{2}\mathcal{F}(\rho_3)+o\left(\mathcal{F}(D)\right).
\end{align*}
Thus, in either case, combining the above estimate with \eqref{rho2es01251} yields
\begin{align*}
\dot{\rho}_2-\frac{\dot{\rho}_{12}+\dot{\rho}_{13}}{2}&\geq \frac{7}{4}\mathcal{F}(\rho_2)-\frac{1}{4}\mathcal{F}(\rho_1)+o\left(\mathcal{F}(D)\right)
\\
&\gtrsim \mathcal{F}(D).
\end{align*}
If $\rho_3=\hat{D}$, the same argument as in the case where $\rho_2$ attains the minimum gives
\begin{align*}
\dot{\rho_3}-\frac{\dot{\rho}_{12}+\dot{\rho}_{13}}{2}\gtrsim \mathcal{F}(D).
\end{align*}
Therefore we obtain
\begin{align}\label{chukanhyouka0131}
\dot{\hat{D}}-\frac{\dot{\rho}_{12}+\dot{\rho}_{13}}{2}\gtrsim \mathcal{F}(D).
\end{align}
Integrating \eqref{chukanhyouka0131} on $[T_1,t]$ completes the proof.
\end{proof}
Using this lemma, we next show that the \emph{second} distance is also sufficiently far from $D$.

\begin{lemma}\label{midhenchou}
Let $\vec{u}$ be a (1,3)-sign solitary waves. Furthermore, we assume that we have for $T_1\leq t\leq T_2$
\begin{align}\label{estimate0125}
\tilde{D}(t)-\hat{D}(t)<D^{\frac{1}{4}}(t).
\end{align}
Then, we have for $T_1\leq t\leq T_2$
\begin{align}\label{DA21}
\m{(\rho_{12}(t),\rho_{13}(t),\rho_{23}(t))}\geq \tilde{D}(t)+D^{\frac{1}{3}}(t).
\end{align}
\end{lemma}
\begin{proof}
We assume that there exists $T_1\leq t^{\prime}\leq T_2$ such that 
\begin{align}\label{DA22}
\m{(\rho_{12}(t^{\prime}),\rho_{13}(t^{\prime}),\rho_{23}(t^{\prime}))}<\tilde{D}(t^{\prime})+D^{\frac{1}{3}}(t^{\prime}).
\end{align} 
In this case, by symmetry and Lemma \ref{hitotuhanagai} and continuity, we may assume that
$\rho_{23}>\tilde{D}+D^{\frac{2}{3}}$ holds for $T_1\leq t\leq T_2$. Then, we introduce the following bootstrap estimate
\begin{align}\label{bsass2}
\hat{D}(t)-\frac{\rho_{12}(t)+\rho_{13}(t)}{2}+D^{\frac{1}{3}}(t)\geq 0.
\end{align}
By Lemma \ref{henchoumod}, there exists $T_3$ such that we have for $t^{\prime}\leq t\leq T_3$
\begin{align*}
|D(t)-D(t^{\prime})|\leq 2D^{\frac{2}{5}}(t^{\prime}),
\end{align*}
and we have
\begin{align*}
c_{\bullet} D^{\frac{2}{5}}(t^{\prime})\leq \int_{t^{\prime}}^{T_3}\mathcal{F}(D(s))ds\leq \frac{1}{c_{\bullet}} D^{\frac{2}{5}}(t^{\prime}).
\end{align*}
Then, by Lemma \ref{daijisiki}, \eqref{repcon43}, \eqref{estimate0125}, and \eqref{DA22}, we have for $t^{\prime}\leq t\leq T_3$
\begin{align*}
\tilde{D}(t)-\hat{D}(t)&\lesssim \tilde{D}(t^{\prime})-\hat{D}(t^{\prime})+D^{\frac{2}{5}}(t^{\prime})
\\
&\lesssim D^{\frac{2}{5}}(t),
\\
\m{(\rho_{12}(t),\rho_{13}(t),\rho_{23}(t))}-\tilde{D}(t)&\lesssim \m{(\rho_{12}(t^{\prime}),\rho_{13}(t^{\prime}),\rho_{23}(t^{\prime}))}-\tilde{D}(t^{\prime})+D^{\frac{2}{5}}(t^{\prime})
\\
&\lesssim D^{\frac{2}{5}}(t).
\end{align*}
In particular, we have for $t^{\prime}\leq t\leq T_3$
\begin{align}\label{DA23}
\begin{aligned}
\tilde{D}(t)-\hat{D}(t)&<D^{\frac{1}{2}}(t),
\\
\m{(\rho_{12}(t),\rho_{13}(t),\rho_{23}(t))}&<\tilde{D}(t)+D^{\frac{1}{2}}(t).
\end{aligned}
\end{align}
Furthermore, since we have
\begin{align*}
&\quad \hat{D}(t^{\prime})-\frac{\rho_{12}(t^{\prime})+\rho_{13}(t^{\prime})}{2}+D^{\frac{1}{3}}(t^{\prime})
\\
&\geq \tilde{D}(t^{\prime})-D^{\frac{1}{4}}(t^{\prime})-\frac{\rho_{12}(t^{\prime})+\rho_{13}(t^{\prime})}{2}+D^{\frac{1}{3}}(t^{\prime})
\\
&\geq \frac{D^{\frac{1}{3}}(t^{\prime})}{2}-D^{\frac{1}{4}}(t^{\prime})>0,
\end{align*}
we define $T_4\in [t^{\prime},T_3]$ by
\begin{align}\label{bs1sono2}
T_4=\sup{\{t\in[t^{\prime},T_3]\ \mbox{such that \eqref{bsass2} holds on}\ [t^{\prime},t]\}}.
\end{align}
We assume that $T_4<T_3$. Then, we have 
\begin{align*}
\hat{D}(T_4)-\frac{\rho_{12}(T_4)+\rho_{13}(T_4)}{2}+D^{\frac{1}{3}}(T_4)=0.
\end{align*}
Then, by \eqref{DA23}, we may apply Lemma \ref{henchoumid2}. Therefore, by \eqref{chukanhyouka0131} and $|\frac{d}{dt}D^{\frac{1}{3}}|\lesssim D^{-\frac{2}{3}}\mathcal{F}(D)$, we obtain
\begin{align*}
\frac{d}{dt}\left(\hat{D}-\frac{\rho_{12}+\rho_{13}}{2}+D^{\frac{1}{3}}\right)(T_4)>0.
\end{align*}
Therefore, we have $T_4=T_3$. Then, by Lemma \ref{henchoumid2}, there exists $c>0$ such that 
\begin{align}\label{DA25}
\begin{aligned}
&\quad \left(\hat{D}(T_3)-\frac{\rho_{12}(T_3)+\rho_{13}(T_3)}{2}\right)-\left(\hat{D}(t^{\prime})-\frac{\rho_{12}(t^{\prime})+\rho_{13}(t^{\prime})}{2}\right)
\\
&\geq c\int_{t^{\prime}}^{T_3}\mathcal{F}(D(s))ds
\\
&\geq  cc_{\bullet}D^{\frac{2}{5}}(t^{\prime}).
\end{aligned}
\end{align}
Furthermore, by \eqref{estimate0125}, \eqref{DA22}, \eqref{DA25}, and
\begin{align*}
\frac{\rho_{12}(T_3)+\rho_{13}(T_3)}{2}\geq \tilde{D}(T_3),
\end{align*}
we obtain
\begin{align*}
&\quad \hat{D}(T_3)-\tilde{D}(T_3)+D^{\frac{1}{4}}(t^{\prime})
\\
&\geq \hat{D}(T_3)-\tilde{D}(T_3)+\left(\tilde{D}(t^{\prime})-\hat{D}(t^{\prime})\right)
\\
&\geq \left(\hat{D}(T_3)-\frac{\rho_{12}(T_3)+\rho_{13}(T_3)}{2}\right)-\left(\hat{D}(t^{\prime})-\frac{\rho_{12}(t^{\prime})+\rho_{13}(t^{\prime})}{2}\right)-D^{\frac{1}{3}}(t^{\prime})
\\
&\geq cc_{\bullet}D^{\frac{2}{5}}(t^{\prime})-D^{\frac{1}{3}}(t^{\prime}).
\end{align*}
Therefore, we have
\begin{align*}
\tilde{D}(T_3)-\hat{D}(T_3)&\leq-cc_{\bullet}D^{\frac{2}{5}}(t^{\prime})+D^{\frac{1}{3}}(t^{\prime})+D^{\frac{1}{4}}(t^{\prime}),
\end{align*}
which contradicts \eqref{repcon43}. Therefore we complete the proof.
\end{proof}

\subsection{Distance analysis III}

From the discussion so far, we have shown that among $\rho_{12},\ \rho_{23},\ \rho_{13}$, it cannot happen that two or more of them are close to $D$.
Therefore, as a final step, we prove that $\tilde{D}$ is sufficiently larger than $\hat{D}$.
To this end, somewhat ad hoc, we define a constant $c_{\heartsuit}$ as
\begin{align}\label{cha-to}
e^{c_{\heartsuit}}=\frac{4+\sqrt{3}}{4}.
\end{align}
Furthermore, we define $D_{mod}$ as
\begin{align}\label{Dmoddef}
D_{mod}(t)=\min{(\rho_1(t),\rho_2(t),\rho_3(t)+c_{\heartsuit})}.
\end{align}
\begin{remark}
At first glance, the above definition of $c_{\heartsuit}$ may look like an unrelated numerical choice.
However, keeping in mind the algebraic relations satisfied by $c_{12},\ c_{23},\ c_{13}$ and seeking a condition that allows an argument similar to Lemma~\ref{henchousystem} and Lemma~\ref{henchoumid2}, one naturally arrives at the constant $c_{\heartsuit}$.
\end{remark}
\begin{lemma}\label{minhenchou}
Let $\vec{u}$ be a (1,3)-sign solitary waves . Furthermore, we assume that there exist $1\ll T_1< T_2\leq \infty$ such that we have for $T_1\leq t<T_2$
\begin{align}
\tilde{D}(t)-\hat{D}(t)<D^{\frac{1}{4}}(t). \label{lemass01253}
\end{align}
Then, we have for $T_1\leq t<T_2$
\begin{align}\label{Dmodes1}
\left(D_{mod}(t)-\tilde{D}(t)\right)-\left(D_{mod}(T_1)-\tilde{D}(T_1)\right)\gtrsim \int_{T_1}^t \mathcal{F}(D(s))ds.
\end{align}
\end{lemma}
\begin{proof}
By symmetry, we may assume that $\tilde{D}(T_1)=\rho_{12}(T_1)$.
In this case, by Lemma \ref{midhenchou} and continuity, we have for $T_1\leq t<T_2$.
\begin{align*}
\rho_{12}(t)=\tilde{D}(t),\ \rho_{13}(t)>\tilde{D}(t)+D^{\frac{1}{3}}(t),\ \rho_{23}(t)>\tilde{D}(t)+D^{\frac{1}{3}}(t).
\end{align*}
Then we have
\begin{align*}
\dot{\rho_1}&=2\mathcal{F}(\rho_1)+\mathcal{F}(\rho_2)c_{12}+\mathcal{F}(\rho_3)c_{13}+\mathcal{F}(\rho_{12})u_{12}\cdot u_1+o\left(\mathcal{F}(D)\right),
\\
\dot{\rho}_2&=\mathcal{F}(\rho_1)c_{12}+2\mathcal{F}(\rho_2)+\mathcal{F}(\rho_3)c_{23}+\mathcal{F}(\rho_{12})u_{21}\cdot u_2+o\left(\mathcal{F}(D)\right),
\\
\dot{\rho}_3&=\mathcal{F}(\rho_{1})c_{13}+\mathcal{F}(\rho_2)c_{23}+2\mathcal{F}(\rho_3)+o\left(\mathcal{F}(D)\right),
\\
\dot{\rho}_{12}&=-2\mathcal{F}(\rho_{12})+\mathcal{F}(\rho_2)(u_{12}\cdot u_2)+\mathcal{F}(\rho_1)(u_1\cdot u_{21})+o\left(\mathcal{F}(D)\right).
\end{align*}
We note that by definition we have $|D_{mod}-\hat{D}|\leq c_{\heartsuit}$.
First, we assume that $D_{mod}=\rho_1$.
Since we have $\rho_1\leq \rho_3+c_{\heartsuit}$, by \eqref{tukaeruF1} we obtain
\begin{align}\label{rho1rho3hyouka}
\mathcal{F}(\rho_3)-e^{c_{\heartsuit}}\mathcal{F}(\rho_1)\lesssim \frac{\mathcal{F}(D)}{D}.
\end{align}
Since we have for $T_1\leq t<T_2$
\begin{align*}
|\rho_1-D|+|\rho_{12}-D|\leq D^{\frac{99}{100}},
\\
\rho_{2}-D\geq -D^{\frac{99}{100}},
\end{align*}
by Lemma \ref{sankakukeilemma} we obtain
\begin{align}\label{kakudo1261}
c_{12}-\frac{1}{2}\gtrsim -D^{-\frac{1}{100}},\ u_1\cdot u_{21}-\frac{1}{2}\lesssim D^{-\frac{1}{100}}.
\end{align}
Furthermore, if we have $\rho_2>D+D^{\frac{99}{100}}$, then we have
\begin{align*}
\mathcal{F}(\rho_2)\lesssim  \frac{\mathcal{F}(D)}{D},
\end{align*}
and therefore we obtain
\begin{align*}
\dot{\rho}_1-\dot{\rho}_{12}&\geq \frac{3}{2}\mathcal{F}(\rho_1)-\mathcal{F}(\rho_3)+o\left(\mathcal{F}(D)\right).
\end{align*}
On the other hand, if we have $\rho_2<D+D^{\frac{99}{100}}$, then by Lemma \ref{sankakukeilemma} we have
\begin{align*}
\left|c_{12}-\frac{1}{2}\right|+\left|u_{12}\cdot u_2-\frac{1}{2}\right|+\left|u_1\cdot u_{21}-\frac{1}{2}\right|\lesssim D^{-\frac{1}{100}},
\end{align*}
and therefore we obtain
\begin{align*}
\dot{\rho}_1-\dot{\rho}_{12}&\geq \frac{3}{2}\mathcal{F}(\rho_1)-\mathcal{F}(\rho_3)+o\left(\mathcal{F}(D)\right).
\end{align*}
Hence, in either case,
\begin{align}\label{rho1-rho12}
\dot{\rho}_1-\dot{\rho}_{12}&\geq \frac{3}{2}\mathcal{F}(\rho_1)-\mathcal{F}(\rho_3)+o\left(\mathcal{F}(D)\right).
\end{align}
By \eqref{kakudo1261} and \eqref{rho1-rho12}, we have
\begin{align}\label{rho1minnnotoki}
\begin{aligned}
\dot{\rho}_1-\dot{\rho}_{12}&\geq\frac{3}{2}\mathcal{F}(\rho_1)-\frac{4+\sqrt{3}}{4}\mathcal{F}(\rho_1)+o\left(\mathcal{F}(D)\right)
\\
&=\frac{2-\sqrt{3}}{4}\mathcal{F}(D_{mod})+o\left(\mathcal{F}(D)\right)
\\
&\gtrsim \mathcal{F}(D).
\end{aligned}
\end{align}
The case $D_{mod}=\rho_2$ can be treated by the same argument as above. Next, we assume that $D_{mod}=\rho_3+c_{\heartsuit}$. When $\rho_1$ and $\rho_2$ satisfy
\begin{align*}
\rho_1&>D+D^{\frac{99}{100}},
\\
\rho_2&>D+D^{\frac{99}{100}},
\end{align*}
we have
\begin{align*}
\dot{\rho}_3-\dot{\rho}_{12}\geq 2\mathcal{F}(\rho_3)+o\left(\mathcal{F}(D)\right)\gtrsim \mathcal{F}(D).
\end{align*}
Next, if $\rho_1$ and $\rho_2$ satisfy 
\begin{align*}
\rho_1&<D+D^{\frac{99}{100}},
\\
\rho_2&>D+D^{\frac{99}{100}},
\end{align*}
by Lemma \ref{sankakukeilemma}, we have
\begin{align*}
u_1\cdot u_{21}-\frac{1}{2}\lesssim D^{-\frac{1}{100}}.
\end{align*}
Therefore we obtain
\begin{align*}
\dot{\rho}_3-\dot{\rho}_{12}&\geq 2\mathcal{F}(\rho_3)-\frac{3}{2}\mathcal{F}(\min{(\rho_1,\rho_2)})+o\left(\mathcal{F}(D)\right)\gtrsim \mathcal{F}(D).
\end{align*}
When $\rho_1, \rho_2$ satisfy $\rho_1>D+D^{\frac{99}{100}}$ and $\rho_2<D+D^{\frac{99}{100}}$, the same argument yields the above inequality.

Last, when $\rho_1$ and $\rho_2$ satisfy
\begin{align*}
\rho_1&<D+D^{\frac{99}{100}},
\\
\rho_2&<D+D^{\frac{99}{100}},
\end{align*}
by Lemma \ref{sankakukeilemma} we have
\begin{align}\label{DA31}
\left|c_{12}-\frac{1}{2}\right|+\left|u_{12}\cdot u_2-\frac{1}{2}\right|+\left|u_1\cdot u_{21}-\frac{1}{2}\right|\lesssim D^{-\frac{1}{100}}.
\end{align}
By \eqref{Apositive} and \eqref{DA31}, we have
\begin{align*}
\frac{3}{4}+c_{13}c_{23}-c_{13}^2-c_{23}^2\gtrsim -D^{-\frac{1}{100}}.
\end{align*}
Furthermore, since we have
\begin{align*}
c_{13}^2+c_{23}^2-c_{13}c_{23}=\frac{1}{4}\left(c_{13}+c_{23}\right)^2+\frac{3}{4}\left(c_{13}-c_{23}\right)^2\geq \frac{1}{4}\left(c_{13}+c_{23}\right)^2,
\end{align*}
we obtain
\begin{align}\label{c13c23es1}
c_{13}+c_{23}+\sqrt{3}\gtrsim -D^{-\frac{1}{100}}.
\end{align}
We now combine the arguments above. First, if $c_{13}, c_{23}$ satisfy $c_{13}\geq 0$ or $c_{23}\geq 0$, then we have
\begin{align*}
\dot{\rho}_3-\dot{\rho}_{12}&\geq 2\mathcal{F}(\rho_3)-2\mathcal{F}(\min{(\rho_1,\rho_2)})+o\left(\mathcal{F}(D)\right)
\\
&\gtrsim \mathcal{F}(D).
\end{align*}
If $c_{13}, c_{23}$ satisfy $c_{13}<0$ and $c_{23}<0$, then by \eqref{c13c23es1} we have
\begin{align}\label{rho3-rho12}
\begin{aligned}
\dot{\rho}_3-\dot{\rho}_{12}&\geq 2\mathcal{F}(\rho_3)+(c_{13}+c_{23}-1)\mathcal{F}(\min{(\rho_1,\rho_2)})+o\left(\mathcal{F}(D)\right)
\\
&\geq 2\mathcal{F}(\rho_3)-(\sqrt{3}+1)\mathcal{F}(\min{(\rho_1,\rho_2)})+o\left(\mathcal{F}(D)\right).
\end{aligned}
\end{align}
Furthermore, since $\rho_3+c_{\heartsuit}\leq \min{(\rho_1,\rho_2)}$ holds, we obtain
\begin{align}\label{minrho1rho2es}
\mathcal{F}(\min{(\rho_1,\rho_2)})\leq \frac{4}{4+\sqrt{3}}\mathcal{F}(\rho_3)+o\left(\mathcal{F}(D)\right).
\end{align}
By \eqref{rho3-rho12} and \eqref{minrho1rho2es}, we have
\begin{align*}
\dot{\rho}_3-\dot{\rho}_{12}&\geq \left(2-\frac{4(\sqrt{3}+1)}{4+\sqrt{3}}\right)\mathcal{F}(\rho_3)+o\left(\mathcal{F}(D)\right)
\\
&=\frac{4-2\sqrt{3}}{4+\sqrt{3}}\mathcal{F}(\rho_3)+o\left(\mathcal{F}(D)\right)
\\
&\gtrsim \mathcal{F}(D).
\end{align*}
From the above, we have for $T_1\leq t<T_2$,
\begin{align}\label{Dmodes}
\frac{d}{dt}D_{mod}-\dot{\rho}_{12}\gtrsim \mathcal{F}(D).
\end{align}
Hence, integrating \eqref{Dmodes} on $[T_1,t]$ yields \eqref{Dmodes1} and completes the proof.
\end{proof}

Using the above lemma, we now prove Proposition \ref{tildeDnagai}.

\begin{proof}[Proof of Proposition \ref{tildeDnagai}]
We assume that there exists $1\ll t^{\prime}$ such that 
\begin{align}\label{DA334}
\tilde{D}(t^{\prime})-\hat{D}(t^{\prime})\leq D^{\frac{1}{5}}(t^{\prime}).
\end{align} 
Then, by Lemma \ref{henchoumod}, there exists $T_1$ such that we have for $t^{\prime}\leq t\leq T_1$
\begin{align*}
&\left|D(t)-D(t^{\prime})\right|\leq 2D^{\frac{2}{9}}(t^{\prime}),
\\
c_{\bullet}D^{\frac{2}{9}}(t^{\prime})\leq &\int_{t^{\prime}}^{T_1}\mathcal{F}(D(s))ds\leq \frac{1}{c_{\bullet}}D^{\frac{2}{9}}(t^{\prime}).
\end{align*}
Then, by Lemma \ref{daijisiki}, we have for $t^{\prime}\leq t\leq T_1$
\begin{align*}
\tilde{D}(t)-\hat{D}(t)&\lesssim \tilde{D}(t^{\prime})-\hat{D}(t^{\prime})+D^{\frac{2}{9}}(t)
\\
&\lesssim D^{\frac{2}{9}}(t^{\prime}),
\end{align*}
and therefore we have for $t^{\prime}\leq t\leq T_1$
\begin{align*}
\tilde{D}(t)-\hat{D}(t)<D^{\frac{1}{4}}(t).
\end{align*}
Thus, by Lemma \ref{minhenchou} and \eqref{Dmoddef} and \eqref{DA334}, there exists $c>0$ such that 
\begin{align*}
&\quad \left(\hat{D}(T_1)-\tilde{D}(T_1)\right)+D^{\frac{1}{5}}(t^{\prime})+2c_{\heartsuit}
\\
&\geq \left(D_{mod}(T_1)-\tilde{D}(T_1)\right)-\left(D_{mod}(t^{\prime})-\tilde{D}(t^{\prime})\right)
\\
&\geq c\int_{t^{\prime}}^{T_1}\mathcal{F}(D(s))ds
\\
&\geq cc_{\bullet}D^{\frac{2}{9}}(t^{\prime}).
\end{align*}
By the above estimate, we obtain
\begin{align*}
\tilde{D}(T_1)-\hat{D}(T_1)\leq -cc_{\bullet}D^{\frac{2}{9}}(t^{\prime})+D^{\frac{1}{5}}(t^{\prime})+2c_{\heartsuit},
\end{align*}
which contradicts \eqref{repcon43}. Therefore we complete the proof.
\end{proof}

\section{A rough equilateral-triangle decomposition}
By Proposition \ref{tildeDnagai}, the contributions coming from $\rho_{12},\ \rho_{23},\ \rho_{13}$ are negligible compared with $\mathcal{F}(D)$.
In particular, this estimate implies that \eqref{Vrep} holds, and hence we can simplify the leading term in Corollary \ref{corollary}.
%\begin{lemma}\label{section5solitondynamics}
%Let $\vec{u}$ be a (1,3)-sign solitary waves. Then $z_0,z_1,z_2,z_3$ satisfy for $t\gg 1$
%\begin{align}\label{(1,3)zdynamicsver2}
%\begin{aligned}
%\dot{z}_0&=-\mathcal{F}(\rho_1)u_1-\mathcal{F}(\rho_2)u_2-\mathcal{F}(\rho_3)u_3+O\left(\frac{\mathcal{F}(D)}{D^8}\right),
%\\
%\dot{z}_1&=\mathcal{F}(\rho_1)u_1+O\left(\frac{\mathcal{F}(D)}{D^8}\right),
%\\
%\dot{z}_2&=\mathcal{F}(\rho_2)u_2+O\left(\frac{\mathcal{F}(D)}{D^8}\right),
%\\
%\dot{z}_3&=\mathcal{F}(\rho_3)u_3+O\left(\frac{\mathcal{F}(D)}{D^8}\right).
%\end{aligned}
%\end{align}
%\end{lemma}
%\begin{proof}
%By Proposition \ref{tildeDnagai}, we have
%\begin{align*}
%\mathcal{F}(\tilde{D})\lesssim e^{-D^{\frac{1}{5}}}\mathcal{F}(D)\lesssim \frac{\mathcal{F}(D)}{D^8}.
%\end{align*}
%In particular, \eqref{Vrep} holds true. Thus, by Lemma \ref{epzes}, we obtain \eqref{(1,3)zdynamicsver2}.
%\end{proof}
%In particular, using the above estimate and repeating the same computation as in Corollary \ref{corollary}, we also obtain the following lemma.

\begin{lemma}\label{daijisiki2}
Let $\vec{u}$ be a (1,3)-sign solitary waves. Then the following hold:
\begin{enumerate}
\item  $z_0,z_1,z_2,z_3$ satisfy for $t\gg 1$
\begin{align}\label{(1,3)zdynamicsver2}
\begin{aligned}
\dot{z}_0&=-\mathcal{F}(\rho_1)u_1-\mathcal{F}(\rho_2)u_2-\mathcal{F}(\rho_3)u_3+O\left(\frac{\mathcal{F}(D)}{D^6}\right),
\\
\dot{z}_1&=\mathcal{F}(\rho_1)u_1+O\left(\frac{\mathcal{F}(D)}{D^6}\right),
\\
\dot{z}_2&=\mathcal{F}(\rho_2)u_2+O\left(\frac{\mathcal{F}(D)}{D^6}\right),
\\
\dot{z}_3&=\mathcal{F}(\rho_3)u_3+O\left(\frac{\mathcal{F}(D)}{D^6}\right).
\end{aligned}
\end{align}

\item $Z_1,Z_2,Z_3$ satisfy for $t\gg 1$,
\begin{align}\label{Zesver2}
\begin{aligned}
\dot{Z}_1&=2\mathcal{F}(\rho_1)u_1+\mathcal{F}(\rho_2)u_2+\mathcal{F}(\rho_3)u_3+O\left(\frac{\mathcal{F}(D)}{D^6}\right),
\\
\dot{Z}_2&=\mathcal{F}(\rho_1)u_1+2\mathcal{F}(\rho_2)u_2+\mathcal{F}(\rho_3)u_3+O\left(\frac{\mathcal{F}(D)}{D^6}\right),
\\
\dot{Z}_3&=\mathcal{F}(\rho_1)u_1+\mathcal{F}(\rho_2)u_2+2\mathcal{F}(\rho_3)u_3+O\left(\frac{\mathcal{F}(D)}{D^6}\right).
\end{aligned}
\end{align}

\item $\rho_1,\ \rho_2,\ \rho_3$ satisfy for $t\gg 1$
\begin{align}\label{rhoesver2}
\begin{aligned}
\dot{\rho_1}&=2\mathcal{F}(\rho_1)+\mathcal{F}(\rho_2)c_{12}+\mathcal{F}(\rho_3)c_{13}+O\left(\frac{\mathcal{F}(D)}{D^6}\right),
\\
\dot{\rho}_2&=\mathcal{F}(\rho_1)c_{12}+2\mathcal{F}(\rho_2)+\mathcal{F}(\rho_3)c_{23}+O\left(\frac{\mathcal{F}(D)}{D^6}\right),
\\
\dot{\rho}_3&=\mathcal{F}(\rho_{1})c_{13}+\mathcal{F}(\rho_2)c_{23}+2\mathcal{F}(\rho_3)+O\left(\frac{\mathcal{F}(D)}{D^6}\right).
\end{aligned}
\end{align}

\item $c_{12},c_{13},c_{23}$ satisfy for $t\gg 1$
\begin{align}\label{cesver2}
\begin{aligned}
\dot{c}_{12}&=(1-c_{12}^2)\left(\frac{\mathcal{F}(\rho_1)}{\rho_2}+\frac{\mathcal{F}(\rho_2)}{\rho_1}\right)+\mathcal{F}(\rho_3)\left( \frac{c_{23}-c_{13}c_{12}}{\rho_1}+\frac{c_{13}-c_{23}c_{12}}{\rho_2}\right)
\\
&\quad+O\left(\frac{\mathcal{F}(D)}{D^7}\right),
\\
\dot{c}_{13}&=(1-c_{13}^2)\left(\frac{\mathcal{F}(\rho_1)}{\rho_3}+\frac{\mathcal{F}(\rho_3)}{\rho_1}\right)+\mathcal{F}(\rho_2)\left( \frac{c_{23}-c_{12}c_{13}}{\rho_1}+\frac{c_{12}-c_{13}c_{23}}{\rho_3}\right)
\\
&\quad+O\left(\frac{\mathcal{F}(D)}{D^7}\right),
\\
\dot{c}_{23}&=(1-c_{23}^2)\left(\frac{\mathcal{F}(\rho_2)}{\rho_3}+\frac{\mathcal{F}(\rho_3)}{\rho_2}\right)+\mathcal{F}(\rho_1)\left( \frac{c_{13}-c_{12}c_{23}}{\rho_2}+\frac{c_{12}-c_{13}c_{23}}{\rho_3}\right)
\\
&\quad+O\left(\frac{\mathcal{F}(D)}{D^7}\right).
\end{aligned}
\end{align}
\end{enumerate}
\end{lemma}
\begin{proof}
By Proposition \ref{tildeDnagai}, we have
\begin{align*}
\mathcal{F}(\tilde{D})\lesssim e^{-D^{\frac{1}{5}}}\mathcal{F}(D)\lesssim \frac{\mathcal{F}(D)}{D^6}.
\end{align*}
In particular, \eqref{Vrep} holds true. Thus, by Lemma \ref{epzes} and Corollary \ref{corollary}, we complete the proof.
\end{proof}
Using these equations, we analyze the long-time behavior of $(1,3)$-sign solitary waves solutions.

\subsection{Estimates for $D$}

First, we give bounds on $D$.

\begin{lemma}\label{Dlogtorder}
Let $\vec{u}$ be a (1,3)-sign solitary waves. Then, we have for $t\gg 1$
\begin{align}
\left|D(t)-\log{t}+\frac{d-1}{2}\log{(\log{t})}\right|&\lesssim 1, \label{Dsimes1}
\\
\mathcal{F}(D)\sim \frac{1}{t}. \label{mathcalFsimes1}
\end{align}
\end{lemma}
\begin{proof}
By Lemma \ref{limeplem}, we have \eqref{Dupper} and \eqref{F(D)lower}.\ Thus, we prove
\begin{align}
\log{t}-\frac{d-1}{2}\log{(\log{t})}-D(t)\lesssim 1, \label{Dlower}
\\
\mathcal{F}(D(t))\lesssim \frac{1}{t}. \label{F(D)upper}
\end{align}
To prove these, we establish that for $t\gg 1$,
\begin{align}\label{Dodees}
\dot{D}\geq \frac{1}{1000}\mathcal{F}(D).
\end{align}
Assume that there exists $t^{\prime}\gg 1$ such that $D$ satisfies
\begin{align}\label{Dodecontra}
\dot{D}(t^{\prime})< \frac{1}{1000}\mathcal{F}(D(t^{\prime})).
\end{align}
By symmetry, we may assume that $\rho_1=D$. If $\rho_2>D+D^{\frac{1}{2}}$ or $\rho_3>D+D^{\frac{1}{2}}$ holds, then
\begin{align*}
\dot{\rho}_1\geq \mathcal{F}(D)+o\left(\mathcal{F}(D)\right),
\end{align*}
and hence $\rho_2\leq D+D^{\frac{1}{2}}$ and $\rho_3\leq D+D^{\frac{1}{2}}$ must hold. Next, if $c_{12}\geq -\frac{99}{100}$ or $c_{13}\geq -\frac{99}{100}$, then
\begin{align*}
\dot{\rho}_1\geq \frac{1}{100}\mathcal{F}(D)+o\left(\mathcal{F}(D)\right),
\end{align*}
which contradicts \eqref{Dodecontra}. Therefore $c_{12},\ c_{13}$ satisfy
\begin{align*}
-1\leq c_{12}(t^{\prime})\leq -\frac{99}{100},\ -1\leq c_{13}(t^{\prime})\leq -\frac{99}{100}.
\end{align*}
By \eqref{Apositive}, we have $c_{23}(t^{\prime})\geq 0$, and we obtain
\begin{align*}
0&\leq 1+2c_{12}c_{13}c_{23}-c_{12}^2-c_{13}^2-c_{23}^2
\\
&\leq 1+2c_{23}-\frac{19}{10}-c_{23}^2
\\
&\leq -(c_{23}-1)^2+\frac{1}{10}.
\end{align*}
In particular, this inequality implies that $c_{23}\geq \frac{2}{3}$. Then, since we have
\begin{align*}
D\leq \rho_2&\leq D+,D^{\frac{1}{2}},
\\
D\leq \rho_3&\leq D+D^{\frac{1}{2}},
\end{align*}
$\rho_{23}$ satisfies
\begin{align*}
\rho_{23}(t^{\prime})<\frac{9}{10}D(t^{\prime}),
\end{align*}
which contradicts Proposition \ref{tildeDnagai}. Therefore, we obtain \eqref{Dodees}. Then, by \eqref{Dodees}, we obtain  \eqref{Dsimes1} and \eqref{mathcalFsimes1}, and we complete the proof.
\end{proof}
Since we have now determined the size of $D$, we introduce a reference function $L:[0,\infty)\to\mathbb{R}$ by setting $L(0)=1$ and defining it as the solution to
\begin{align}\label{Ldef}
\dot{L}=\mathcal{F}(L).
\end{align}
Then, by direct calculation, there exists $c_{\star}$ depending only on $d,\ \alpha,\ p$ such that for $t\gg 1$
\begin{align}\label{Lesti}
\left|L(t)-\log{t}+\frac{d-1}{2}\log{(\log{t})}-c_{\star}\right|\lesssim \frac{\log{(\log{t})}}{\log{t}}.
\end{align}
Furthermore, by Lemma \ref{Dlogtorder}, $D$ and $L$ satisfy
\begin{align}
\left|L(t)-D(t)\right|\lesssim 1, \label{LDes}
\\
\mathcal{F}(L(t))\sim \mathcal{F}(D(t))\sim \frac{1}{t}. \label{mathcalFLDes}
\end{align}

\subsection{A Lyapunov-function argument}

In this subsection, we use a Lyapunov function to show that, as $t\to\infty$, the quantities $\rho_1,\rho_2,\rho_3$ converge to asymptotic values determined by $c_{12},\ c_{23},\ c_{13}$.
As a preliminary step, we define $b_1,b_2,b_3$ as
\begin{align}\label{bdef}
\begin{aligned}
b_1&=\frac{(2-c_{23})(2+c_{23}-c_{12}-c_{13})}{\mathcal{D}},
\\
b_2&=\frac{(2-c_{13})(2+c_{13}-c_{12}-c_{23})}{\mathcal{D}},
\\
b_3&=\frac{(2-c_{12})(2+c_{12}-c_{23}-c_{13})}{\mathcal{D}},
\end{aligned}
\end{align}
where 
\begin{align}\label{mathcalDdef}
\mathcal{D}(t)=2(4+c_{12}c_{13}c_{23}-c_{12}^2-c_{23}^2-c_{13}^2).
\end{align}
\begin{remark}
The quantities $b_k$ correspond to $e^{-a_k}$ in the following sense:
if we set $a_k$ so that $a_k$ represents the difference between $\rho_k$ and $L$,
and we choose $a_k$ so that the leading part of $\dot{\rho}_k$ vanishes, then $e^{-a_k}$ solves the linear system
\begin{align*}
\begin{pmatrix}
2&c_{12}&c_{13}
\\
c_{12}&2&c_{23}
\\
c_{13}&c_{23}&2
\end{pmatrix}
\begin{pmatrix}
e^{-a_1}
\\
e^{-a_2}
\\
e^{-a_3}
\end{pmatrix}
=
\begin{pmatrix}
1
\\
1
\\
1
\end{pmatrix}
.
\end{align*}
\end{remark}
Furthermore, we define $\mathcal{C}$ as
\begin{align}\label{mathcalCdef}
\mathcal{C}=
\begin{pmatrix}
2&c_{12}&c_{13}
\\
c_{12}&2&c_{23}
\\
c_{13}&c_{23}&2
\end{pmatrix}
.
\end{align}
As a basic property, we have the following.
\begin{lemma}\label{mathcalCpro1}
For any $x\in\mathbb{R}^3$, we have
\begin{align}\label{Cuku}
x\cdot \mathcal{C}x\geq |x|^2.
\end{align}
\end{lemma}
\begin{proof}
We define $x=(x_1,x_2,x_3)$. Then, we have
\begin{align*}
x\cdot \mathcal{C}x=|x|^2+|x_1u_1+x_2u_2+x_3u_3|^2\geq |x|^2,
\end{align*}
and we obtain \eqref{Cuku}.
\end{proof}
We also have the following lemma.
\begin{lemma}\label{bDes1}
We have for $t\geq 0$
\begin{align}
4\leq \mathcal{D}(t)\leq10, \label{mathcalDes1}
\\
\frac{1}{20}\leq b_1(t)\leq\frac{15}{4}, \label{b1es1}
\\
\frac{1}{20}\leq b_2(t)\leq\frac{15}{4}, \label{b2es1}
\\
\frac{1}{20}\leq b_3(t)\leq\frac{15}{4}. \label{b3es1}
\end{align}
\end{lemma}
\begin{proof}
By \eqref{mathcalDdef} and $-1\leq c_{12},c_{13},c_{23}\leq 1$, we obtain 
\begin{align*}
\mathcal{D}(t)\leq 2(4+1)=10.
\end{align*}
Furthermore, by \eqref{Adef} and \eqref{Apositive}, we have
\begin{align*}
\mathcal{D}=7+A-c_{12}^2-c_{13}^2-c_{23}^2\geq 4.
\end{align*}
Therefore, we obtain \eqref{mathcalDes1}. Next, we prove \eqref{b1es1}. Since we have
\begin{align*}
0\leq |u_1-u_2-u_3|^2=3-2c_{12}-2c_{13}+2c_{23},
\end{align*}
we obtain
\begin{align}\label{naisekies126}
2+c_{23}-c_{12}-c_{13}\geq \frac{1}{2}.
\end{align}
By \eqref{mathcalDes1} and \eqref{naisekies126}, we have
\begin{align*}
b_1(t)\geq \frac{1}{10}(2-1)\frac{1}{2}=\frac{1}{20}.
\end{align*}
On the other hand, since $c_{12},\ c_{23},\ c_{13}$ are bounded, \eqref{mathcalDes1} yields
\begin{align*}
b_1(t)\leq \frac{1}{4}(2+1)(2+1+1+1)=\frac{15}{4},
\end{align*}
and we obtain \eqref{b1es1}. The bounds \eqref{b2es1} and \eqref{b3es1} follow by the same argument, which completes the proof.
\end{proof}
Moreover, from the form of the expressions and \eqref{mathcalDes1}, we have
\begin{align}\label{bkes2}
|\dot{b}_1|+|\dot{b}_2|+|\dot{b}_3|\lesssim \frac{\mathcal{F}(D)}{D}\sim \frac{\mathcal{F}(L)}{L}.
\end{align}
Here, we define $\tilde{a}_1,\tilde{a}_2, \tilde{a}_3$ as 
\begin{align}\label{tildeadef}
e^{-\tilde{a}_k}=b_k.
\end{align}
By definition, $\tilde{a}_k$ are bounded, and moreover for $t$ sufficiently large,
\begin{align}\label{tildeaes1}
|\dot{\tilde{a}}_1|+|\dot{\tilde{a}}_2|+|\dot{\tilde{a}}_3|\lesssim \frac{\mathcal{F}(D)}{D}\sim \frac{\mathcal{F}(L)}{L}.
\end{align}
Using these, we study the evolution of the lengths $\rho_1,\rho_2,\rho_3$. For $k=1,2,3$, we define $a_k$ by
\begin{align}\label{akdef}
a_k=\rho_k-L.
\end{align}
Note that by definition, for $k=1,2,3$ $a_k$ is bounded. Then the following lemma holds.

\begin{lemma}\label{aklemma}
Let $\vec{u}$ be a (1,3)-sign solitary waves. Then we have for $k=1,2,3$ and $t\gg 1$
\begin{align}\label{aes2}
\left|a_k(t)-\tilde{a}_k(t)\right|\lesssim \frac{1}{(\log{t})^{\frac{1}{2}}}.
\end{align}
\end{lemma}
\begin{proof}
For $k=1,2,3$, we define $\zeta_k$ as
\begin{align}\label{zetadef}
\zeta_k(t)=a_k(t)-\tilde{a}_k(t).
\end{align}
Note that $\zeta_k$ are bounded, and moreover we have
\begin{align*}
\rho_k=L+\tilde{a}_k+\zeta_k.
\end{align*}
By the above estimate and \eqref{tukaeruF1}, we have for $k=1,2,3$ and $t\gg 1$
\begin{align}\label{rhokes126}
\left|\mathcal{F}(\rho_k)-\mathcal{F}(L)e^{-\tilde{a}_k-\zeta_k}\right|\lesssim \frac{\mathcal{F}(L)}{L}.
\end{align}
By \eqref{tildeaes1}, \eqref{rhokes126}, and
\begin{align*}
\begin{pmatrix}
2&c_{12}&c_{13}
\\
c_{12}&2&c_{23}
\\
c_{13}&c_{23}&2
\end{pmatrix}
\begin{pmatrix}
b_1
\\
b_2
\\
b_3
\end{pmatrix}
=
\begin{pmatrix}
1
\\
1
\\
1
\end{pmatrix}
,
\end{align*}
$\zeta_1$ satisfies
\begin{align}\label{zeta1es}
\begin{aligned}
\dot{\zeta}_1&=\dot{\rho}_1-\dot{L}-\dot{\tilde{a}}_1
\\
&=2\mathcal{F}(\rho_1)+\mathcal{F}(\rho_2)c_{12}+\mathcal{F}(\rho_3)c_{13}-\mathcal{F}(L)+O\left(\frac{1}{t\log{t}}\right)
\\
&=\mathcal{F}(L)\left(2e^{-\tilde{a}_1-\zeta_1}+c_{12}e^{-\tilde{a}_2-\zeta_2}+c_{13}e^{-\tilde{a}_3-\zeta_3}-1\right)+O\left(\frac{1}{t\log{t}}\right)
\\
&=\mathcal{F}(L)\left\{2b_1(e^{-\zeta_1}-1)+c_{12}b_2(e^{-\zeta_2}-1)+c_{13}b_3(e^{-\zeta_3}-1)\right\}+O\left(\frac{1}{t\log{t}}\right).
\end{aligned}
\end{align}
A similar computation yields, for $\zeta_2$ and $\zeta_3$,
\begin{align}
\dot{\zeta}_2&=\mathcal{F}(L)\left\{c_{12}b_1(e^{-\zeta_1}-1)+2b_2(e^{-\zeta_2}-1)+c_{23}b_3(e^{-\zeta_3}-1) \right\}+O\left(\frac{1}{t\log{t}}\right), \label{zeta2es}
\\
\dot{\zeta}_3&=\mathcal{F}(L)\left\{c_{13}b_1(e^{-\zeta_1}-1)+c_{23}b_2(e^{-\zeta_2}-1)+2b_3(e^{-\zeta_3}-1) \right\}+O\left(\frac{1}{t\log{t}}\right). \label{zeta3es}
\end{align}
Here, we define $\mathcal{L}$ as
\begin{align}\label{mathcalLdef}
\mathcal{L}(t)=\sum_{k=1}^3 b_k(t)\left(e^{-\zeta_k(t)}+\zeta_k(t)-1\right).
\end{align}
Since $\zeta_k$ are bounded and \eqref{b1es1}, \eqref{b2es1}, \eqref{b3es1} hold, we have
\begin{align}\label{Lsim}
\mathcal{L}(t)\sim |\zeta_1(t)|^2+|\zeta_2(t)|^2+|\zeta_3(t)|^2.
\end{align}
Furthermore, by \eqref{bkes2}, we have
\begin{align}\label{Lbibunes}
\dot{\mathcal{L}}&=\sum_{k=1}^3 b_k\dot{\zeta}_k(1-e^{-\zeta_k})+O\left(\frac{1}{t\log{t}}\right).
\end{align}
We define $v$ as
\begin{align*}
v=
\begin{pmatrix}
v_1
\\
v_2
\\
v_3
\end{pmatrix}
=
\begin{pmatrix}
b_1&0&0
\\
0&b_2&0
\\
0&0&b_3
\end{pmatrix}
\begin{pmatrix}
e^{-\zeta_1}-1
\\
e^{-\zeta_2}-1
\\
e^{-\zeta_3}-1
\end{pmatrix}
.
\end{align*}
Then, we have $|v|\sim |\zeta_1|+|\zeta_2|+|\zeta_3|$. By \eqref{zeta1es}, \eqref{zeta2es}, \eqref{zeta3es}, we obtain
\begin{align}\label{Les2}
\dot{\mathcal{L}}=-\mathcal{F}(L)(v\cdot \mathcal{C}v)+O\left(\frac{1}{t\log{t}}\right).
\end{align}
Therefore, by \eqref{Lsim}, \eqref{Les2}, and Lemma \ref{mathcalCpro1}, there exists $c_1>0$ such that
\begin{align*}
\dot{\mathcal{L}}\leq -\frac{c_1}{t}\mathcal{L}+\frac{1}{c_1t\log{t}}.
\end{align*}
By Gronwall's inequality, we have for $t\gg 1$
\begin{align*}
\mathcal{L}\lesssim \frac{1}{\log{t}}.
\end{align*}
Combining this with \eqref{Lsim} yields \eqref{aes2}.
\end{proof}

\subsection{Long-time behavior of the inner products}
By Lemma \ref{aklemma}, we may estimate $\mathcal{F}(\rho_1),\mathcal{F}(\rho_2),\mathcal{F}(\rho_3)$ in terms of formulas involving inner products.
Using these estimates, we refine the relations among $c_{12},\ c_{13},\ c_{23}$.
Before stating the lemma, we define $\mathcal{N}$ as
\begin{align}\label{Ndef}
\begin{aligned}
\mathcal{N}&=8+4c_{12}c_{23}c_{13}-3c_{12}^2-3c_{23}^2-3c_{13}^2
\\
&\quad-c_{12}c_{13}-c_{12}c_{23}-c_{23}c_{13}+2c_{12}+2c_{23}+2c_{13}.
\end{aligned}
\end{align}
\begin{lemma}\label{naisekiyokusuru}
Let $\vec{u}$ be a $(1,3)$-sign solitary waves. Then, we have for $t\gg 1$
\begin{align}\label{codenew}
\begin{aligned}
\dot{c}_{12}&=\frac{\mathcal{F}(L)}{L}\frac{(1-c_{12})(\mathcal{N}-c_{12}^2-c_{23}c_{13}+2c_{12})}{\mathcal{D}}+O\left(\frac{1}{t(\log{t})^{\frac{3}{2}}}\right),
\\
\dot{c}_{13}&=\frac{\mathcal{F}(L)}{L}\frac{(1-c_{13})(\mathcal{N}-c_{13}^2-c_{12}c_{23}+2c_{13})}{\mathcal{D}}+O\left(\frac{1}{t(\log{t})^{\frac{3}{2}}}\right),
\\
\dot{c}_{23}&=\frac{\mathcal{F}(L)}{L}\frac{(1-c_{23})(\mathcal{N}-c_{23}^2-c_{12}c_{13}+2c_{23})}{\mathcal{D}}+O\left(\frac{1}{t(\log{t})^{\frac{3}{2}}}\right).
\end{aligned}
\end{align}
\end{lemma}
\begin{proof}
We recall the definition \eqref{zetadef}. Then, by \eqref{akdef}, we have for $k=1,2,3$
\begin{align*}
\rho_k=L+\tilde{a}_k+\zeta_k=L+a_k.
\end{align*}
Then, by \eqref{tukaeruF1}, \eqref{tukaeruF2}, and \eqref{tildeadef}, we have for $k=1,2,3$
\begin{align}\label{rhoksharpes1}
\begin{aligned}
\left| \mathcal{F}(\rho_k)-b_k\mathcal{F}(L)\right|&\leq\left|\mathcal{F}(L+\tilde{a}_k+\zeta_k)-\mathcal{F}(L+\tilde{a}_k)\right|+\left| \mathcal{F}(L+\tilde{a}_k)-e^{-\tilde{a}_k}\mathcal{F}(L)\right|
\\
&\lesssim \frac{\mathcal{F}(L)}{L^{\frac{1}{2}} }.
\end{aligned}
\end{align}
Moreover, we have for $k=1,2,3$
\begin{align}\label{rhoksharpes2}
\left|\frac{1}{\rho_k}-\frac{1}{L}\right|\lesssim \frac{1}{L^2}.
\end{align}
By \eqref{rhoksharpes1} and \eqref{rhoksharpes2}, we have for $j=1,2,3$, $k=1,2,3$
\begin{align}\label{rhoksharpes3}
\left|\frac{\mathcal{F}(\rho_j)}{\rho_k}-\frac{b_j\mathcal{F}(L)}{L}\right|\lesssim \frac{\mathcal{F}(L)}{L^{\frac{3}{2}}}\sim \frac{1}{t(\log{t})^{\frac{3}{2}}}.
\end{align}
Therefore, by \eqref{cesver2} and \eqref{rhoksharpes3}, we obtain
\begin{align}\label{cjkesrev}
\begin{aligned}
\dot{c}_{12}&=\frac{\mathcal{F}(L)}{L}\left\{ (b_1+b_2)(1-c_{12}^2)+b_3(1-c_{12})(c_{23}+c_{13})\right\}+O\left(\frac{1}{t(\log{t})^{\frac{3}{2}}}\right),
\\
\dot{c}_{23}&=\frac{\mathcal{F}(L)}{L}\left\{ (b_2+b_3)(1-c_{23}^2)+b_1(1-c_{23})(c_{12}+c_{13})\right\}+O\left(\frac{1}{t(\log{t})^{\frac{3}{2}}}\right),
\\
\dot{c}_{13}&=\frac{\mathcal{F}(L)}{L}\left\{ (b_1+b_3)(1-c_{13}^2)+b_2(1-c_{13})(c_{12}+c_{23})\right\}+O\left(\frac{1}{t(\log{t})^{\frac{3}{2}}}\right).
\end{aligned}
\end{align}
By symmetry, it suffices to establish \eqref{codenew} for $c_{12}$. By \eqref{bdef}, we have
\begin{align}\label{turaikeisan0}
\begin{aligned}
&\quad (b_1+b_2)(1-c_{12}^2)+b_3(1-c_{12})(c_{23}+c_{13})
\\
&=(1-c_{12})\left\{(1+c_{12})(b_1+b_2)+b_3(c_{23}+c_{13}) \right\}
\\
&=\frac{1-c_{12}}{\mathcal{D}}\{ (1+c_{12})\left( (2-c_{23})(2+c_{23}-c_{12}-c_{13})+(2-c_{13})(2+c_{13}-c_{12}-c_{23})\right)
\\
&\quad+(2-c_{12})(2+c_{12}-c_{23}-c_{13})(c_{23}+c_{13})  \}.
\end{aligned}
\end{align}
Furthermore, by direct computation, we have
\begin{align*}
(2-c_{23})(2+c_{23}-c_{12}-c_{13})&=4-2c_{12}-2c_{13}-c_{23}^2+c_{12}c_{23}+c_{13}c_{23},
\\
(2-c_{13})(2+c_{13}-c_{12}-c_{23})&=4-2c_{12}-2c_{23}-c_{13}^2+c_{12}c_{13}+c_{13}c_{23},
\end{align*}
and therefore, we obtain
\begin{align*}
&\quad (2-c_{23})(2+c_{23}-c_{12}-c_{13})+(2-c_{13})(2+c_{13}-c_{12}-c_{23})
\\
&=8-4c_{12}-2c_{13}-2c_{23}-c_{13}^2-c_{23}^2+2c_{13}c_{23}+c_{12}c_{23}+c_{12}c_{13}.
\end{align*}
Thus, we obtain
\begin{align}\label{turaikeisan1}
\begin{aligned}
&\quad (1+c_{12})\left( (2-c_{23})(2+c_{23}-c_{12}-c_{13})+(2-c_{13})(2+c_{13}-c_{12}-c_{23})\right)
\\
&=(1+c_{12})(8-4c_{12}-2c_{13}-2c_{23}-c_{13}^2-c_{23}^2+2c_{13}c_{23}+c_{12}c_{23}+c_{12}c_{13})
\\
&=8+4c_{12}-2c_{13}-2c_{23}-4c_{12}^2-c_{13}^2-c_{23}^2-c_{12}c_{13}-c_{12}c_{23}+2c_{13}c_{23}
\\
&\quad+c_{12}^2c_{13}+c_{12}^2c_{23}-c_{12}c_{13}^2-c_{12}c_{23}^2+2c_{12}c_{13}c_{23}.
\end{aligned}
\end{align}
Moreover, we have
\begin{align}\label{turaikeisan2}
\begin{aligned}
&\quad (2-c_{12})(2+c_{12}-c_{23}-c_{13})(c_{13}+c_{23})
\\
&=(4-2c_{23}-2c_{13}-c_{12}^2+c_{12}c_{23}+c_{12}c_{13})(c_{13}+c_{23})
\\
&=4c_{13}+4c_{23}-2c_{13}^2-2c_{23}^2-4c_{13}c_{23}
\\
&\quad-c_{12}^2c_{23}-c_{12}^2c_{13}+c_{12}c_{23}^2+c_{12}c_{13}^2+2c_{12}c_{13}c_{23}.
\end{aligned}
\end{align}
By \eqref{turaikeisan1} and \eqref{turaikeisan2}, we obtain
\begin{align}\label{turaikeisan3}
\begin{aligned}
&\quad  (1+c_{12})\left( (2-c_{23})(2+c_{23}-c_{12}-c_{13})+(2-c_{13})(2+c_{13}-c_{12}-c_{23})\right)
\\
&\quad +(2-c_{12})(2+c_{12}-c_{23}-c_{13})(c_{13}+c_{23})
\\
&=8+4c_{12}+2c_{13}+2c_{23}-4c_{12}^2-3c_{13}^2-3c_{23}^2+4c_{12}c_{13}c_{23}
\\
&\quad-2c_{13}c_{23}-c_{12}c_{13}-c_{12}c_{23}
\\
&=\mathcal{N}-c_{12}^2-c_{13}c_{23}+2c_{12}.
\end{aligned}
\end{align}
By \eqref{cjkesrev}, \eqref{turaikeisan0}, and \eqref{turaikeisan3}, $\dot{c}_{12}$ satisfies \eqref{codenew}.
The same computation for $\dot{c}_{13}$ and $\dot{c}_{23}$ yields the corresponding identities in \eqref{codenew}, and the lemma follows.
\end{proof}

\subsection{Rough estimates for lengths and angles}

In this subsection, using Lemma \ref{naisekiyokusuru}, we prove that $c_{12},\ c_{13},\ c_{23}$ converge to $-\frac{1}{2}$ as $t\to\infty$. To this end, we first control the sum of the other two inner products by the largest one.
\begin{lemma}\label{naisekiest127}
We have for $t\geq 0$
\begin{align}\label{naisekies127}
\begin{aligned}
&\quad \mathcal{N}-\left(\max{(c_{12},c_{13},c_{23})}\right)^2+2\max{(c_{12},c_{13},c_{23})}
\\
&\quad-\m{(c_{12},c_{13},c_{23})}\min{(c_{12},c_{13},c_{23})}
\\
&\geq \frac{2}{3}\max{\left(c_{12}+\frac{1}{2},c_{13}+\frac{1}{2},c_{23}+\frac{1}{2}\right)}.
\end{aligned}
\end{align}
\end{lemma}
\begin{proof}
By symmetry, it suffices to prove \eqref{naisekies127} in the case $c_{23}\leq c_{13}\leq c_{12}$.
First, in the case $c_{23}\leq c_{13}\leq c_{12}$, we show that
\begin{align}\label{maxes1}
-\sqrt{2(c_{12}+1)}\leq c_{13}+c_{23}\leq 2c_{12}.
\end{align}
The right inequality in \eqref{maxes1} is obvious, so it suffices to prove the left one.
When $c_{12}=1$, it is obvious. When $c_{12}<1$, it follows from
\begin{align*}
0\leq A(t)&=1-c_{12}^2-(c_{13}+c_{23})^2+2(c_{12}+1)c_{13}c_{23}
\\
&\leq 1-c_{12}^2-(c_{13}+c_{23})^2+\frac{(c_{12}+1)(c_{13}+c_{23})^2}{2}
\\
&=\frac{(1-c_{12})\{2(1+c_{12})-(c_{13}+c_{23})^2\}}{2}.
\end{align*}
In this proof, we define $x=c_{13}+c_{23}$. Then, we have $-\sqrt{2(c_{12}+1)}\leq x\leq 2c_{12}$, and \eqref{naisekies127} becomes
\begin{align*}
&\quad \mathcal{N}-\left(\max{(c_{12},c_{13},c_{23})}\right)^2+2\max{(c_{12},c_{13},c_{23})}
\\
&\quad-\m{(c_{12},c_{13},c_{23})}\min{(c_{12},c_{13},c_{23})}
\\
&=\mathcal{N}-c_{12}^2-c_{23}c_{13}+2c_{12}
\\
&=2A+6-2c_{12}^2-c_{23}^2-c_{13}^2-c_{12}(c_{23}+c_{13})-2c_{13}c_{23}+4c_{12}+2c_{13}+2c_{23}
\\
&=2A+6-2c_{12}^2-x^2-c_{12}x+4c_{12}+2x
\\
&\geq -x^2+(2-c_{12})x+6-2c_{12}^2+4c_{12}.
\end{align*}
Viewing the last expression as a quadratic function of $x$, we have
\begin{align}\label{nijikansuuimines0}
\begin{aligned}
&\quad  -x^2+(2-c_{12})x+6-2c_{12}^2+4c_{12}
\\
&\geq \min_{x=-\sqrt{2(c_{12}+1)},2c_{12}}\left(-x^2+(2-c_{12})x+6-2c_{12}^2+4c_{12}\right).
\end{aligned}
 \end{align}
When $x=-\sqrt{2(c_{12}+1)}$, setting $\nu_{12}=\sqrt{2(c_{12}+1)}$ so that $c_{12}=\frac{\nu_{12}^2-2}{2}$, we obtain
 \begin{align}\label{chyouka0213}
 \begin{aligned}
&\quad -x^2+(2-c_{12})x+6-2c_{12}^2+4c_{12}
 \\
 &= -2(c_{12}+1)-(2-c_{12})\sqrt{2(c_{12}+1)}+6-2c_{12}^2+4c_{12}
 \\
 &=-\nu_{12}^2-\left(3-\frac{\nu_{12}^2}{2}\right)\nu_{12}+6-\frac{(\nu_{12}^2-2)^2}{2}+2\nu_{12}^2-4
 \\
 &=\frac{\nu_{12}(\nu_{12}-1)(6-\nu_{12}^2)}{2}.
 \end{aligned}
 \end{align}
Here, we estimate $\nu_{12}$. First, we have
 \begin{align}\label{cpositive}
0\leq  \left|u_1+u_2+u_3\right|^2=3+2c_{12}+2c_{13}+2c_{23}.
\end{align}
By \eqref{cpositive} and $c_{12}=\max{(c_{12},c_{13},c_{23})}$, we have $-\frac{1}{2}\leq c_{12}\leq 1$. Therefore, we have $1\leq \nu_{12}\leq 2$, and we obtain
 \begin{align}\label{chyouka0213-2}
 \begin{aligned}
\frac{\nu_{12}(\nu_{12}-1)(6-\nu_{12}^2)}{2}&\geq \nu_{12}-1
 \\
 &=\frac{2c_{12}+1}{\sqrt{2(c_{12}+1)}+1}
 \\
 &\geq \frac{2}{3}\left(c_{12}+\frac{1}{2}\right).
 \end{aligned}
 \end{align}
By \eqref{chyouka0213} and \eqref{chyouka0213-2}, we obtain
 \begin{align}\label{nijikansuumines1}
 \begin{aligned}
&\quad -\left(-\sqrt{2(c_{12}+1)}\right)^2+(2-c_{12})\left(-\sqrt{2(c_{12}+1)}\right)+6-2c_{12}^2+4c_{12}
\\
&\geq \frac{2}{3}\left(c_{12}+\frac{1}{2}\right).
 \end{aligned}
 \end{align}
Next, when $x=2c_{12}$, we have
\begin{align*}
-x^2+(2-c_{12})x+6-2c_{12}^2+4c_{12}&=-8c_{12}^2+8c_{12}+6
 \\
 &=(6-4c_{12})(2c_{12}+1)
 \\
 &\geq c_{12}+\frac{1}{2},
 \end{align*}
 and hence
 \begin{align}\label{nijikansuumines2}
 -\left(2c_{12}\right)^2+(2-c_{12})\left(2c_{12}\right)+6-2c_{12}^2+4c_{12}\geq c_{12}+\frac{1}{2}.
 \end{align}
 By \eqref{nijikansuuimines0}, \eqref{nijikansuumines1}, and \eqref{nijikansuumines2}, we obtain 
 \begin{align*}
\min_{x=-\sqrt{2(c_{12}+1)},2c_{12}}\left(-x^2+(2-c_{12})x+6-2c_{12}^2+4c_{12}\right)\geq \frac{2}{3}\left(c_{12}+\frac{1}{2}\right).
 \end{align*}
Hence \eqref{naisekies127} holds in the case $c_{23}\leq c_{13}\leq c_{12}$, and we complete the proof.
\end{proof}
Using this lemma, we estimate $c_{12},\ c_{13},\ c_{23}$. Before we estimate $c_{12},\ c_{13},\ c_{23}$, we define $\mathfrak{C}$ as
\begin{align}\label{cookisa}
\mathfrak{C}(t)=\left|c_{12}(t)+\frac{1}{2}\right|+\left|c_{23}(t)+\frac{1}{2}\right|+\left|c_{13}(t)+\frac{1}{2}\right|.
\end{align}
\begin{lemma}\label{0213lemma}
We have for $t\geq 0$,
\begin{align}\label{cnorm}
\frac{1}{4}\mathfrak{C}(t)\leq \max{\left(c_{12}(t)+\frac{1}{2},c_{13}(t)+\frac{1}{2},c_{23}(t)+\frac{1}{2}\right)}\leq \mathfrak{C}(t).
\end{align}
\end{lemma}
\begin{proof}
The right-hand inequality in \eqref{cnorm} is obvious. It therefore suffices to prove the left-hand inequality. Since we have
\begin{align*}
0\leq |u_1+u_2+u_3|^2=3+2c_{12}+2c_{13}+2c_{23},
\end{align*}
we have
\begin{align}\label{cwapositive}
\left(c_{12}+\frac{1}{2}\right)+\left(c_{13}+\frac{1}{2}\right)+\left(c_{23}+\frac{1}{2}\right)\geq 0.
\end{align}
By symmetry, we may assume that $c_{23}\leq c_{13}\leq c_{12}$. When $c_{23}\geq -\frac{1}{2}$, we have
\begin{align*}
\mathfrak{C}=\left(c_{12}+\frac{1}{2}\right)+\left(c_{13}+\frac{1}{2}\right)+\left(c_{23}+\frac{1}{2}\right),
\end{align*}
and we obtain
\begin{align*}
\frac{1}{3}\mathfrak{C}\leq \max{\left(c_{12}+\frac{1}{2},c_{13}+\frac{1}{2},c_{23}+\frac{1}{2}\right)}.
\end{align*}
When $c_{23}<-\frac{1}{2}\leq c_{13}$, by \eqref{cwapositive}, we have
\begin{align*}
2\left(c_{12}+\frac{1}{2}\right)&\geq \left(c_{12}+\frac{1}{2}\right)+\left(c_{13}+\frac{1}{2}\right)
\\
&\geq -\left(c_{23}+\frac{1}{2}\right)=\left| c_{23}+\frac{1}{2}\right|. 
\end{align*}
Therefore, we obtain
\begin{align*}
4\left(c_{12}+\frac{1}{2}\right)&\geq \left(c_{12}+\frac{1}{2}\right)+\left(c_{13}+\frac{1}{2}\right)+2\left(c_{12}+\frac{1}{2}\right)\geq \mathfrak{C}.
\end{align*}
When  $c_{23}\leq c_{13}<-\frac{1}{2}\leq c_{12}$, by \eqref{cwapositive}, we have
\begin{align*}
\left(c_{12}+\frac{1}{2}\right)&\geq -\left(c_{13}+\frac{1}{2}\right)-\left(c_{23}+\frac{1}{2}\right)
\\
&=\left|c_{13}+\frac{1}{2}\right|+\left|c_{23}+\frac{1}{2}\right|.
\end{align*}
Thus, we have
\begin{align*}
2\left(c_{12}+\frac{1}{2}\right)\geq \left(c_{12}+\frac{1}{2}\right)+\left|c_{13}+\frac{1}{2}\right|+\left|c_{23}+\frac{1}{2}\right|=\mathfrak{C}.
\end{align*}
Gathering these arguments, we obtain \eqref{cnorm}.
\end{proof}

\begin{lemma}\label{xxx}
Let $\vec{u}$ be a (1,3)-sign solitary waves. Then, we have for $t\gg 1$
\begin{align}\label{naisekrough}
\left|c_{12}(t)+\frac{1}{2}\right|+\left|c_{23}(t)+\frac{1}{2}\right|+\left|c_{13}(t)+\frac{1}{2}\right|\lesssim \frac{1}{(\log{t})^{\frac{1}{2}}}.
\end{align}
\end{lemma}
\begin{proof}
First, by Lemma \ref{naisekiyokusuru}, there exists $C>0$ such that we have for $t\gg 1$,
\begin{align}\label{naisekibsode}
\begin{aligned}
\left|\dot{c}_{12}-\frac{\mathcal{F}(L)}{L}\frac{(1-c_{12})(\mathcal{N}-c_{12}^2-c_{23}c_{13}+2c_{12})}{\mathcal{D}}\right|&\leq \frac{C\mathcal{F}(L)}{L^{\frac{3}{2}}},
\\
\left|\dot{c}_{23}-\frac{\mathcal{F}(L)}{L}\frac{(1-c_{23})(\mathcal{N}-c_{23}^2-c_{12}c_{13}+2c_{23})}{\mathcal{D}}\right|&\leq \frac{C\mathcal{F}(L)}{L^{\frac{3}{2}}},
\\
\left|\dot{c}_{13}-\frac{\mathcal{F}(L)}{L}\frac{(1-c_{13})(\mathcal{N}-c_{13}^2-c_{12}c_{23}+2c_{13})}{\mathcal{D}}\right|&\leq \frac{C\mathcal{F}(L)}{L^{\frac{3}{2}}}.
\end{aligned}
\end{align}
Moreover, by Lemma \ref{aklemma} we have
\begin{align}\label{nagasayoi}
|\rho_1-D|+|\rho_2-D|+|\rho_3-D|\lesssim 1.
\end{align}
By Proposition \ref{tildeDnagai} and \eqref{nagasayoi}, we have for $t\gg 1$
\begin{align}\label{naisekisankaku}
\max{(c_{12},c_{13},c_{23})}<\frac{1}{3}.
\end{align}
Here, we assume that there exists $t^{\prime}\gg 1$ such that 
\begin{align}\label{naisekiass}
\max{\left(c_{12}(t^{\prime})+\frac{1}{2},c_{13}(t^{\prime})+\frac{1}{2},c_{23}(t^{\prime})+\frac{1}{2}\right)}>\frac{45C}{L^{\frac{1}{2}}(t^{\prime})}.
\end{align}
Then, we introduce the following bootstrap estimate
\begin{align}\label{naisekibs1}
\max{\left(c_{12}(t)+\frac{1}{2},c_{13}(t)+\frac{1}{2},c_{23}(t)+\frac{1}{2}\right)}\geq \frac{45C}{L^{\frac{1}{2}}(t)}.
\end{align}
We define $T_1\in [t^{\prime}, \infty]$ by
\begin{align}\label{bs1sono3}
T_1=\sup{\{t\in[t^{\prime},\infty) \ \mbox{such that \eqref{naisekibs1} holds on} [t^{\prime},t]\}}.
\end{align}
Furthermore, we assume that $T_1<\infty$. Suppose that at $t=T_1$ the maximum is attained by $c_{12}(T_1)$.
Then, by Lemma \ref{naisekiest127}, \eqref{naisekibsode}, and \eqref{naisekisankaku}, at $t=T_1$ we have
\begin{align*}
\dot{c}_{12}&\geq \frac{\mathcal{F}(L)}{L}\frac{(1-c_{12})(\mathcal{N}-c_{12}^2-c_{23}c_{13}+2c_{12})}{\mathcal{D}}- \frac{C\mathcal{F}(L)}{L^{\frac{3}{2}}}
\\
&\geq \frac{\mathcal{F}(L)}{L}\left\{ \frac{1}{10}\left(1-\frac{1}{3}\right)\frac{2}{3}\left(c_{12}+\frac{1}{2}\right)-\frac{C}{L^{\frac{1}{2}}}\right\}
\\
&>0.
\end{align*}
Since $L^{-\frac{1}{2}}$ is decreasing, \eqref{naisekibs1} continues to hold for $0<t-T_1\ll 1$, contradicting the definition of $T_1$. Therefore $T_1=\infty$. In particular, the above computation shows that $\max(c_{12},c_{13},c_{23})$ is monotone increasing, and since it is bounded above by \eqref{naisekisankaku}, it converges as $t\to\infty$ to some $c_{\max}\in\left(-\frac{1}{2},\frac{1}{3}\right]$.
For $t$ sufficiently large, if $c_{12}$ is the maximum, then
\begin{align*}
\dot{c}_{12}\gtrsim \frac{\mathcal{F}(L)}{L}\sim \frac{1}{t\log{t}}.
\end{align*}

From this, taking $T_2\gg 1$ sufficiently large, we obtain
\begin{align*}
1&\gtrsim \max{(c_{12}(e^{T_2}),c_{13}(e^{T_2}),c_{23}(e^{T_2}))}-\max{(c_{12}(T_2),c_{13}(T_2),c_{23}(T_2))}
\\
&\gtrsim \int_{T_2}^{e^{T_2}} \frac{ds}{s\log{s}},
\end{align*}
which is a contradiction. Hence we have for $t\gg 1$
\begin{align}\label{ue0213}
\max{\left(c_{12}(t)+\frac{1}{2},c_{13}(t)+\frac{1}{2},c_{23}(t)+\frac{1}{2}\right)}\lesssim \frac{1}{L^{\frac{1}{2}}(t)}\sim \frac{1}{(\log{t})^{\frac{1}{2}}}.
\end{align}
By Lemma \ref{0213lemma} and \eqref{ue0213}, we obtain \eqref{naisekrough}.
\end{proof}

\section{A sharp equilateral-triangle decomposition}
In this section, we use the results obtained in the previous sections to prove Theorem~\ref{maintheorem}.

\subsection{Sharp estimates for lengths and angles}
In the previous section, we obtained \eqref{naisekrough}. This allows us to refine the estimates for $\rho_1,\rho_2,\rho_3$.
\begin{lemma}\label{rhokrefies}
Let $\vec{u}$ be a (1,3)-sign solitary waves. Then, we have for $t\gg 1$
\begin{align}\label{akrefes}
|a_1|+|a_2|+|a_3|\lesssim \frac{1}{(\log{t})^{\frac{1}{2}}}.
\end{align}
\end{lemma}
\begin{proof}
By \eqref{bdef}, \eqref{tildeadef}, and \eqref{naisekrough}, we have
\begin{align}\label{tildearef}
\left|\tilde{a}_1\right|+\left|\tilde{a}_2\right|+\left|\tilde{a}_3\right|\lesssim \frac{1}{L^{\frac{1}{2}}}.
\end{align}
By \eqref{aklemma} and \eqref{tildearef}, we obtain \eqref{akrefes}.
\end{proof}
To quantify the deviations of $c_{12},c_{13},c_{23}$ from $-\frac{1}{2}$, we define $d_{12},d_{13},d_{23}$ as
\begin{align}\label{djkdef}
\begin{aligned}
d_{12}&=c_{12}+\frac{1}{2},
\\
d_{13}&=c_{13}+\frac{1}{2},
\\
d_{23}&=c_{23}+\frac{1}{2}.
\end{aligned}
\end{align}
With this notation, \eqref{naisekrough} can be rewritten as
\begin{align}\label{djkrough}
|d_{12}|+|d_{13}|+|d_{23}|\lesssim \frac{1}{(\log{t})^{\frac{1}{2}}}.
\end{align}
We first obtain a joint estimate for $a_1,a_2,a_3$ and $d_{12},d_{13},d_{23}$.
\begin{lemma}\label{aknaiseki}
Let $\vec{u}$ be a (1,3)-sign solitary waves. Then, we have for $t\gg 1$
\begin{align}\label{kituiakdjkes}
|a_1|+|a_2|+|a_3|+|d_{12}|+|d_{13}|+|d_{23}|\lesssim \frac{1}{(\log{t})^3}.
\end{align}
\end{lemma}
Before we start the proof, we define
\begin{align*}
|a|&=\sqrt{a_1^2+a_2^2+a_3^2},
\\
|d|&=\sqrt{d_{12}^2+d_{13}^2+d_{23}^2}.
\end{align*}
\begin{proof}
We divide the proof into two steps.

Step 1:  We prove
\begin{align}\label{aknaisekies}
\begin{aligned}
\left|a_1-\frac{1}{5}(4d_{12}+4d_{13}+2d_{23})\right|&\lesssim \frac{|d|}{(\log{t})^{\frac{1}{4}}}+\frac{1}{(\log{t})^3},
\\
\left|a_2-\frac{1}{5}(4d_{12}+2d_{13}+4d_{23})\right|&\lesssim\frac{|d|}{(\log{t})^{\frac{1}{4}}}+\frac{1}{(\log{t})^3},
\\
\left|a_3-\frac{1}{5}(2d_{12}+4d_{13}+4d_{23})\right|&\lesssim \frac{|d|}{(\log{t})^{\frac{1}{4}}}+\frac{1}{(\log{t})^3}.
\end{aligned}
\end{align}

We first refine the estimates for the differential equations satisfied by $a_1,a_2,a_3$. By \eqref{tukaeruF2} and \eqref{akrefes}, we have
\begin{align}\label{akkitu}
\left|\mathcal{F}(\rho_k)-(1-a_k)\mathcal{F}(L)\right|\lesssim \frac{\mathcal{F}(L)|a|}{L^{\frac{1}{2}}}\sim \frac{|a|}{t(\log{t})^{\frac{1}{2}}}.
\end{align}
By \eqref{djkrough} and \eqref{akkitu}, $a_1$ satisfies
\begin{align*}
\dot{a}_1&=\dot{\rho}_1-\dot{L}
\\
&=\left\{2(1-a_1)+c_{12}(1-a_2)+c_{13}(1-a_3)-1 \right\} \mathcal{F}(L)
\\
&\quad+O\left(\frac{|a|}{t(\log{t})^{\frac{1}{2}}}+\frac{1}{t(\log{t})^6}\right),
\\
&=\left(-2a_1+\frac{1}{2}a_2+\frac{1}{2}a_3+d_{12}+d_{13}\right)\mathcal{F}(L)
\\
&\quad+O\left(\frac{|a|}{t(\log{t})^{\frac{1}{2}}}+\frac{1}{t(\log{t})^6}\right).
\end{align*}
By the same calculation, $a_1,a_2,a_3$ satisfy
\begin{align}\label{aodenaiseki}
\begin{aligned}
\dot{a}_1&=\left(-2a_1+\frac{1}{2}a_2+\frac{1}{2}a_3+d_{12}+d_{13}\right)\mathcal{F}(L)
\\
&\quad+O\left(\frac{|a|}{t(\log{t})^{\frac{1}{2}}}+\frac{1}{t(\log{t})^6}\right),
\\
\dot{a}_2&=\left(\frac{1}{2}a_1-2a_2+\frac{1}{2}a_3+d_{12}+d_{23}\right)\mathcal{F}(L)
\\
&\quad+O\left(\frac{|a|}{t(\log{t})^{\frac{1}{2}}}+\frac{1}{t(\log{t})^6}\right),
\\
\dot{a}_3&=\left(\frac{1}{2}a_1+\frac{1}{2}a_2-2a_3+d_{13}+d_{23}\right)\mathcal{F}(L)
\\
&\quad+O\left(\frac{|a|}{t(\log{t})^{\frac{1}{2}}}+\frac{1}{t(\log{t})^6}\right).
\end{aligned}
\end{align}
Next, we derive differential equations for $d_{12},d_{13},d_{23}$ from \eqref{cesver2}. By 
\begin{align*}
\left|\frac{1}{\rho_k}-\frac{1}{L}\right|&\lesssim \frac{|a|}{L^2}
\end{align*}
and \eqref{akkitu}, $d_{12}=c_{12}+\frac{1}{2}$ satisfies
\begin{align*}
\dot{d}_{12}&=\frac{\mathcal{F}(L)}{L}\left\{ (1-c_{12}^2)(2-a_1-a_2)+(1-a_3)(1-c_{12})(c_{13}+c_{23})\right\}
\\
&\quad+O\left(\frac{|a|}{t(\log{t})^{\frac{3}{2}}}+\frac{1}{t(\log{t})^7}\right).
\end{align*}
Furthermore, by direct computation, we have
\begin{align*}
&\quad (1-c_{12}^2)(2-a_1-a_2)+(1-a_3)(1-c_{12})(c_{13}+c_{23})
\\
&=\left(\frac{3}{4}+d_{12}-d_{12}^2\right)(2-a_1-a_2)+(1-a_3)\left(\frac{3}{2}-d_{12}\right)(-1+d_{13}+d_{23})
\\
&=3d_{12}+\frac{3}{2}d_{13}+\frac{3}{2}d_{23}-\frac{3}{4}a_1-\frac{3}{4}a_2+\frac{3}{2}a_3+O(|a|^2+|d|^2).
\end{align*}
From symmetry, we have
\begin{align}\label{djkoderev}
\begin{aligned}
\dot{d}_{12}&=\left( 3d_{12}+\frac{3}{2}d_{13}+\frac{3}{2}d_{23}-\frac{3}{4}a_1-\frac{3}{4}a_2+\frac{3}{2}a_3\right)\frac{\mathcal{F}(L)}{L}
\\
&\quad +O\left(\frac{|a|+|d|}{t(\log{t})^{\frac{3}{2}}}+\frac{1}{t(\log{t})^7}\right),
\\
\dot{d}_{13}&=\left(\frac{3}{2}d_{12}+3d_{13}+\frac{3}{2}d_{23}-\frac{3}{4}a_1+\frac{3}{2}a_2-\frac{3}{4}a_3\right)\frac{\mathcal{F}(L)}{L}
\\
&\quad+O\left(\frac{|a|+|d|}{t(\log{t})^{\frac{3}{2}}}+\frac{1}{t(\log{t})^7}\right),
\\
\dot{d}_{23}&=\left(\frac{3}{2}d_{12}+\frac{3}{2}d_{13}+3d_{23}+\frac{3}{2}a_1-\frac{3}{4}a_2-\frac{3}{4}a_3\right)\frac{\mathcal{F}(L)}{L}
\\
&\quad+O\left(\frac{|a|+|d|}{t(\log{t})^{\frac{3}{2}}}+\frac{1}{t(\log{t})^7}\right).
\end{aligned}
\end{align}
In particular, we have
\begin{align}\label{dbibunes}
|\dot{d}_{12}|+|\dot{d}_{13}|+|\dot{d}_{23}|\lesssim \frac{|a|+|d|}{t\log{t}}+\frac{1}{t(\log{t})^7}.
\end{align}
Here, we define $\xi_1,\xi_2,\xi_3$ as 
\begin{align}\label{xidef}
\begin{aligned}
\xi_1&=a_1-\frac{1}{5}(4d_{12}+4d_{13}+2d_{23}),
\\
\xi_2&=a_2-\frac{1}{5}(4d_{12}+2d_{13}+4d_{23}),
\\
\xi_3&=a_3-\frac{1}{5}(2d_{12}+4d_{13}+4d_{23}),
\\
|\xi|&=\sqrt{\xi_1^2+\xi_2^2+\xi_3^2}.
\end{aligned}
\end{align}
Then, by direct computation, we have
\begin{align}\label{xisim}
\begin{aligned}
|\xi|+|d|&\sim |a|+|d|,
\\
|\xi|^2+|d|^2&\sim |a|^2+|d|^2.
\end{aligned}
\end{align}
By \eqref{aodenaiseki}, \eqref{xidef}, and \eqref{xisim}, we obtain
\begin{align}\label{xiode}
\begin{aligned}
\dot{\xi}_1&=\left(-2\xi_1+\frac{1}{2}\xi_2+\frac{1}{2}\xi_3\right)\mathcal{F}(L)+O\left(\frac{|\xi|+|d|}{t(\log{t})^{\frac{1}{2}}}+\frac{1}{t(\log{t})^6}\right),
\\
\dot{\xi}_2&=\left(\frac{1}{2}\xi_1-2\xi_2+\frac{1}{2}\xi_3\right)\mathcal{F}(L)+O\left(\frac{|\xi|+|d|}{t(\log{t})^{\frac{1}{2}}}+\frac{1}{t(\log{t})^6}\right),
\\
\dot{\xi}_3&=\left(\frac{1}{2}\xi_1+\frac{1}{2}\xi_2-2\xi_3\right)\mathcal{F}(L)+O\left(\frac{|\xi|+|d|}{t(\log{t})^{\frac{1}{2}}}+\frac{1}{t(\log{t})^6}\right).
\end{aligned}
\end{align}
Then, we have
\begin{align*}
\frac{d}{dt}|\xi|^2&=\left(-4\xi_1^2-4\xi_2^2-4\xi_3^2+2\xi_1\xi_2+2\xi_2\xi_3+2\xi_3\xi_1\right)\mathcal{F}(L)
\\
&\quad+O\left(\frac{|\xi|^2+|d|^2}{t(\log{t})^{\frac{1}{2}}}+\frac{1}{t(\log{t})^6}\right).
\end{align*}
Since we have
\begin{align*}
\xi_1^2+\xi_2^2+\xi_3^2\geq \xi_1\xi_2+\xi_2\xi_3+\xi_3\xi_1,
\end{align*}
there exists $c_1>0$ such that we have for $t\gg 1$
\begin{align}\label{xiodees}
\frac{d}{dt}|\xi|^2\leq -\frac{c_1}{t}|\xi|^2+\frac{1}{c_1}\left(\frac{|d|^2}{t(\log{t})^{\frac{1}{2}}}+\frac{1}{t(\log{t})^6}\right).
\end{align}
By Gronwall's inequality and \eqref{dbibunes}, we obtain
\begin{align*}
|\xi|^2\lesssim \frac{|d|^2}{(\log{t})^{\frac{1}{2}}}+\frac{1}{(\log{t})^6},
\end{align*}
which implies \eqref{aknaisekies}.

Step 2: We prove \eqref{kituiakdjkes}.

By \eqref{aknaisekies} and \eqref{djkoderev},  there exists $c_2>0$ such that 
\begin{align}\label{djktotemokitui}
\begin{aligned}
\left|\dot{d}_{12}-\left(\frac{12}{5}d_{12}+\frac{9}{5}d_{13}+\frac{9}{5}d_{23}\right)\frac{\mathcal{F}(L)}{L}\right|&\leq c_2\left(\frac{\mathcal{F}(L)}{L^{\frac{5}{4}}}|d|+\frac{\mathcal{F}(L)}{L^4}\right),
\\
\left|\dot{d}_{13}-\left(\frac{9}{5}d_{12}+\frac{12}{5}d_{13}+\frac{9}{5}d_{23}\right)\frac{\mathcal{F}(L)}{L}\right|&\leq c_2\left(\frac{\mathcal{F}(L)}{L^{\frac{5}{4}}}|d|+\frac{\mathcal{F}(L)}{L^4}\right),
\\
\left|\dot{d}_{23}-\left(\frac{9}{5}d_{12}+\frac{9}{5}d_{13}+\frac{12}{5}d_{23}\right)\frac{\mathcal{F}(L)}{L}\right|&\leq c_2\left(\frac{\mathcal{F}(L)}{L^{\frac{5}{4}}}|d|+\frac{\mathcal{F}(L)}{L^4}\right).
\end{aligned}
\end{align}
Since we have
\begin{align*}
0\leq |u_1+u_2+u_3|^2=3+2c_{12}+2c_{13}+2c_{23},
\end{align*}
we obtain
\begin{align}\label{dwapositive}
d_{12}+d_{13}+d_{23}\geq 0.
\end{align}
Furthermore, by Lemma \ref{0213lemma}, we have
\begin{align}\label{dsim}
\max{(d_{12},d_{13},d_{23})}\sim |d|.
\end{align}
%Then, by considering the possible sign patterns among $d_{12},d_{13},d_{23}$ (i.e., the numbers of positive and negative entries), we obtain
%\begin{align}\label{dsimx}
%\frac{1}{4}|d|\leq \max{(d_{12},d_{13},d_{23})}\leq 3|d|.
%\end{align}
Here, we assume that there exists $t^{\prime}\gg 1$ such that 
\begin{align}\label{djksharpkatei}
\max{(d_{12}(t^{\prime}),d_{13}(t^{\prime}),d_{23}(t^{\prime}))}>\frac{100c_2}{L^3(t^{\prime})}.
\end{align}
Then, we introduce the following bootstrap estimate
\begin{align}\label{dbs}
\max{(d_{12}(t),d_{13}(t),d_{23}(t))}\geq \frac{10c_2}{L^3(t)}.
\end{align}
We define $T_1\in [t^{\prime}, \infty]$ by
\begin{align}\label{T1defdbs}
T_1=\sup{\{t\in[t^{\prime},\infty) \ \mbox{such that \eqref{dbs} holds on} [t^{\prime},t]\}}.
\end{align}
Furthermore, we assume that $T_1<\infty$. Then, we assume
\begin{align*}
d_{12}(T_1)=\max{(d_{12}(T_1),d_{13}(T_1),d_{23}(T_1))}.
\end{align*}
Then, by \eqref{djktotemokitui} and \eqref{dsim}, we have at $t=T_1$
\begin{align*}
\dot{d}_{12}&\geq \left(\frac{12}{5}d_{12}+\frac{9}{5}d_{13}+\frac{9}{5}d_{23}\right)\frac{\mathcal{F}(L)}{L}- c_2\left(\frac{\mathcal{F}(L)}{L^{\frac{5}{4}}}|d|+\frac{\mathcal{F}(L)}{L^4}\right)
\\
&\geq \frac{3}{5}\frac{\mathcal{F}(L)}{L}d_{12}- c_2\left(\frac{\mathcal{F}(L)}{L^{\frac{5}{4}}}|d|+\frac{\mathcal{F}(L)}{L^4}\right)
\\
&\geq \frac{1}{2}\frac{\mathcal{F}(L)}{L}d_{12}-c_2\frac{\mathcal{F}(L)}{L^4}
\\
&>0.
\end{align*}
Since $L^{-3}$ is decreasing, \eqref{dbs} holds for all $t$ with $0<t-T_1\ll 1$, which contradicts the definition of $T_1$. Hence $T_1=\infty$. In this case, $\max{(d_{12},d_{13},d_{23})}$ is bounded and monotone increasing, so it converges to a nonzero limit; this contradicts \eqref{djkrough}. Therefore, we have for $t\gg 1$
\begin{align}\label{djktibi}
\max{(d_{12},d_{13},d_{23})}\lesssim \frac{1}{L^3}.
\end{align}
By \eqref{aknaisekies}, \eqref{dsim} and \eqref{djktibi}, we obtain \eqref{kituiakdjkes}. This completes the proof.

\end{proof}

\subsection{Long-time dynamics of \((1,3)\)-sign solitary waves}
In the previous subsection, we refine the estimates for $\rho_1,\rho_2,\rho_3$. Using these improved bounds, we derive a more precise system of ODEs for $z_0,z_1,z_2,z_3$.
\begin{lemma}\label{Ldehyouka}
Let $\vec{u}$ be a (1,3)-sign solitary waves. Then, $z_0,z_1,z_2,z_3$ satisfy 
\begin{align}\label{(1,3)zdynamicsver3}
\begin{aligned}
\dot{z}_0&=-\frac{\mathcal{F}(L)}{L}(Z_1+Z_2+Z_3)+O\left(\frac{1}{t(\log{t})^{3}}\right),
\\
\dot{z}_1&=\frac{\mathcal{F}(L)}{L}Z_1+O\left(\frac{1}{t(\log{t})^{3}}\right),
\\
\dot{z}_2&=\frac{\mathcal{F}(L)}{L}Z_2+O\left(\frac{1}{t(\log{t})^{3}}\right),
\\
\dot{z}_3&=\frac{\mathcal{F}(L)}{L}Z_3+O\left(\frac{1}{t(\log{t})^{3}}\right).
\end{aligned}
\end{align}
\end{lemma}
\begin{proof}
Applying \eqref{tukaeruF2} and Lemma \ref{aknaiseki} to Lemma \ref{daijisiki2}, we obtain for $k=1,2,3$
\begin{align*}
\left|\mathcal{F}(\rho_k)-\mathcal{F}(L)\right|\lesssim \frac{1}{t(\log{t})^3},\ \left|\frac{1}{\rho_k}-\frac{1}{L}\right|\lesssim \frac{1}{(\log{t})^5}.
\end{align*}
Substituting these bounds into the equations in Lemma \ref{daijisiki2} yields \eqref{(1,3)zdynamicsver3}.
\end{proof}

Here, we analyze the long-time behavior of $z_0,z_1,z_2,z_3$.
We first estimate the barycenter of $z_0,z_1,z_2,z_3$ using Lemma \ref{Ldehyouka}.
\begin{lemma}\label{jusines}
Let $\vec{u}$ be a (1,3)-sign solitary waves. Then, there exists $z_g\in \mathbb{R}^d$ such that we have for $t\gg 1$
\begin{align}\label{jusines1}
\left|\frac{z_0(t)+z_1(t)+z_2(t)+z_3(t)}{4}-z_g\right|\lesssim \frac{1}{(\log{t})^2}.
\end{align}
\end{lemma}
\begin{proof}
By Lemma \ref{Ldehyouka}, we have for $t\gg 1$
\begin{align*}
\left|\frac{\dot{z}_0+\dot{z}_1+\dot{z}_2+\dot{z}_3}{4}\right|\lesssim \frac{1}{t(\log{t})^3}.
\end{align*}
Since the right-hand side is integrable, there exists $z_g\in\mathbb{R}^d$ such that
$\frac{z_0+z_1+z_2+z_3}{4}$ converges to $z_g$ as $t\to\infty$. In particular, we have for $t\gg 1$
\begin{align*}
\left|\frac{z_0(t)+z_1(t)+z_2(t)+z_3(t)}{4}-z_g\right|&\lesssim \int_t^{\infty} \frac{ds}{s(\log{s})^3}
\\
&\lesssim \frac{1}{(\log{t})^2}.
\end{align*}
Therefore, we obtain \eqref{jusines1}.
\end{proof}
Next, we prove the convergence of $z_0$. The idea, as in \cite[Lemma 4.8]{I2}, is to argue that if $Z_1+Z_2+Z_3$ did not converge, then its magnitude would grow and eventually become much larger than $\rho_{12},\rho_{13},\rho_{23}$, leading to a contradiction.
\begin{lemma}\label{z0ugokanai}
Let $\vec{u}$ be a (1,3)-sign solitary waves. Then, there exists $z_{\infty}\in \mathbb{R}^d$ such that we have for $t\gg 1$
\begin{align}\label{z0es128}
|z_0(t)-z_{\infty}|\lesssim \frac{1}{(\log{t})^{2}}.
\end{align}
\end{lemma}
\begin{proof}
First, rewriting Lemma \ref{Ldehyouka} in terms of $Z_1,Z_2,Z_3$, we obtain
\begin{align}\label{Zkei}
\begin{aligned}
\dot{Z}_1&=\frac{\mathcal{F}(L)}{L}(2Z_1+Z_2+Z_3)+O\left(\frac{1}{t(\log{t})^{3}}\right),
\\
\dot{Z}_2&=\frac{\mathcal{F}(L)}{L}(Z_1+2Z_2+Z_3)+O\left(\frac{1}{t(\log{t})^{3}}\right),
\\
\dot{Z}_3&=\frac{\mathcal{F}(L)}{L}(Z_1+Z_2+2Z_3)+O\left(\frac{1}{t(\log{t})^{3}}\right).
\end{aligned}
\end{align}
Moreover, we have for $t\gg 1$
\begin{align*}
\dot{\rho}_{12}&=\frac{\mathcal{F}(L)}{L}\rho_{12}+O\left(\frac{1}{t(\log{t})^{3}}\right),
\\
\dot{\rho}_{13}&=\frac{\mathcal{F}(L)}{L}\rho_{13}+O\left(\frac{1}{t(\log{t})^{3}}\right),
\\
\dot{\rho}_{23}&=\frac{\mathcal{F}(L)}{L}\rho_{23}+O\left(\frac{1}{t(\log{t})^{3}}\right).
\end{align*}
Here, we define $\mathcal{R}$ as 
\begin{align}\label{Rdef}
\mathcal{R}=\rho_{12}+\rho_{13}+\rho_{23}.
\end{align}
Then, we have for $t\gg 1$
\begin{align}\label{Rkei}
\dot{\mathcal{R}}=\frac{\mathcal{F}(L)}{L}\mathcal{R}+O\left(\frac{1}{t(\log{t})^{3}}\right).
\end{align}
Furthermore, we define $W=Z_1+Z_2+Z_3$ and $\mathcal{W}=|W|$. Then, $W$ satisfies
\begin{align*}
\dot{W}=\frac{4\mathcal{F}(L)}{L}W+O\left(\frac{1}{t(\log{t})^{3}}\right).
\end{align*}
Therefore, there exists $C_1>1$ such that we have for $t\gg 1$
\begin{align}\label{Wkei}
\left|\dot{\mathcal{W}}-\frac{4\mathcal{F}(L)}{L}\mathcal{W}\right|\leq \frac{C_1\mathcal{F}(L)}{L^3}.
\end{align}
Here, we assume that there exists $t^{\prime}>0$ such that 
\begin{align}\label{dekaikatei}
\mathcal{W}(t^{\prime})>\frac{C_1}{L^2(t^{\prime})}.
\end{align}
Then, we introduce the following bootstrap estimate
\begin{align}\label{Wbs}
\mathcal{W}(t)>\frac{C_1}{L^2(t)}.
\end{align}
We define $T_1\in [t^{\prime}, \infty]$ by
\begin{align}\label{bs1}
T_1=\sup{\{t\in[t^{\prime},\infty) \ \mbox{such that \eqref{Wbs} holds on}\ [t^{\prime},t]\}}.
\end{align}
Furthermore, we assume that $T_1<\infty$. Then, we have $\mathcal{W}(T_1)=\frac{C_1}{L^2(T_1)}$. Furthermore, by \eqref{Wkei}, we have
\begin{align*}
\dot{\mathcal{W}}(T_1)\geq \frac{4\mathcal{F}(L(T_1))}{L(T_1)}\mathcal{W}(T_1)-\frac{C_1\mathcal{F}(L(T_1))}{L^3(T_1)}>0.
\end{align*}
Moreover, since $\frac{C_1}{L^2}$ is monotone decreasing, the estimate \eqref{Wbs} also holds for $0<t-T_1\ll 1$, which contradicts the maximality of $T_1$.
Hence $T_1=\infty$. Then, since $\mathcal{W}$ is monotone increasing for $t>t'$, if $\mathcal{W}(t)$ did not tend to $+\infty$ as $t\to\infty$, then $\mathcal{W}$ would converge.
In that case, for $1\ll T_2$ we would have
\begin{align*}
1\gtrsim \mathcal{W}(e^{T_2})-\mathcal{W}(T_2)\gtrsim \int_{T_2}^{e^{T_2}}\frac{dt}{t\log t},
\end{align*}
which is impossible. Therefore, $\lim_{t\to\infty}\mathcal{W}(t)=\infty$. Here we define $\mathcal{X}$ as 
\begin{align}\label{Xdef}
\mathcal{X}=\frac{\mathcal{R}}{\mathcal{W}}.
\end{align}
Then, we have
\begin{align*}
\dot{\mathcal{X}}&=\frac{1}{\mathcal{W}^2}\left(\dot{\mathcal{R}}\mathcal{W}-\dot{\mathcal{W}}\mathcal{R} \right)
\\
&=-\frac{3\mathcal{F}(L)}{L}\frac{\mathcal{R}}{\mathcal{W}}+O\left(\frac{\mathcal{W}+\mathcal{R}}{\mathcal{W}^2t(\log{t})^3}\right)
\\
&=-\frac{3\mathcal{F}(L)}{L}\mathcal{X}+O\left(\frac{1+\mathcal{X}}{t(\log{t})^3}\right),
\end{align*}
which implies 
\begin{align}\label{Xinfty}
\lim_{t\to\infty}\mathcal{X}(t)=0.
\end{align}
On the other hand, by Proposition \ref{tildeDnagai}, we have for $t\gg 1$
\begin{align*}
\mathcal{W}&=|z_1+z_2+z_3-3z_0|
\\
&=\left|(z_1-z_2)+2(z_2-z_3)+3(z_3-z_1)+3(z_1-z_0)\right|
\\
&\leq 3\mathcal{R}+3\rho_1
\\
&\leq 7\mathcal{R},
\end{align*}
which contradicts \eqref{Xinfty}. Thus, we obtain
\begin{align*}
\mathcal{W}=|z_1+z_2+z_3-3z_0|\lesssim \frac{1}{L^2}\sim \frac{1}{(\log{t})^2}.
\end{align*}
Then, defining $z_{\infty}=z_g$, by \eqref{jusines1} we obtain \eqref{z0es128}.

\end{proof}

\subsection{Proof of Theorem \ref{maintheorem}}

We now prove Theorem \ref{maintheorem}. First, we establish the following lemma.
\begin{lemma}\label{modsitenai}
Let $\vec{u}$ be a (1,3)-sign solitary waves. Then, there exist $z_{\infty}\in \mathbb{R}^d$ and equilateral triple $\omega_1,\omega_2,\omega_3\in S^{d-1}$ such that we have for $k=1,2,3$ and $t\gg 1$
\begin{align}
\left|z_k(t)-z_{\infty}-\omega_k\left(\log{t}-\frac{d-1}{2}\log{(\log{t})}+c_{\star}\right)\right|\lesssim \frac{\log{(\log{t})}}{\log{t}}.
\end{align}
\end{lemma}
\begin{proof}
By Lemma \ref{daijisiki2} and Lemma \ref{z0ugokanai}, we have for $k=1,2,3$
\begin{align*}
\dot{Z}_k&=\mathcal{F}(\rho_k)u_k+O\left(\frac{1}{t(\log{t})^3}\right),
\\
\dot{\rho}_k&=\mathcal{F}(\rho_k)+O\left(\frac{1}{t(\log{t})^3}\right).
\end{align*}
Moreover, we have
\begin{align*}
\dot{u}_k&=\frac{\dot{Z}_k-\dot{\rho}_ku_k}{\rho_k}
\\
&=O\left(\frac{1}{t(\log{t})^4}\right).
\end{align*}
Therefore, there exists $\omega_k\in S^{d-1}$ such that $\lim_{t\to\infty}u_k(t)=\omega_k$. Then, we obtain
\begin{align*}
\left|u_k(t)-\omega_k\right|&\lesssim \int_t^{\infty} \frac{ds}{s(\log{s})^4}
\\
&\lesssim \frac{1}{(\log{t})^3}.
\end{align*}
By Lemma \ref{aknaiseki} and Lemma \ref{z0ugokanai} and \eqref{Lesti}, we obtain for $k=1,2,3$
\begin{align}\label{saigonitukau}
\begin{aligned}
z_k(t)&=z_0(t)+\rho_k(t)u_k(t)
\\
&=z_{\infty}+\rho_k(t)\omega_k+O\left(\frac{1}{(\log{t})^2}\right)
\\
&=z_{\infty}+\omega_k\left(\log{t}-\frac{d-1}{2}\log{(\log{t})}+c_{\star}\right)+O\left(\frac{\log{(\log{t})}}{\log{t}}\right).
\end{aligned}
\end{align}
In particular, by Lemma \ref{jusines}, $\frac{z_0+z_1+z_2+z_3}{4}$ is bounded. On the other hand, by Lemma \ref{z0ugokanai} and \eqref{saigonitukau}, we have for $t\gg 1$
\begin{align*}
\frac{z_0(t)+z_1(t)+z_2(t)+z_3(t)}{4}&=z_{\infty}+\frac{\omega_1+\omega_2+\omega_3}{4}\left(\log{t}-\frac{d-1}{2}\log{(\log{t})}+c_{\star}\right)
\\
&\quad+O\left(\frac{\log{(\log{t})}}{\log{t}}\right).
\end{align*}
Therefore, $\omega_1,\omega_2,\omega_3$ satisfy
\begin{align*}
\omega_1+\omega_2+\omega_3=0,
\end{align*}
which implies $\omega_1,\omega_2,\omega_3$ form an equilateral triple. Therefore we complete the proof.

\end{proof}

Finally, we prove Theorem \ref{maintheorem}.

\begin{proof}[Proof of Theorem \ref{maintheorem}]
By Lemma \ref{modKsoli}, there exist $C^1$ functions $z_0,z_1,z_2,z_3:[0,\infty)\to \mathbb{R}^d$ such that
\begin{align*}
\lim_{t\to\infty}\left\| u(t)-\left(Q(\cdot-z_0(t))-\sum_{k=1}^3 Q(\cdot-z_k(t))\right)\right\|_{H^1}+\|{\partial}_tu(t)\|_{L^2}=0, 
\\
\lim_{t\to\infty}\left(\min_{0\leq j<k\leq 3}|z_k(t)-z_j(t)|\right)=\infty,
\end{align*}
and we have for $1\leq l\leq d$, $0\leq i\leq 3$, and $t\gg 1$,
\begin{align*}
\int_{\mathbb{R}^d} \Bigl\{ {\partial}_tu(t)+2\alpha\Bigl(u(t)-Q(\cdot-z_0(t))+\sum_{k=1}^3 Q(\cdot-z_k(t))\Bigr)\Bigr\}
\, {\partial}_lQ(\cdot-z_i(t))=0.
\end{align*}
In this case, we apply the preceding analysis to this choice of $z_0,z_1,z_2,z_3$.
First, by Proposition \ref{tildeDnagai}, \eqref{Vrep} holds true. Therefore, by Lemma \ref{epzes} and Lemma \ref{Dlogtorder}, we have for $t\gg 1$
\begin{align*}
\left\| u(t)-\left(Q(\cdot-z_0(t))-\sum_{k=1}^3 Q(\cdot-z_k(t))\right)\right\|_{H^1}+\|{\partial}_tu(t)\|_{L^2}\lesssim t^{-1},
\end{align*}
which is \eqref{3soli}. Choosing $c_0=c_{\star}$, the asymptotic behaviors of $z_0,z_1,z_2,z_3$ satisfy \eqref{zestheorem} by Lemma \ref{modsitenai}. This completes the proof of  Theorem \ref{maintheorem}.
\end{proof}

\end{document}